\documentclass[10pt]{amsart}

\usepackage{amssymb,verbatim,color}
\usepackage{esint}
\usepackage[colorlinks=true,urlcolor=blue, citecolor=red,linkcolor=blue,linktocpage,pdfpagelabels, bookmarksnumbered,bookmarksopen]{hyperref}
\usepackage[hyperpageref]{backref}

\usepackage{verbatim}
\newtheorem{lm}{Lemma}[section]
\newtheorem{prop}[lm]{Proposition}
\newtheorem{coro}[lm]{Corollary}
\newtheorem{teo}[lm]{Theorem}
\theoremstyle{definition}
\newtheorem{oss}[lm]{Remark}
\newtheorem{defi}[lm]{Definition}

\newtheorem*{ack}{Acknowledgements}

\numberwithin{equation}{section}

\author[Brasco]{Lorenzo Brasco}
\address[L.\ Brasco]{Dipartimento di Matematica e Informatica
\newline\indent
Universit\`a degli Studi di Ferrara
\newline\indent
Via Machiavelli 35, 44121 Ferrara, Italy}
\email{lorenzo.brasco@unife.it}

\author[Lindgren]{Erik Lindgren}
\address[E. Lindgren]{Department of Mathematics, Uppsala University
\newline\indent
Box 480\\
751 06 Uppsala, Sweden}
\email{erik.lindgren@math.uu.se}

\author[Str\"omqvist]{Martin Str\"omqvist}
\address[M. Str\"omqvist]{\,}
\email{stromqv@gmail.com}
\title[Nonlinear fractional diffusion]{Continuity of solutions to a nonlinear\\ fractional diffusion equation}


\begin{document}
\begin{abstract} 
We study a parabolic equation for the fractional $p-$Laplacian of order $s$, for $p\ge 2$ and $0<s<1$.
We provide space-time H\"older estimates for weak solutions, with explicit exponents. The proofs are based on iterated discrete differentiation of the equation in the spirit of J. Moser.
\end{abstract}
\maketitle

\begin{center}
\begin{minipage}{8cm}
\small
\tableofcontents
\end{minipage}
\end{center}

\section{Introduction}
\subsection{The problem}
In this paper, we study the regularity of {\it weak solutions} to the nonlinear and nonlocal parabolic equation
\begin{equation}\label{MainPDE}
\partial_t u+(-\Delta_p)^s u=0,
\end{equation}
where $2\le p<\infty$, $0<s<1$ and $(-\Delta_p)^s$ is the {\it fractional $p$-Laplacian of order $s$}, i.e. the operator formally defined by
\begin{equation}\label{fracplap}
(-\Delta_p)^s u\, (x):=2\, \mathrm{P.V.} \int_{\mathbb{R}^N}\frac{|u(x)-u(x+h)|^{p-2}(u(x)-u(x+h))}{|h|^{N+s\,p}}\, dh.
\end{equation}
Here $\mathrm{P.V.}$ denotes the principal value in Cauchy sense. The operator $(-\Delta_p)^s$ arises as the first variation of the Sobolev-Slobodecki\u{\i} seminorm (see Section \ref{sec:not})
\[
u\mapsto \iint_{\mathbb{R}^N\times \mathbb{R}^N} \frac{|u(x)-u(x)|^p}{|x-y|^{N+s\,p}}\,dx\,dy.
\]
This operator can be seen as a nonlocal (or fractional) version of the $p-$Laplace operator, 
\[
-\Delta_pu=-\mathrm{div\,}(|\nabla u|^{p-2}\nabla u),
\]
since, as $s$ goes to $1$, solutions of $(-\Delta_p)^s u =0$ converge to solutions of $-\Delta_p u =0$, once suitably rescaled. See for instance \cite[Section 1.4]{Brolin} and \cite{IN}.
\begin{oss}[Homogeneity and scalings]
\label{oss:scalings}
It is important to notice that equation \eqref{MainPDE} {\it is not homogeneous}, i.e. if $u$ is a solution, then $\lambda\,u$ does not solve the same equation. Rather, it solves
\[
\partial_t u+\lambda^{2-p}\,(-\Delta_p)^s u=0.
\] 
On the other hand, solutions are invariant with respect to the natural scaling $(x,t)\mapsto (\lambda\,x,\lambda^{s\,p}\,t)$, for any $\lambda>0$. In other words, if $u$ is a solution of \eqref{MainPDE}, then the rescaled function
\[
u_\lambda(x,t)=u\left(\lambda\,x,\lambda^{s\,p}\,t\right),
\]
is still a solution. By combining the last two facts, we also get that 
\[
u_{\lambda,\mu}=\mu\,u\left(\lambda\,x,\mu^{p-2}\,\lambda^{s\,p}\,t\right),\qquad \mbox{ for } \lambda,\mu>0,
\]
still solves \eqref{MainPDE}. We will make a repeated use of this simple fact.
\end{oss}
In this paper, we are concerned with the H\"older regularity for weak solutions of \eqref{MainPDE}. More precisely, we prove that local weak solutions (see Definition \ref{locweak} below) are locally $\delta-$H\"older continuous in space and $\gamma-$H\"older continuous in time, whenever 
\[
0<\delta <\Theta(s,p):=\left\{\begin{array}{rl}
\dfrac{s\,p}{p-1},& \mbox{ if } s<\dfrac{p-1}{p},\\
&\\
1,& \mbox{ if } s\ge \dfrac{p-1}{p},
\end{array}
\right.\quad \mbox{ and }\quad 0<\gamma< \Gamma(s,p):=\left\{\begin{array}{rl}
1,& \mbox{ if } s<\dfrac{p-1}{p},\\
&\\
\dfrac{1}{s\,p-(p-2)},& \mbox{ if } s\ge \dfrac{p-1}{p}.
\end{array}
\right.
\] 
To the best of our knowledge, our result is the first pointwise continuity estimate for solutions of this equation. 

\subsection{Background and recent developments}
In recent years there has been a surge of interest around the operator \eqref{fracplap}, after its introduction in \cite{IN}. In particular, equation \eqref{MainPDE} has been studied in \cite{abd, MRT, puh, str, vasquez} and \cite{warma}. References \cite{puh}, \cite{MRT} and \cite{vasquez} dealt with existence and uniqueness of solutions, together with their long time asymptotic behaviour. Similar properties for \eqref{MainPDE} with a general right-hand side in place of $0$ are studied in \cite{abd}. In \cite{warma}, some regularity of the semigroup operator generated by $(-\Delta_p)^s$ was studied. In \cite{str}, the local boundedness of weak solutions of \eqref{MainPDE} is proved.
\par
Recently, in \cite{HL}, a weaker pointwise regularity result was obtained for {\it viscosity solutions} of the doubly nonlinear equation
\begin{equation}
\label{stromberg}
|\partial_t u|^{p-2}\,\partial_t u+(-\Delta_p)^s u=0, 
\end{equation}
by using completely different methods. This equation and its large time behavior is related to the eigenvalue problem for the fractional $p-$Laplacian. A crucial difference between this equation and \eqref{MainPDE}, is that the former is homogeneous, a feature which is not shared by our equation, as already observed in Remark \ref{oss:scalings}.
Moreover, the nonlinearity in the time derivative in \eqref{stromberg} makes the notion of weak solutions less useful. It is not clear whether the methods in \cite{HL} can be adapted to the present situation or not.
\par
In the linear or non-degenerate case, corresponding to $p=2$, the literature on regularity is vast. We mention only a fraction of it, namely  \cite{CV, LD, LD2, Sil10, Sil12} and \cite{str2}. However, we point out that neither of these results apply to our setting.
\par
The stationary version of \eqref{MainPDE}, i.e., 
$$
(-\Delta_p)^s u=0,
$$
has attracted a lot of attention, as well. The regularity of solutions has been studied for instance in \cite{Brolin, bralinschi, Co, DKP, DKP2, IMS, IMS2, KKP, KMS, KMS2, Li, Sc} and \cite{warma}. In particular, the regularity result proved in the present paper can be seen as the parabolic version of that obtained by the first two authors and Schikorra in \cite{bralinschi} for the stationary equation. 
\par
The local counterpart of \eqref{MainPDE} is the parabolic equation for the $p-$Laplacian  
\[
\partial_t u-\Delta_p u=0.
\]
This has been intensively studied and only in the last decades has its theory reached a rather complete state. We refer to \cite{DiBbook} and \cite{DGV} for a complete account on the regularity results for this equation and some of its generalizations.  
At present, the best local regularity known is spatial $C^{1,\alpha}-$regularity for some $\alpha>0$ (see \cite[Chapter IX]{DiBbook}) and $C^{0,1/2}-$regularity in time (see \cite[Theorem 2.3]{Bo}). None of these exponents is known to be sharp. However, due to the explicit solution
\[
u(x,t)=N\,t-\frac{p-1}{p}|x|^\frac{p}{p-1}, 
\]
it is clear that solutions cannot be better than $C^{1,1/(p-1)}$ in space. 

\subsection{Main result}
The main result of our paper is the following H\"older regularity for local weak solutions of \eqref{MainPDE}. 
Here, we use the following notation for parabolic cylinders
\[
Q_{R,r}(x_0,t_0)=B_R(x_0)\times(t_0-r,t_0],
\]
with $B_r(x_0)$ denoting the $N-$dimensional ball of radius $r$ centered at the point $x_0$. For the precise definition of {\it local weak solution}, as well as of the spaces $C^\delta_{x,\rm loc}(\Omega\times I)$ and 
$C^\gamma_{t,\rm loc}(\Omega\times I)$, we refer the reader to Sections \ref{sec:local} and \ref{pBspaces}, respectively
\begin{teo}\label{thm1}
Let $\Omega\subset\mathbb{R}^N$ be a bounded and open set, $I=(t_0,t_1]$, $p\geq 2$ and $0<s<1$. Suppose $u$
is a local weak solution of
\[
u_t+(-\Delta_p)^s u=0\qquad \mbox{ in }\Omega\times I,
\]
such that 
\begin{equation}
\label{finitesup}
u\in L^{\infty}_{\rm loc}(I;L^\infty(\mathbb{R}^N)).
\end{equation}
Define the exponents
\begin{equation}
\label{exponents}
\Theta(s,p):=\left\{\begin{array}{rl}
\dfrac{s\,p}{p-1},& \mbox{ if } s<\dfrac{p-1}{p},\\
&\\
1,& \mbox{ if } s\ge \dfrac{p-1}{p},
\end{array}
\right.\quad \mbox{ and }\quad \Gamma(s,p):=\left\{\begin{array}{rl}
1,& \mbox{ if } s<\dfrac{p-1}{p},\\
&\\
\dfrac{1}{s\,p-(p-2)},& \mbox{ if } s\ge \dfrac{p-1}{p}.
\end{array}
\right.
\end{equation}
Then 
\[
u\in C^\delta_{x,\rm loc}(\Omega\times I)\cap C^\gamma_{t,\rm loc}(\Omega\times I),\qquad \mbox{ for every }0<\delta<\Theta(s,p) \ \mbox{ and } \ 0<\gamma<\Gamma(s,p).
\] 
More precisely, for every $0<\delta<\Theta(s,p)$, $0<\gamma<\Gamma(s,p)$, $R>0$, $x_0\in\Omega$ and $T_0$ such that 
\[
Q_{2R,2R^{s\,p}}(x_0,T_0)\Subset\Omega\times (t_0,t_1],
\] 
there exists a constant $C=C(N,s,p,\delta, \gamma)>0$ such that
\begin{equation}
\label{apriori_spacetime}
\begin{split}
|u(x_1,\tau_1)-u(x_2,\tau_2)|&\leq C\,(\|u\|_{L^\infty(Q_{\infty,R^{s\,p}}(x_0,T_0))}+1)\, \left(\frac{|x_1-x_2|}{R}\right)^\delta\\
&+C\,(\|u\|_{L^\infty(Q_{\infty,R^{s\,p}}(x_0,T_0))}+1)^{\gamma\,(p-2)+1}\,\left(\frac{|\tau_1-\tau_2|}{R^{s\,p}}\right)^\gamma.
\end{split}
\end{equation}
for any $(x_1,\tau_1),\,(x_2,\tau_2)\in Q_{R/4,R^{s\,p}/4}(x_0,T_0)$. 
\end{teo}
\begin{oss}[Comment on the time regularity] The regularity in time is almost sharp for $s\,p\leq (p-1)$. Indeed, our result in this case gives H\"older continuity for any exponent less than 1. The following example from \cite{LD2} shows that solutions are not $C^1$ in time in general. Let  
$$
v(x,t)=\left\{\begin{array}{rl} 0,& \mbox{ if }t<-1/2,\\
C\,(1/2+t)+1_{B_3\setminus B_2}(x) ,& \mbox{ if }t\geq -1/2,\end{array}\right.
$$
where $C\neq 0$ is chosen so that $v$ is a local weak subsolution (see Definition \ref{locweak}) in $B_1\times (-1,0]$. Then, if $u$ is the unique solution (given by Theorem \ref{thm:exunsol}) of 
\[
 \left\{\begin{array}{rcll}
\partial_tu + (-\Delta_p)^su&=&0,&\mbox{ in } B_1\times (-1,0],\\
u&=&v,& \mbox{ on }(\mathbb{R}^N\setminus B_1)\times (-1,0],\\
u(\cdot,0) &=& 0,&\mbox{ on }\Omega, 
\end{array}\right.
\]
by Proposition \ref{prop:subcomp} we get $u\geq v$ in $B_1\times (-1,0]$. Moreover, by Proposition \ref{prop:linftyglobal}, $u=0$ in $B_1\times (-1,1/2)$. Therefore, 
$$
u(x,-1/2+h)-u(x,-1/2-h)\geq C\,h, 
$$
for $h>0$ and $x\in B_1$. Hence, $u$ cannot have a continuous time derivative.
\end{oss}
\begin{oss}[Comments on the assumption]
We have chosen to assume the global boundedness \eqref{finitesup} of our weak solutions, in order to simplify the presentation.
Actually, the estimate \eqref{apriori_spacetime} could be proved under the weaker assumption
\begin{equation}\label{localbound}
u\in L^\infty_{\rm loc}(I;L^\infty_{\rm loc}(\Omega)),
\end{equation}
and	
\begin{equation}\label{thesuptail}
u\in L^\infty_{\rm loc}(I;L^{p-1}_{s\,p}(\mathbb{R}^N)),
\end{equation}
where the {\it tail space} $L^{p-1}_{s\,p}(\mathbb{R}^N)$ is defined by
\[
L^{p-1}_{s\,p}(\mathbb{R}^N)=\left\{u \in L^{p-1}_{\rm loc}(\mathbb{R}^N)\, :\, \int_{\mathbb{R}^N} \frac{|u|^{p-1}}{1+|x|^{N+s\,p}}\,dx<+\infty\right\}.
\]
We point out that by \cite[Lemma 2.6]{str}, condition \eqref{thesuptail} is a natural one in order to guarantee the local boundedness \eqref{localbound}.
However, it is not known apriori if the quantity \eqref{thesuptail} is finite whenever $u$ is a weak solution. Indeed, even if $u$ solves the initial boundary value problem
\[
 \left\{\begin{array}{rcll}
\partial_tu + (-\Delta_p)^su&=&0,&\mbox{ in }\Omega\times I,\\
u&=&g,& \mbox{ on }(\mathbb{R}^N\setminus\Omega)\times I,\\
u &=& u_0,&\mbox{ on }\Omega\times\{t=t_0\}, 
\end{array}\right.
\]
with the boundary data $g$ satisfying 
\[
g\in L^\infty_{\rm loc}(I;L^{p-1}_{s\,p}(\mathbb{R}^N)),
\]
it is not evident that this is sufficient to entail \eqref{thesuptail}.
For this reason, and to not overburden an already technical proof, we have chosen to assume the simpler condition \eqref{finitesup}. For completeness, in Appendix \ref{sec:A} we give some sufficient conditions assuring that our weak solutions verify \eqref{finitesup}, see Corollary \ref{coro:linftyglobal} below.
\end{oss}

\subsection{Main ideas of the paper} The idea we use to prove Theorem \ref{thm1} is very similar to the method employed in \cite{bralinschi} for the elliptic case: we differentiate equation \eqref{MainPDE} in a discrete sense and then test the differentiated equation against functions of the form
$$
\left|\frac{\delta_h u}{|h|^\vartheta}\right|^{\beta-1}\,\frac{\delta_h u}{|h|^\vartheta},\qquad \mbox{ where } \delta_h u(x,t):=u(x+h,t)-u(x,t).
$$
For suitable choices of $\vartheta>0$ and $\beta\ge 1$, this gives an integrability gain (see Proposition \ref{prop:improve}) of the form
\begin{equation}
\label{imp}
\int_{-1+\mu}^T\left\|\frac{\delta^2_h u(x,t)}{|h|^{s}}\right\|_{L^{q+1}(B_{1/2})}^{q+1}dt +\left\|\frac{\delta_h u(\cdot,T)}{|h|^{\frac{(q+2-p)\,s}{q+3-p}}}\right\|_{L^{q+3-p}(B_{1/2})}^{q+3-p}\lesssim \int_{-1}^T\left\|\frac{\delta^2_h u(x,t)}{|h|^s}\right\|_{L^{q}(B_1)}^q dt,
\end{equation}
for $-1/2\le T\le 0$ and an arbitrary $\mu>0$.
By first fixing $T=0$ and ignoring the second term in the left-hand side of \eqref{imp}, this can be iterated finitely many times in order to obtain
\[
\frac{\delta_h u}{|h|^s}\in L^q([-1/2,0];L^q_{\rm loc}),\qquad \mbox{ for every } q<\infty, \mbox{ uniformly in } |h|\ll 1.
\]
We can then use the second term in the left-hand side of \eqref{imp}, so to get
$$
\frac{\delta_h u(\cdot,T)}{|h|^s}\in L^q_{\rm loc},\qquad \mbox{ for every } q<\infty, \mbox{ uniformly in } |h|\ll 1 \mbox{ and } -\frac{1}{2}\le T\le 0.
$$
Thus, by using a Morrey-type embedding result, we can conclude that $u\in C_\text{loc}^\delta$ spatially for any $0<\delta<s$. 
\par
After this, we prove Proposition \ref{prop:improve2}, which comprises a refined version of the scheme \eqref{imp}. Namely, an estimate of the form
\begin{equation}
\label{imperik}
\int_{-1+\mu}^T\left\|\frac{\delta^2_h u(x,t)}{|h|^{\frac{1+s\,p+\vartheta\,\beta}{\beta-1+p}}}\right\|_{L^{\beta-1+p}(B_{1/2})}^{\beta-1+p}dt +\left\|\frac{\delta_h u(\cdot,T)}{|h|^{\frac{1+\vartheta\beta}{\beta+1}}}\right\|_{L^{\beta+1}(B_{1/2})}^{\beta+1}\lesssim \int_{-1}^T\left\|\frac{\delta^2_h u(x,t)}{|h|^\frac{1+\vartheta\, \beta}{\beta}}\right\|_{L^{\beta}(B_1)}^\beta dt.
\end{equation}
Also \eqref{imperik} can be iterated, where now both the differentiability $\vartheta$ and the integrability $\beta$ change. The result is that
\[
u\in C_{\rm loc}^{\delta} \mbox{ spatially},\qquad \mbox{ for every } 0<\delta <\Theta(s,p),
\] 
again uniformly in time.
The last part of the paper, where we obtain the regularity in time, is quite standard for this kind of diffusion equations (see for example \cite[page 118]{CD}). It amounts to using the already established spatial regularity and the information given by the equation. However, due to the fractional character of the spatial part of our equation, some care is needed in order to properly handle the time regularity. In particular, we have to treat the cases
\[
s< \frac{p-1}{p}\qquad \mbox{ and }\qquad s\ge \frac{p-1}{p}, 
\]
separately. This is done in Proposition \ref{prop_time} and it yields the $\gamma-$H\"older continuity in time for any 
\[
\gamma = \frac{1}{\dfrac{s\,p}{\delta}-\,(p-2)},
\]
given that the solution is $\delta-$H\"older continuous in the $x$ variable. In particular, by the possible choice of $\delta$, this yields
that we may choose any $\gamma<\Gamma(s,p)$, where the latter exponent is the one defined in \eqref{exponents}.

\subsection{Plan of the paper}
The plan of the paper is as follows. In Section \ref{sec:prel}, we introduce the expedient spaces and notation used in this paper. In Section \ref{sec:weak}, we define local weak solutions and justify that we can insert certain test functions in the differentiated equation (see Lemma \ref{Mlemreg} below). This is followed by Section \ref{sec:almost}, where we prove that weak solutions are almost $s-$H\"older continuous in the spatial variable. In Section \ref{sec:higher}, we improve this result up to the exponent $\Theta(s,p)$ defined in \eqref{exponents}. This result is then used in Section \ref{sec:time}, where we prove the corresponding H\"older regularity in time. Finally, in Section \ref{sec:main} we prove our main theorem.
\par
The paper is complemented by an appendix, where for completeness we prove existence and uniqueness of weak solutions for the initial boundary value problem related to our equation. A comparison principle is also presented.

\begin{ack}
We thank Eleonora Cinti for drawing our attention on the papers \cite{LD, LD2}.
E.\,L.  is supported by the Swedish Research Council, grant no. 2012-3124 and 2017-03736. Part of this work has been done during a visit of L.\,B. to Uppsala and a visit of E.\,L. to Bologna and Ferrara. The paper has been finalized during the conference ``{\it Nonlinear averaging and PDEs\,}'', held in Levico Terme in June 2019.
The hosting institutions and the organizers are kindly acknowledged.
\end{ack}

\section{Preliminaries}\label{sec:prel}

\subsection{Notation}\label{sec:not} We denote by $B_r(x_0)$ the $N-$dimensional open ball of radius $r$ centered at the point $x_0$. The ball of radius $r$ centered at the origin is denoted by $B_r$. Its Lebesgue measure is given by
\[
|B_r(x_0)|=\omega_N\,r^N.
\]
We use the following notation for the parabolic cylinder
\[
Q_{R,r}(x_0,t_0)=B_R(x_0)\times(t_0-r,r].
\]
Again, when $x_0=0$ and $t_0=0$, we simply write $Q_{R,r}$.
\par
Let $1<p<\infty$, we denote by $p'=p/(p-1)$ the dual exponent of $p$. For every $\beta> 1$, we define the monotone function $J_\beta:\mathbb{R}\to\mathbb{R}$ by
\[
J_\beta(t)=|t|^{\beta-2}\, t.
\]
For a function $\psi:\mathbb{R}^N\times \mathbb{R} \to\mathbb{R}$ and a vector $h\in\mathbb{R}^N$, we define
\[
\psi_h(x,t)=\psi(x+h,t),\quad \delta_h \psi(x,t)=\psi_h(x,t)-\psi(x,t),\quad \delta^2_h \psi(x,t)=\delta_h(\delta_h \psi(x,t))=\psi_{2\,h}(x,t)+\psi(x,t)-2\,\psi_h(x,t).
\]
It is not difficult to see that the following {\it discrete Leibniz rule} holds
\[
\delta_h(\varphi\,\psi)=\psi_h\,\delta_h \varphi+\varphi\,\delta_h \psi.
\]
\subsection{Sobolev spaces}
We now recall the main notations and definitions for the relevant fractional Sobolev--type spaces throughout the paper.
\par
Let $1\le q<\infty$ and let $\psi\in L^q(\mathbb{R}^N)$, for $0<\beta\le 1$ we set
\[
[\psi]_{\mathcal{N}^{\beta,q}_\infty(\mathbb{R}^N)}:=\sup_{|h|>0} \left\|\frac{\delta_h \psi}{|h|^{\beta}}\right\|_{L^q(\mathbb{R}^N)},
\]
and for $0<\beta<2$
\[
[\psi]_{\mathcal{B}^{\beta,q}_\infty(\mathbb{R}^N)}:=\sup_{|h|>0} \left\|\frac{\delta_h^2 \psi}{|h|^{\beta}}\right\|_{L^q(\mathbb{R}^N)}.
\]
We then introduce the two Besov-type spaces
\[
\mathcal{N}^{\beta,q}_\infty(\mathbb{R}^N)=\left\{\psi\in L^q(\mathbb{R}^N)\, :\, [\psi]_{\mathcal{N}^{\beta,q}_\infty(\mathbb{R}^N)}<+\infty\right\},\qquad 0<\beta\le 1,
\]
and
\[
\mathcal{B}^{\beta,q}_\infty(\mathbb{R}^N)=\left\{\psi\in L^q(\mathbb{R}^N)\, :\, [\psi]_{\mathcal{B}^{\beta,q}_\infty(\mathbb{R}^N)}<+\infty\right\},\qquad 0<\beta<2.
\]
We also need the {\it Sobolev-Slobodecki\u{\i} space}
\[
W^{\beta,q}(\mathbb{R}^N)=\left\{\psi\in L^q(\mathbb{R}^N)\, :\, [\psi]_{W^{\beta,q}(\mathbb{R}^N)}<+\infty\right\},\qquad 0<\beta<1,
\]
where the seminorm $[\,\cdot\,]_{W^{\beta,q}(\mathbb{R}^N)}$ is defined by
\[
[\psi]_{W^{\beta,q}(\mathbb{R}^N)}=\left(\iint_{\mathbb{R}^N\times \mathbb{R}^N} \frac{|\psi(x)-\psi(y)|^q}{|x-y|^{N+\beta\,q}}\,dx\,dy\right)^\frac{1}{q}.
\]
We endow these spaces with the norms
\[
\|\psi\|_{\mathcal{N}^{\beta,q}_\infty(\mathbb{R}^N)}=\|\psi\|_{L^q(\mathbb{R}^N)}+[\psi]_{\mathcal{N}^{\beta,q}_\infty(\mathbb{R}^N)},
\]
\[
\|\psi\|_{\mathcal{B}^{\beta,q}_\infty(\mathbb{R}^N)}=\|\psi\|_{L^q(\mathbb{R}^N)}+[\psi]_{\mathcal{B}^{\beta,q}_\infty(\mathbb{R}^N)},
\]
and
\[
\|\psi\|_{W^{\beta,q}(\mathbb{R}^N)}=\|\psi\|_{L^q(\mathbb{R}^N)}+[\psi]_{W^{\beta,q}(\mathbb{R}^N)}.
\]
A few times we will also work with the space $W^{\beta,q}(\Omega)$ for a subset $\Omega\subset \mathbb{R}^N$,
\[
W^{\beta,q}(\Omega)=\left\{\psi\in L^q(\Omega)\, :\, [\psi]_{W^{\beta,q}(\Omega)}<+\infty\right\},\qquad 0<\beta<1,
\]
where we define
\[
 [\psi]_{W^{\beta,q}(\Omega)}=\left(\iint_{\Omega\times \Omega} \frac{|\psi(x)-\psi(y)|^q}{|x-y|^{N+\beta\,q}}\,dx\,dy\right)^\frac{1}{q}.
\]
The space $W_0^{\beta,q}(\Omega)$ is the subspace of $W^{\beta,q}(\mathbb{R}^N)$ consisting of functions that are identically zero in the complement of $\Omega$.

\subsection{Parabolic Banach spaces}\label{pBspaces}
Let $I\subset \mathbb{R}$ be an interval and let $V$ be a separable, reflexive Banach space, endowed with a norm $\|\cdot\|_V$. We denote by $V^*$ its topological dual space. 
Let us suppose that $v$ is a mapping such that for almost every $t\in I$, $v(t)$ belongs to $V$. If the function $t\mapsto \|v(t)\|_V$ is measurable on $I$ and $1\le p\le \infty$, then $v$ is an element of the Banach space $L^p(I;V)$ if and only if 
\[
\int_I\|v(t)\|_V^pdt<+\infty.
\]
By \cite[Theorem 1.5]{Sh}, the dual space of $L^p(I;V)$ can be characterized according to 
\[
(L^p(I;V))^* = L^{p'}(I;V^*). 
\]
We write $v\in C(I;V)$ if the mapping $t\mapsto v(t)$ is continuous with respect to the norm on $V$. 
We say that $u$ is {\it locally $\alpha-$H\"older continuous in space} (respectively, {\it locally  $\beta-$H\"older continuous in time}) on $\Omega\times I$ and write 
\[
u\in C^\alpha_{x,\text{loc}}(\Omega\times I), \qquad\left(\mbox{respectively, } u\in C^\beta_{t,\text{loc}}(\Omega\times I)\right), 
\]
if for any compact set $K\times J\subset \Omega\times I$, 
\[
\sup_{t\in J}[u(\cdot,t)]_{C^\alpha(K)} < +\infty,\qquad\left(\mbox{respectively, }\sup_{x\in K}[u(x,\cdot)]_{C^\beta(J)} < +\infty\right). 
\]
That is, if $u\in C^\alpha_{x}(K\times J)$ (respectively, $u\in C^\beta_{t}(K\times J)$). 

\subsection{Tail spaces} We recall the definition of {\it tail space}
\[
L^{q}_{\alpha}(\mathbb{R}^N)=\left\{u\in L^{q}_{\rm loc}(\mathbb{R}^N)\, :\, \int_{\mathbb{R}^N} \frac{|u|^q}{1+|x|^{N+\alpha}}\,dx<+\infty\right\},\qquad q\ge 1 \mbox{ and } \alpha>0,
\]
which is endowed with the norm 
\[
\|u\|_{L_\alpha^{q}(\mathbb{R}^N)} = \left(\int_{\mathbb{R}^N} \frac{|u|^q}{1+|x|^{N+\alpha}}\,dx\right)^{\frac{1}{q}}.
\]
For every $x_0\in\mathbb{R}^N$, $R>0$ and $u\in L^q_{\alpha}(\mathbb{R}^N)$, the following quantity
\[
\mathrm{Tail}_{q,\alpha}(u;x_0,R)=\left[R^{\alpha}\,\int_{\mathbb{R}^N\setminus B_R(x_0)} \frac{|u|^q}{|x-x_0|^{N+\alpha}}\,dx\right]^\frac{1}{q},
\]
plays an important role in regularity estimates for solutions of fractional problems.
We recall the following result, see for example \cite[Lemmas 2.1 \& 2.2]{bralinschi} for the proof.
\begin{lm}
Let $\alpha>0$ and $1\le q<m<\infty$. Then:
\begin{itemize}
\item we have the continuous inclusion
\[
L^{m}_{\alpha}(\mathbb{R}^N)\subset L^{q}_{\alpha}(\mathbb{R}^N);
\]
\item for every $0<r<R$ and $x_0\in\mathbb{R}^N$ we have
\[
R^\alpha\,\sup_{x\in B_r(x_0)}\int_{\mathbb{R}^N\setminus B_R(x_0)} \frac{|u(y)|^q}{|x-y|^{N+\alpha}}\,dy\le \left(\frac{R}{R-r}\right)^{N+\alpha}\,\mathrm{Tail}_{q,\alpha}(u;x_0,R)^q.
\]
\end{itemize}
\end{lm}

\section{Weak formulation}
\label{sec:weak}

\subsection{Local weak solutions}
\label{sec:local}
In the following, we assume that $\Omega\subset \mathbb{R}^N$ is a bounded open set in $\mathbb{R}^N$. 

\begin{defi}\label{locweak} For any $t_0,t_1\in\mathbb{R}$ with $t_0<t_1$, we define $I=(t_0,t_1]$.
Let 
\[
f\in L^{p'}(I;(W^{s,p}(\Omega))^*).
\]
We say that $u$ is a \emph{local weak solution} to the equation
\begin{equation}\label{locweaksol}
\partial_t u + (-\Delta_p)^su = f,\qquad\mbox{ in }\Omega\times I,
\end{equation}
if for any closed interval $J=[T_0,T_1]\subset I$, the function $u$ is such that
\[
u\in L^p(J;W_{{\rm loc}}^{s,p}(\Omega))\cap L^{p-1}(J;L_{s\,p}^{p-1}(\mathbb{R}^N))\cap C(J;L_{{\rm loc}}^2(\Omega)),
\]
and it satisfies 
\begin{equation}
\begin{split}
\label{locweakeq}
-\int_J\int_\Omega u(x,t)\,\partial_t\phi(x,t)\,dx\,dt&+ \int_J\iint_{\mathbb{R}^N\times\mathbb{R}^N}\frac{J_p(u(x,t)-u(y,t))\,(\phi(x,t)-\phi(y,t))}{|x-y|^{N+s\,p}}\,dx\,dy\,dt  \\   
& = \int_\Omega u(x,T_0)\,\phi(x,T_0)\,dx -\int_\Omega u(x,T_1)\,\phi(x,T_1)\,dx \\
&+ \int_J\langle f(\cdot,t),\phi(\cdot,t)\rangle\,dt,
\end{split}
\end{equation}
for any $\phi\in L^p(J;W^{s,p}(\Omega))\cap C^1(J;L^2(\Omega))$ which has spatial support compactly contained in $\Omega$. 
In equation \eqref{locweakeq}, the symbol $\langle\cdot,\cdot\rangle$  stands for the duality pairing between $W^{s,p}(\Omega)$ and its dual space $(W^{s,p}(\Omega))^*$. 
\par
We also say that $u$ is a \emph{local weak subsolution} if instead of the equality above, we have the $\leq$ sign, for any \emph{non-negative} $\phi$ as above. A \emph{local weak supersolution} is defined similarly.
\end{defi}
\begin{oss}
We observe that $L^\infty(\mathbb{R}^N)\subset L^{p-1}_{s\,p}(\mathbb{R}^N)$. This in turn implies that 
\[
L^\infty(J;L^\infty(\mathbb{R}^N))\subset L^{p-1}(J;L^{p-1}_{s\,p}(\mathbb{R}^N)).
\]
We will use this fact repeatedly.
\end{oss}

\subsection{Regularization of test functions}

Let $\zeta:\mathbb{R}\mapsto \mathbb{R}$ be a nonnegative, even smooth function with compact support in $(-1/2,1/2)$, satisfying $\int_{\mathbb{R}}\zeta(\tau)\,d\tau=1$. 
If $f\in L^1((a,b))$, we define the convolution 
\begin{equation}
\label{convolution}
f^\varepsilon(t) = \frac{1}{\varepsilon}\,\int_{t-\frac{\varepsilon}{2}}^{t+\frac{\varepsilon}{2}}\zeta\left(\frac{t-\ell}{\varepsilon}\right)\,f(\ell)\,d\ell=\frac{1}{\varepsilon}\,\int_{-\frac{\varepsilon}{2}}^{\frac{\varepsilon}{2}}\zeta\left(\frac{\sigma}{\varepsilon}\right)\,f(t-\sigma)\,d\sigma,\qquad \mbox{ for } t\in (a,b),
\end{equation}
where $0<\varepsilon<\min\{b-t,\,t-a\}$. 
The following result justifies that we may take powers of differential quotients of a solution, as test functions. This is needed in the sequel. In the rest of the paper, we will use the abbreviated notation
\[
d\mu(x,y)=\frac{dx\,dy}{|x-y|^{N+s\,p}}.
\]
\begin{lm}[Discrete differentiation of the equation]
\label{Mlemreg}
Assume that $u$ is a local weak solution of \eqref{locweaksol} with $f=0$ in $B_2\times (-2,0]$, such that
\[
u\in L^\infty([-1,0]\times E),\qquad \mbox{ for every } E\Subset B_2.
\] 
Let $\eta$ be a non-negative Lipschitz function, with compact support in $B_2$. Let $\tau$ be a smooth non-negative function such that $0\le \tau\le 1$ and
\[
\tau(t)=0\quad \mbox{ for }t\leq T_0, \qquad \tau(t)=1 \quad \mbox{ for }t\ge T_1 
\]
for some $-1<T_0<T_1< 0$.
\par
Then, for any locally Lipschitz function $F:\mathbb{R}\to\mathbb{R}$ and any $h\in\mathbb{R}^N$ such that $0<h<\mathrm{dist\,}(\mathrm{supp\,}\eta, \partial B_2)/4$, we have
\begin{equation}
\label{eq:diffeq}
\begin{split}
\int_{T_0}^{T_1}\iint_{\mathbb{R}^N\times\mathbb{R}^N}& {\Big(J_p(u_h(x,t)-u_h(y,t))-J_p(u(x,t)-u(y,t))\Big)}\\
&\times\Big(F(u_h(x,t)-u(x,t))\,\eta(x)^p-F(u_h(y,t)-u(y,t))\,\eta(y)^p\Big)\tau(t)\,d\mu\, dt\\
&+\int_{B_2} \mathcal{F}(\delta_h u(x,T_1))\, \eta(x)^p\, dx=\int_{T_0}^{T_1}\int_{B_2} \mathcal{F}(\delta_h u)\, \eta^p\,\tau'\, dx \,dt,
\end{split}
\end{equation} 
where $\mathcal{F}(t)=\int_0^t F(\rho)\,d\rho$.
\end{lm}
\begin{proof} We take $J=[T_0,T_1]\subset (-1,0)$ and $\phi\in L^p(J;W^{s,p}(B_2))\cap C^1(J;L^2(B_2))$, whose spatial support is compactly contained in $B_2$. 
We want to use the time-regularization $\phi^\varepsilon$ as test function in \eqref{locweaksol}. For this, we take 
\[
0<\varepsilon<\varepsilon_0:=\frac{1}{2}\,\min\{-T_1,T_0+1,T_1-T_0\}.
\]
Then, we preliminary observe that from elementary properties of convolutions, Fubini's Theorem and integration by parts, we have
\[
\begin{split}
-\int_{T_0}^{T_1} \int_{B_2} u(x,t)\,\partial_t \phi^\varepsilon(x,t)\,dx\,dt&=-\int_{B_2}\int_{T_0}^{T_1}  u(x,t)\,(\partial_t \phi)^\varepsilon\,dt\,dx\\
&=-\int_{B_2} \int_{T_0}^{T_1} \frac{1}{\varepsilon}\,\int_{t-\frac{\varepsilon}{2}}^{t+\frac{\varepsilon}{2}} u(x,t)\,\partial_\ell\phi(x,\ell)\,\zeta\left(\frac{t-\ell}{\varepsilon}\right)\,d\ell\,dt\,dx\\
&=-\int_{B_2}\int_{T_0+\frac{\varepsilon}{2}}^{T_1-\frac{\varepsilon}{2}} u^\varepsilon(x,\ell)\,\partial_\ell\phi(x,\ell)\,d\ell\,dx\\
&-\int_{B_2}\int_{T_0-\frac{\varepsilon}{2}}^{T_0+\frac{\varepsilon}{2}}  \left(\frac{1}{\varepsilon}\,\int_{T_0}^{\ell+\frac{\varepsilon}{2}} u(x,t)\,\zeta\left(\frac{\ell-t}{\varepsilon}\right)\,dt\right)\,\partial_\ell\phi(x,\ell)\,d\ell\,dx\\
&-\int_{B_2}\int_{T_1-\frac{\varepsilon}{2}}^{T_1+\frac{\varepsilon}{2}}  \left(\frac{1}{\varepsilon}\,\int_{\ell-\frac{\varepsilon}{2}}^{T_1} u(x,t)\,\zeta\left(\frac{\ell-t}{\varepsilon}\right)\,dt\right)\,\partial_\ell\phi(x,\ell)\,d\ell\,dx\\
&=\int_{B_2}\int_{T_0+\frac{\varepsilon}{2}}^{T_1-\frac{\varepsilon}{2}} \partial_\ell u^\varepsilon(x,\ell)\,\phi(x,\ell)\,d\ell\,dx+\Sigma(\varepsilon)\\
&-\int_{B_2} \left[u^\varepsilon\left(x,T_1-\frac{\varepsilon}{2}\right)\,\phi\left(x,T_1-\frac{\varepsilon}{2}\right)-u^\varepsilon\left(x,T_0+\frac{\varepsilon}{2}\right)\,\phi\left(x,T_0+\frac{\varepsilon}{2}\right)\right]\,dx.
\end{split}
\]
For simplicity, we have set
\[
\begin{split}
\Sigma(\varepsilon)=&-\int_{B_2}\int_{T_0-\frac{\varepsilon}{2}}^{T_0+\frac{\varepsilon}{2}}  \left(\frac{1}{\varepsilon}\,\int_{T_0}^{\ell+\frac{\varepsilon}{2}} u(x,t)\,\zeta\left(\frac{\ell-t}{\varepsilon}\right)\,dt\right)\,\partial_\ell\phi(x,\ell)\,d\ell\,dx\\
&-\int_{B_2}\int_{T_1-\frac{\varepsilon}{2}}^{T_1+\frac{\varepsilon}{2}}  \left(\frac{1}{\varepsilon}\,\int_{\ell-\frac{\varepsilon}{2}}^{T_1} u(x,t)\,\zeta\left(\frac{\ell-t}{\varepsilon}\right)\,dt\right)\,\partial_\ell\phi(x,\ell)\,d\ell\,dx.
\end{split}
\]
Thus from \eqref{locweakeq} it follows that for $0<\varepsilon<\varepsilon_0$ 
\begin{equation}
\label{eq:epseq}
\begin{split}
\int_{T_0}^{T_1} \iint_{\mathbb{R}^N\times\mathbb{R}^N} &\Big(J_p(u(x,t)-u(y,t))\Big)\,\Big(\phi^\varepsilon(x,t)-\phi^\varepsilon(y,t)\Big)\,d\mu(x,y)\,dt\\
&+\int_{B_2}\int_{T_0+\frac{\varepsilon}{2}}^{T_1-\frac{\varepsilon}{2}} \partial_t u^\varepsilon(x,t)\,\phi(x,t)\,dt\,dx+\Sigma(\varepsilon)\\
& = \int_{B_2} \left[u(x,T_0)\,\phi(x,T_0)-u^\varepsilon\left(x,T_0+\frac{\varepsilon}{2}\right)\,\phi\left(x,T_0+\frac{\varepsilon}{2}\right)\right]\,dx \\
&+\int_{B_2} \left[u^\varepsilon\left(x,T_1-\frac{\varepsilon}{2}\right)\,\phi\left(x,T_1-\frac{\varepsilon}{2}\right)-u(x,T_1)\,\phi(x,T_1)\right]\,dx,
\end{split}
\end{equation}
Before proceeding further, we observe that by using an integration by parts, the term $\Sigma(\varepsilon)$ can be rewritten as
\[
\begin{split}
\Sigma(\varepsilon)=&-\int_{B_2}\left(\frac{1}{\varepsilon}\,\int_{T_0}^{T_0+\varepsilon} u(x,t)\,\zeta\left(\frac{T_0-t}{\varepsilon}+\frac{1}{2}\right)\,dt\right)\,\phi\left(x,T_0+\frac{\varepsilon}{2}\right)\,dx\\
&+\int_{B_2}\int_{T_0-\frac{\varepsilon}{2}}^{T_0+\frac{\varepsilon}{2}}  \left(\frac{1}{\varepsilon^2}\,\int_{T_0}^{\ell+\frac{\varepsilon}{2}} u(x,t)\,\zeta'\left(\frac{\ell-t}{\varepsilon}\right)\,dt\right)\,\phi(x,\ell)\,d\ell\,dx\\
&+\int_{B_2}\left(\frac{1}{\varepsilon}\,\int^{T_1}_{T_1-\varepsilon} u(x,t)\,\zeta\left(\frac{T_1-t}{\varepsilon}-\frac{1}{2}\right)\,dt\right)\,\phi\left(x,T_1-\frac{\varepsilon}{2}\right)\,dx\\
&-\int_{B_2}\int_{T_1-\frac{\varepsilon}{2}}^{T_1+\frac{\varepsilon}{2}}  \left(\frac{1}{\varepsilon^2}\,\int^{T_1}_{\ell-\frac{\varepsilon}{2}} u(x,t)\,\zeta'\left(\frac{\ell-t}{\varepsilon}\right)\,dt\right)\,\phi(x,\ell)\,d\ell\,dx,
\end{split}
\]
where we also used that $\zeta$ has compact support in $(-1/2,1/2)$.
By further using a suitable change of variables, we can also write
\begin{equation}
\label{sigma}
\begin{split}
\Sigma(\varepsilon)&=-\int_{B_2}\left(\frac{1}{\varepsilon}\,\int_{T_0}^{T_0+\varepsilon} u(x,t)\,\zeta\left(\frac{T_0-t}{\varepsilon}+\frac{1}{2}\right)\,dt\right)\,\phi\left(x,T_0+\frac{\varepsilon}{2}\right)\,dx\\
&+\int_{B_2}\int_{-\frac{1}{2}}^{\frac{1}{2}}  \left(\int_{-\frac{1}{2}}^{\rho} u(x,\varepsilon\,\rho+T_0-\varepsilon\,\sigma)\,\zeta'\left(\sigma\right)\,d\sigma\right)\,\phi(x,\varepsilon\,\rho+T_0)\,d\rho\,dx\\
&+\int_{B_2}\left(\frac{1}{\varepsilon}\,\int^{T_1}_{T_1-\varepsilon} u(x,t)\,\zeta\left(\frac{T_1-t}{\varepsilon}-\frac{1}{2}\right)\,dt\right)\,\phi\left(x,T_1-\frac{\varepsilon}{2}\right)\,dx\\
&-\int_{B_2}\int_{-\frac{1}{2}}^{\frac{1}{2}}  \left(\int^{\rho}_{\frac{1}{2}} u(x,\varepsilon\,\rho+T_1-\varepsilon\,\sigma)\,\zeta'\left(\sigma\right)\,d\sigma\right)\,\phi(x,\varepsilon\,\rho+T_1)\,d\rho\,dx\\
\end{split}
\end{equation}
By testing \eqref{eq:epseq} with $\phi_{-h}(x,t)=\phi(x-h,t)$ for 
\[
0<|h|<\frac{h_0}{4}:=\frac{1}{4}\,\mathrm{dist}(\mathrm{supp\,}\phi, \partial B_2),
\] 
and then changing variables, we get
\begin{equation}
\label{equationepsh}
\begin{split}
\int_{T_0}^{T_1}\iint_{\mathbb{R}^N\times\mathbb{R}^N}& \Big(J_p(u_h(x,t)-u_h(y,t))\Big)\,\Big(\phi^\varepsilon(x,t)-\phi^\varepsilon(y,t)\Big)\,d\mu(x,y)\,dt\\
&+\int_{B_2}\int_{T_0-\frac{\varepsilon}{2}}^{T_1+\frac{\varepsilon}{2}} \partial_t u^\varepsilon_h\,\phi\,dt\,dx+\Sigma_h(\varepsilon)\\
& = \int_{B_2} \left[u_h(x,T_0)\,\phi(x,T_0)-u^\varepsilon_h\left(x,T_0+\frac{\varepsilon}{2}\right)\,\phi\left(x,T_0+\frac{\varepsilon}{2}\right)\right]\,dx \\
&+\int_{B_2} \left[u^\varepsilon_h\left(x,T_1-\frac{\varepsilon}{2}\right)\,\phi\left(x,T_1-\frac{\varepsilon}{2}\right)-u_h(x,T_1)\,\phi(x,T_1)\,dx\right]\,dx.
\end{split}
\end{equation}
The quantity $\Sigma_h(\varepsilon)$ is defined as in \eqref{sigma}, with $u_h$ in place of $u$.
We subtract \eqref{eq:epseq} from \eqref{equationepsh}, so to get
\begin{equation}
\label{differentiatedeps}
\begin{split}
&\int_{T_0}^{T_1}\iint_{\mathbb{R}^N\times\mathbb{R}^N} \Big(J_p(u_h(x,t)-u_h(y,t))-J_p(u(x,t)-u(y,t))\Big)\,\Big(\phi^\varepsilon(x,t)-\phi^\varepsilon(y,t)\Big)\,d\mu\, dt\\
&+\int_{T_0}^{T_1}\int_{B_2}\partial_t(u^\varepsilon_h-u^\varepsilon)\,\phi\, dx\, dt+(\Sigma_h(\varepsilon)-\Sigma(\varepsilon))\\
& = \int_{B_2} \left[\delta_h u(x,T_0)\,\phi(x,T_0)-\delta_h u^\varepsilon\left(x,T_0+\frac{\varepsilon}{2}\right)\,\phi\left(x,T_0+\frac{\varepsilon}{2}\right)\right]\,dx \\
&+\int_{B_2} \left[\delta_h u^\varepsilon\left(x,T_1-\frac{\varepsilon}{2}\right)\,\phi\left(x,T_1-\frac{\varepsilon}{2}\right)-\delta_ h u(x,T_1)\,\phi(x,T_1)\,dx\right]\,dx,
\end{split}
\end{equation}
for every $\phi\in L^p([T_0,T_1];W^{s,p}(B_2))\cap C^1([T_0,T_1];L^2(B_2))$, whose spatial support is compactly contained in $B_2$. We take $F$ as in the statement and use \eqref{differentiatedeps} with the test function
\[
\phi=F(u^\varepsilon_h-u^\varepsilon)\,\eta^p\,\tau_\varepsilon=F(\delta_h u^\varepsilon)\,\eta^p\,\tau_\varepsilon,
\]
where
\[
\tau_\varepsilon(t)=\tau\left(\frac{T_1-T_0}{T_1-T_0-\varepsilon}\,\left(t-T_1+\frac{\varepsilon}{2}\right)+T_1\right),
\] 
and $\eta$ and $\tau$ are as in the statement. By observing that
\[
\tau_\varepsilon(t)=0,\quad \mbox{ for }t\le T_0+\frac{\varepsilon}{2},\qquad \tau_\varepsilon(t)=1,\quad \mbox{ for }t\ge T_1-\frac{\varepsilon}{2},
\]
we get
\begin{equation}
\label{everyday}
\begin{split}
\int_{T_0}^{T_1}\iint_{\mathbb{R}^N\times\mathbb{R}^N}& {\Big(J_p(u_h(x,t)-u_h(y,t))-J_p(u(x,t)-u(y,t))\Big)}\\
&\times\left(\Big(F(\delta_h u^\varepsilon(x,t))\,\tau_\varepsilon(t)\Big)^\varepsilon\,\eta(x)^p-\Big( F(\delta_h u^\varepsilon(y,t))\,\tau_\varepsilon(t)\Big)^\varepsilon\,\eta(y)^p\right)\,d\mu\, dt\\
&+\int_{T_0}^{T_1}\int_{B_2}\partial_t(\delta_h u^\varepsilon)\,F(\delta_h u^\varepsilon)\,\eta^p(x)\,\tau_\varepsilon(t)\, dx\, dt+(\Sigma_h(\varepsilon)-\Sigma(\varepsilon))\\
&=\int_{B_2} \left[\delta_h u^\varepsilon\left(x,T_1-\frac{\varepsilon}{2}\right) F\left(\delta_h u^\varepsilon\left(x,T_1-\frac{\varepsilon}{2}\right)\right)-\delta_h u\left(x,T_1\right) F\left(\delta_h u\left(x,T_1\right)\right)\right]\,\eta^p\,dx.
\end{split}
\end{equation}
Observe that we used the properties of $\tau_\varepsilon$. In order to deal with the integral containing the time derivative of $\delta_h u^\varepsilon$, we first observe that
\[
\partial_t(\delta_h u^\varepsilon)\,F(\delta_h u^\varepsilon)=\partial_t \mathcal{F}(\delta_h u^\varepsilon).
\] 
Thus we can use integration by parts, which yields
\[
\begin{split}
\int_{T_0}^{T_1}\int_{B_2}\partial_t(\delta_h u^\varepsilon(x,t))\,F(\delta_h u^\varepsilon)\,\eta^p(x)\,\tau_\varepsilon(t)\, dx\, dt&=\int_{B_2} \mathcal{F}(\delta_h u^\varepsilon(x,T_1))\, \eta(x)^p \,dx-\int_{-1}^{T_1}\int_{B_2} \mathcal{F}(\delta_h u^\varepsilon)\, \eta^p\,\tau'_\varepsilon\, dx\, dt.
\end{split}
\]
By inserting this into \eqref{everyday}, we get
\begin{equation}
\label{dioboe}
\begin{split}
\int_{T_0}^{T_1}\iint_{\mathbb{R}^N\times\mathbb{R}^N}& {\Big(J_p(u_h(x,t)-u_h(y,t))-J_p(u(x,t)-u(y,t))\Big)}\\
&\times\left(\Big(F(\delta_h u^\varepsilon(x,t))\,\tau_\varepsilon(t)\Big)^\varepsilon\,\eta(x)^p-\Big( F(\delta_h u^\varepsilon(y,t))\,\tau_\varepsilon(t)\Big)^\varepsilon\,\eta(y)^p\right)\,d\mu\, dt\\
&+\int_{B_2}\mathcal{F}(\delta_h u^\varepsilon(x,T_1))\, \eta(x)^p \,dx\\
&-\int_{T_0}^{T_1}\int_{B_2} \mathcal{F}(\delta_h u^\varepsilon)\, \eta^p\,\tau'_\varepsilon\, dx\, dt+(\Sigma_h(\varepsilon)-\Sigma(\varepsilon))\\
&=\int_{B_2} \left[\delta_h u^\varepsilon\left(x,T_1-\frac{\varepsilon}{2}\right) F\left(\delta_h u^\varepsilon\left(x,T_1-\frac{\varepsilon}{2}\right)\right)-\delta_h u\left(x,T_1\right) F\left(\delta_h u\left(x,T_1\right)\right)\right]\,\eta^p\,dx.
\end{split}
\end{equation}
We recall that this is valid for 
\[
0<|h|<\frac{h_0}{4}\qquad \mbox{ and }\qquad 0<\varepsilon<\varepsilon_0.
\]
Before taking the limit as $\varepsilon$ goes to $0$, we first observe that for $t\in [T_0-\varepsilon/2,T_1+\varepsilon/2]$ and $x\in B_{2-2\,h}$ we have
\[
\begin{split}
|\delta_h u^\varepsilon(x,t)|&\le \frac{1}{\varepsilon}\,\int_{-\frac{\varepsilon}{2}}^\frac{\varepsilon}{2}\,\zeta\left(\frac{\sigma}{\varepsilon}\right)\,|\delta_h u(x,t-\sigma)|\,d\sigma\\
&=\int_{-\frac{1}{2}}^\frac{1}{2} \zeta(\sigma)\,|\delta_h u(x,t-\varepsilon\,\sigma)|\,d\sigma\le \|\delta_h u\|_{L^\infty([T_0-\varepsilon_0,T_1+\varepsilon_0]\times B_{2-2\,h})}.
\end{split}
\]
This shows that we have the uniform $L^\infty$ estimate
\begin{equation}
\label{unifaus}
\|\delta_h u^\varepsilon\|_{L^\infty\left(\left[T_0-\frac{\varepsilon}{2},T_1+\frac{\varepsilon}{2}\right]\times B_{2-2\,h}\right)}\le 2\,\|u\|_{L^\infty([T_0-\varepsilon_0,T_1+\varepsilon_0]\times B_{2-h})},\qquad \mbox{ for } 0<\varepsilon<\varepsilon_0,\, 0<|h|<\frac{h_0}{4}.
\end{equation}
Finally, we pass to the limit in \eqref{dioboe} as $\varepsilon$ goes to $0$.
We start from the right-hand side: by using the local Lipschitz regularity of $F$ and \eqref{unifaus}, we have
\[
\begin{split}
\Big|\delta_h u^\varepsilon\left(x,T_1-\frac{\varepsilon}{2}\right)& F\left(\delta_h u^\varepsilon\left(x,T_1-\frac{\varepsilon}{2}\right)\right)-\delta_h u\left(x,T_1\right) F\left(\delta_h u\left(x,T_1\right)\right)\Big|\,\eta(x)^p\\
&\le C\, \left|\delta_h u^\varepsilon\left(x,T_1-\frac{\varepsilon}{2}\right)-\delta_h u\left(x,T_1\right)\right|\,\eta(x)^p\\
&\le C\, \left|u^\varepsilon\left(x+h,T_1-\frac{\varepsilon}{2}\right)-u(x+h,T_1)\right|\,\eta(x)^p\\
&+C\, \left|u^\varepsilon\left(x,T_1-\frac{\varepsilon}{2}\right)-u(x,T_1)\right|\,\eta(x)^p,
\end{split}
\] 
where $C>0$ does not depend on $\varepsilon$. Thus, by using that $\eta$ has compact support in $B_2$ and $0<|h|<h_0/4$, we get from the last estimate (after a change of variable)
\[
\begin{split}
\Big|\int_{B_2}& 
\Big|\delta_h u^\varepsilon\left(x,T_1-\frac{\varepsilon}{2}\right) F\left(\delta_h u^\varepsilon\left(x,T_1-\frac{\varepsilon}{2}\right)\right)-\delta_h u\left(x,T_1\right) F\left(\delta_h u\left(x,T_1\right)\right)\Big|
\,\eta(x)^p\,dx\Big|\\
&\le C\,\int_{B_{2-2\,h}} \left|u^\varepsilon\left(x,T_1-\frac{\varepsilon}{2}\right)-u(x,T_1)\right|\,dx\\
&\le C\,\int_{B_{2-2\,h}} \left|\frac{1}{\varepsilon}\,\int_{-\frac{\varepsilon}{2}}^{\frac{\varepsilon}{2}}\zeta\left(\frac{\sigma}{\varepsilon}\right)\,\left[u\left(x,T_1-\frac{\varepsilon}{2}-\sigma\right)-u(x,T_1)\right]\,d\sigma \right|\,dx\\
&\le C\,\int_{-\frac{1}{2}}^{\frac{1}{2}}\zeta\left(\rho\right)\,\left(\int_{B_{2-2\,h}} \left|u\left(x,T_1-\frac{\varepsilon}{2}-\varepsilon\,\rho\right)-u(x,T_1)\right|\,dx\right)\,d\rho \\
&\le C\,\sup_{-\varepsilon\le t\le 0} \int_{B_{2-2\,h}} \left|u\left(x,T_1-t\right)-u(x,T_1)\right|\,dx. 
\end{split}
\]
The constant $C$ is still independent of $0<\varepsilon<\varepsilon_0$.
If we now use that $u\in C((-2,0];L^2_{\rm loc}(B_2))$, we get that the last quantity converges to $0$, as $\varepsilon$ goes to $0$.
\par
For the term 
\[
\int_{B_2} \mathcal{F}(\delta_h u^\varepsilon(x,T_1))\, \eta(x)^p \,dx,
\]
we proceed similarly as above.
We observe that 
\[
\begin{split}
&\left|\int_{B_2} \mathcal{F}(\delta_h u^\varepsilon(x,T_1))\, \eta(x)^p \,dx-\int_{B_2} \mathcal{F}(\delta_h u(x,T_1))\, \eta(x)^p\, dx\right|
\\
&\le \int_{B_{2-2\,h}}\left( \frac{1}{\varepsilon}\,\int_{-\frac{\varepsilon}{2}}^\frac{\varepsilon}{2} \zeta\left(\frac{\sigma}{\varepsilon}\right)\,\Big|\mathcal{F}(\delta_h u(x,T_1-\sigma))-\mathcal{F}(\delta_h u(x,T_1))\Big|\,d\sigma\right)\,dx\\
&=\frac{1}{\varepsilon}\,\int_{-\frac{\varepsilon}{2}}^\frac{\varepsilon}{2} \zeta\left(\frac{\sigma}{\varepsilon}\right)\left( \int_{B_{2-2\,h}}\,\Big|\mathcal{F}(\delta_h u(x,T_1-\sigma))-\mathcal{F}(\delta_h u(x,T_1))\Big|\,\eta(x)^p\,dx\right)\,d\sigma\\
&\le C\, \sup_{-\frac{\varepsilon}{2}\le t\le \frac{\varepsilon}{2}}\int_{B_{2-2\,h}}\,\Big|\delta_h u(x,T_1-t)-\delta_h u(x,T_1)\Big|\,\eta(x)^p\,dx.
\end{split}
\]
We can now use again that $u\in C((-2,0];L^2_{\rm loc}(B_2))$ and obtain that the last quantity converges to $0$, as $\varepsilon$ goes to $0$.
\par
As for the term
\[
-\int_{T_0}^{T_1}\int_{B_2} \mathcal{F}(\delta_h u^\varepsilon)\, \eta^p\,\tau'_\varepsilon\, dx\, dt,
\] 
we can proceed exactly as before, we omit the details. In a similar fashion, we can also show that
\[
\lim_{\varepsilon\to 0} |\Sigma_h(\varepsilon)-\Sigma(\varepsilon)|=0.
\]
This is still similar to the previous limits.
It is sufficient to use the expression \eqref{sigma}, the uniform $L^\infty$ estimate \eqref{unifaus} and the fact $u\in C((-2,0];L^2_{\rm loc}(B_2))$, in order to apply the Lebesgue Dominated Convergence Theorem. 
\par
Finally, the convergence of the double integral requires quite lengthy computations and thus we prefer to postpone them to Appendix \ref{sec:lemma33} below.
\end{proof}

\begin{oss}
We observe that the global $L^\infty$ bound on the weak solution is not needed in the previous result. It is sufficient to know that the weak solution is locally bounded. We refer to \cite[Theorem 1.1]{str2} for local boundedness of weak solutions. 
\end{oss}

\section{Spatial almost $C^s$-regularity}
\label{sec:almost}
The following result is an integrability gain for the discrete derivative of order $s$ of a local weak solution. This is the parabolic version of \cite[Proposition 4.1]{bralinschi}, to which we refer for all the missing details.
\begin{prop}
\label{prop:improve} 
Assume $p\ge 2$ and $0<s<1$. Let  $u$ be a local weak solution of $u_t+(-\Delta_p)^s u=0$ in $B_2\times (-2,0]$. We assume that 
\[
\|u\|_{L^\infty(\mathbb{R}^N\times [-1,0])}\leq 1,
\]
and that, for some $q\ge p$ and $0<h_0<1/10$, we have
\[
\int_{T_0}^{T_1}\sup_{0<|h|< h_0}\left\|\frac{\delta^2_h u }{|h|^s}\right\|_{L^q(B_{R+4\,h_0})}^q dt<+\infty,
\]
for a radius $4\,h_0<R\le 1-5\,h_0$ and two time instants $-1<T_0<T_1\le 0$. Then we have
\begin{equation}
\label{iteralo}
\begin{split}
\int_{T_0+\mu}^{T_1}\sup_{0<|h|< h_0}\left\|\frac{\delta^2_h u}{|h|^{s}}\right\|_{L^{q+1}(B_{R-4\,h_0})}^{q+1}dt &+\frac{1}{q+3-p}\sup_{0<|h|< h_0}\left\|\frac{\delta_h u(\cdot,T_1)}{|h|^{\frac{(q+2-p)\,s}{q+3-p}}}\right\|_{L^{q+3-p}(B_{R-4\,h_0})}^{q+3-p}\\
&\leq C\,\int_{T_0}^{T_1}\left(\sup_{0<|h|< h_0}\left\|\frac{\delta^2_h u }{|h|^s}\right\|_{L^q(B_{R+4h_0})}^q+1\right)dt,
\end{split}
\end{equation}
for every $0<\mu<T_1-T_0$.
Here $C=C(N,s,p,q,h_0,\mu)>0$ and $C\nearrow +\infty$ as $h_0\searrow 0$ or $\mu\searrow 0$.
\end{prop}
\begin{proof} We divide the proof into seven steps. 
\vskip.2cm\noindent
{\noindent}{\bf Step 1: Discrete differentiation of the equation.} 
We take for the moment $T_1<0$, then we will show at the end of the proof how to include the case $T_1=0$.
We already introduced the notation
\[
d\mu(x,y)=\frac{dx\,dy}{|x-y|^{N+s\,p}}.
\]
For notational simplicity, we also set
\[
r=R-4\,h_0.
\]
Let $\beta\ge 2$ and $\vartheta\in \mathbb{R}$ be such that $0<1+\vartheta\,\beta<\beta$, and use \eqref{eq:diffeq} for $0<|h|< h_0$, 
where:
\begin{itemize}
\item $F(t)=J_{\beta+1}(t)=|t|^{\beta-1}\,t$, which is locally Lipschitz for $\beta\ge 1$;
\vskip.2cm
\item $\eta$ is a non-negative standard Lipschitz cut-off function supported in $B_{(R+r)/2}$, such that
\[
\eta\equiv 1 \quad \mbox{ on } B_r\qquad \mbox{�and }\qquad |\nabla \eta|\le \frac{C}{R-r}=\frac{C}{4\,h_0}.
\] 
\item $\tau$ is a smooth function such that $0\le \tau\le 1$ and
\[
\tau\equiv 1\quad \mbox{ on }  [T_0+\mu,+\infty), \qquad \tau\equiv 0\quad \mbox{ on } (-\infty,T_0],\qquad |\tau'|\le \frac{C}{\mu}.
\]
Here $\mu$ is as in the statement, i.e. any positive number such that $\mu<T_1-T_0$. 
\end{itemize}
Note that the assumptions on $\eta$ imply
$$
\left|\frac{\delta_h \eta}{|h|}\right|\leq \frac{C}{h_0}.
$$
After dividing by $|h|^{1+\vartheta\,\beta}$, we obtain from Lemma \ref{Mlemreg},
\[
\begin{split}
\int_{T_0}^{T_1}\iint_{\mathbb{R}^N\times\mathbb{R}^N}& \frac{\Big(J_p(u_h(x,t)-u_h(y,t))-J_p(u(x,t)-u(y,t))\Big)}{|h|^{1+\vartheta\,\beta}}\\
&\times\Big(J_{\beta+1}(u_h(x,t)-u(x,t))\,\eta(x)^p-J_{\beta+1}(u_h(y,t)-u(y,t))\,\eta(y)^p\Big)\tau(t)\,d\mu\, dt\\
&+\frac{1}{\beta+1}\int_{B_2}\frac{|\delta_h u(x,T_1)|^{\beta+1}}{|h|^{1+\vartheta\beta}}\, \eta^p dx\,=\frac{1}{\beta+1}\int_{T_0}^{T_1}\int_{B_2}\frac{|\delta_h u|^{\beta+1}}{|h|^{1+\vartheta\beta}}\, \eta^p\,\tau'\, dx\, dt.
\end{split}
\]
The triple integral is now divided into three pieces:
\[
\widetilde{\mathcal{I}_i}:=\int_{T_0}^{T_1} {\mathcal{I}_i}(t)\,\tau(t)\,   dt, \qquad i=1,2,3,
\]
where
\[
\begin{split}
\mathcal{I}_1(t):=\iint_{B_R\times B_R}& \frac{\Big(J_p(u_h(x)-u_h(y))-J_p(u(x)-u(y))\Big)}{|h|^{1+\vartheta\,\beta}}\\
&\times\Big(J_{\beta+1}(u_h(x)-u(x))\,\eta(x)^p-J_{\beta+1}(u_h(y)-u(y))\,\eta(y)^p\Big)\,d\mu ,
\end{split}
\]
\[
\begin{split}
\mathcal{I}_2(t):=\int_{B_\frac{R+r}{2}\times (\mathbb{R}^N\setminus B_R)}& \frac{\Big(J_p(u_h(x)-u_h(y))-J_p(u(x)-u(y))\Big)}{|h|^{1+\vartheta\,\beta}}\,J_{\beta+1}(u_h(x)-u(x))\,\eta(x)^p\,d\mu ,
\end{split}
\]
and
\[
\begin{split}
\mathcal{I}_3(t):=-\iint_{(\mathbb{R}^N\setminus B_R)\times B_\frac{R+r}{2}}& \frac{\Big(J_p(u_h(x)-u_h(y))-J_p(u(x)-u(y))\Big)}{|h|^{1+\vartheta\,\beta}}\,J_{\beta+1}(u_h(y)-u(y))\,\eta(y)^p\,d\mu ,
\end{split}
\]
where we used that $\eta$ vanishes identically outside $B_{(R+r)/2}$. We also suppressed the $t-$dependence, for notational simplicity. We also have the term in the right-hand side
$$
\mathcal{I}_4:=\frac{1}{\beta+1}\int_{T_0}^{T_1}\int_{B_2}\frac{|\delta_h u|^{\beta+1}}{|h|^{1+\vartheta\beta}}\, \eta^p\,\tau'\, dx\, dt.
$$
By proceeding exactly as in Step 1 of the proof of \cite[Proposition 4.1]{bralinschi}, 
we get the following lower bound for $\mathcal{I}_1(t)$
\[
\begin{split}
\mathcal{I}_1\ge c& \left[\frac{|\delta_h u|^\frac{\beta-1}{p}\,\delta_h u}{|h|^\frac{1+\vartheta\,\beta}{p}}\,\eta\right]^p_{W^{s,p}(B_R)}\\
&-C\,\iint_{B_R\times B_R} \left(|u_h(x)-u_h(y)|^\frac{p-2}{2}+|u(x)-u(y)|^\frac{p-2}{2}\right)^2\,\left|\eta(x)^\frac{p}{2}-\eta(y)^\frac{p}{2}\right|^2\\
&\times \frac{|u_h(x)-u(x)|^{\beta+1}+|u_h(y)-u(y)|^{\beta+1}}{|h|^{1+\vartheta\,\beta}}\,d\mu\\
&-C\,\iint_{B_R\times B_R}\, \left(\frac{|\delta_h u(x)|^{\beta-1+p}}{|h|^{1+\vartheta\,\beta}}+\frac{|\delta_h u(y)|^{\beta-1+p}}{|h|^{1+\vartheta\,\beta}}\right)\, |\eta(x)-\eta(y)|^p\,d\mu,
\end{split}
\]
where $c=c(p,\beta)>0$ and $C=C(p,\beta)>0$. We use that 
\[
\widetilde{\mathcal{I}_1}+\widetilde{\mathcal{I}_2}+\widetilde{\mathcal{I}_3} +\frac{1}{\beta+1}\int_{B_2}\frac{|\delta_h u(x,T_1)|^{\beta+1}}{|h|^{1+\vartheta\beta}}\, \eta^p\, dx=\mathcal{I}_4
\]
and the estimate for $\mathcal{I}_1(t)$. This entails that
\begin{equation}
\label{Ieq}
\begin{split}
\int_{T_0}^{T_1} \left[\frac{|\delta_h u|^\frac{\beta-1}{p}\,\delta_h u}{|h|^\frac{1+\vartheta\,\beta}{p}}\,\eta\right]^p_{W^{s,p}(B_R)}\tau\, dt&+\frac{1}{\beta+1}\int_{B_2}\frac{|\delta_h u(x,T_1)|^{\beta+1}}{|h|^{1+\vartheta\beta}}\, \eta^p\, dx\\
&\leq C\,\Big(\mathcal{\widetilde I}_{11}+\mathcal{\widetilde I}_{12}+|\mathcal{\widetilde I}_2|+|\mathcal{\widetilde I}_3|\Big)+\mathcal{I}_4,\quad \mbox{ for }C=C(p,\beta)>0, 
\end{split}
\end{equation}
where we set $\mathcal{\widetilde I}_{11}=\int_{T_0}^{T_1} \mathcal{I}_{11}\, \tau\, dt$, $\mathcal{\widetilde I}_{12}=\int_{T_0}^{T_1} \mathcal{I}_{12}\, \tau\, dt$ and
\begin{equation}
\label{I11}
\begin{split}
\mathcal{I}_{11}(t)&:=\,\iint_{B_R\times B_R} \left(|u_h(x)-u_h(y)|^\frac{p-2}{2}+|u(x)-u(y)|^\frac{p-2}{2}\right)^2\,\left|\eta(x)^\frac{p}{2}-\eta(y)^\frac{p}{2}\right|^2\, \frac{|\delta_h u(x)|^{\beta+1}+|\delta_h u(y)|^{\beta+1}}{|h|^{1+\vartheta\,\beta}}\,d\mu,
\end{split}
\end{equation}
and
\begin{equation}
\label{I12}
\begin{split}
\mathcal{I}_{12}(t)&:=\,\iint_{B_R\times B_R}\, \left(\frac{|\delta_h u(x)|^{\beta-1+p}}{|h|^{1+\vartheta\,\beta}}+\frac{|\delta_h u(y)|^{\beta-1+p}}{|h|^{1+\vartheta\,\beta}}\right)\, |\eta(x)-\eta(y)|^p\,d\mu.
\end{split}
\end{equation}
\vskip.2cm\noindent
{\bf Step 2: Estimates of the local terms $\mathcal{I}_{11}$ and $\mathcal{I}_{12}$.} Here we can follow the same computations as in Step 2 of the proof of \cite[Proposition 4.1]{bralinschi}, so to get
\[
|\mathcal{I}_{11}|\leq C\,\left(1+\int_{B_{R}}\left|\frac{\delta_h u}{|h|^{\frac{1+\vartheta\beta}{\beta}}}\right|^{\frac{\beta\, q}{q-p+2}}\,dx +\sup_{0<|h|< h_0}\left\|\frac{\delta^2_h u}{|h|^s}\right\|_{L^q(B_{R+4h_0})}^q+1\right),
\]
and 
\[
|\mathcal{I}_{12}|\leq C\left( 1+\int_{B_{R}}\left|\frac{\delta_h u}{|h|^{\frac{1+\vartheta\beta}{\beta}}}\right|^{\frac{\beta\, q}{q-p+2}}\,dx+1\right),
\]
for some $C=C(N,h_0,p,s,q)>0$.
If we now use these estimates in \eqref{Ieq}, we get
\begin{equation}
\label{Ieq2}
\begin{split}
\int_{T_0}^{T_1} \left[\frac{|\delta_h u|^\frac{\beta-1}{p}\,\delta_h u}{|h|^\frac{1+\vartheta\,\beta}{p}}\,\eta\right]^p_{W^{s,p}(B_R)}&\,\tau\, dt+\frac{1}{\beta+1}\int_{B_1}\frac{|\delta_h u(x,T_1)|^{\beta+1}}{|h|^{1+\vartheta\beta}}\, \eta^p\, dx\\
&\leq C\int_{T_0}^{T_1} \,\left(\int_{B_R}\left|\frac{\delta_h u}{|h|^\frac{1+\vartheta\,\beta}{\beta}}\right|^\frac{\beta\,q}{q-p+2}\, dx 
 +\sup_{0<|h|< h_0}\left\|\frac{\delta^2_h u}{|h|^s}\right\|_{L^q(B_{R+4h_0})}^q+1\right)\,\tau\, dt\\
&+C\,\Big(|\widetilde{\mathcal{I}_2}|+|\widetilde{\mathcal{I}_3}|+\mathcal{I}_4\Big),
\end{split}
\end{equation}
with $C=C(h_0,N,p,s,q,\beta)>0$.
\vskip.2cm\noindent
{\noindent}{\bf Step 3: Estimates of the nonlocal terms $\mathcal{I}_2$ and $\mathcal{I}_3$.}
Both nonlocal terms $\mathcal{I}_2$ and $\mathcal{I}_3$ can be treated in the same way. We only estimate $\mathcal{I}_2$ for simplicity. 
We can use that $|u|\le 1$ on $\mathbb{R}^N\times[-1,0]$ to infer that 
\[
\begin{split}
\Big|(J_p(u_h(x)-u_h(y))-J_p(u(x)-u(y))\, J_{\beta+1}(\delta_h u(x))\Big|&\leq C\left(1+|u_h(y)|^{p-1}+|u(y)|^{p-1}\right)\,|\delta_h u(x)|^{\beta}
\\
&\le 3\,C\,|\delta_h u(x)|^{\beta},
\end{split}
\]
where $C=C(p)>0$. As in \cite{bralinschi}, we observe that for $x\in B_{(R+r)/2}$ we have $B_{(R-r)/2}(x)\subset B_{R}$. This entails
$$
\int_{\mathbb{R}^N\setminus B_{R}}\frac{1}{|x-y|^{N+s\,p}}\,  dy\leq \int_{\mathbb{R}^N\setminus B_\frac{R-r}{2}(x)} \frac{1}{|x-y|^{N+s\,p}}\, dy\leq C(N,h_0,p,s).
$$
Hence, we obtain
\begin{align}
|\widetilde{\mathcal{I}_2}|+|\widetilde{\mathcal{I}_3}|&\leq C\,\int^{T_1}_{T_0}\int_{B_{\frac{R+r}{2}}}\frac{|\delta_h u|^{\beta}}{|h|^{1+\vartheta\beta}} \,\tau\,dx\, dt\leq C\,\int_{T_0}^{T_1}\left(1+\int_{B_{R}}\left|\frac{\delta_h u}{|h|^{\frac{1+\vartheta\beta}{\beta}}}\right|^{\frac{\beta\, q}{q-p+2}}\,dx\right)\tau \,dt \label{I3I4est}, 
\end{align}
by Young's inequality. Here $C=C(h_0,N,s,p,q,\beta)>0$ as before. 
\vskip.2cm\noindent
{\bf Step 4: Estimates of $\mathcal{I}_4$.} 
By using that $|u|\le 1$ in $\mathbb{R}^N\times [-1,0]$ and the properties of $\tau$, we get
\begin{equation}
\label{I4est2}
\begin{split}
|I_4|&=\frac{1}{\beta+1}\,\left|\int_{T_0}^{T_1}\,\int_{B_1}\frac{|\delta_h u|^{\beta+1}}{|h|^{1+\vartheta\beta}}\, \eta^p\,\tau'\, dx\, dt\right|\\
&\leq \frac{C}{\mu}\,\int_{T_0}^{T_1}\int_{B_{\frac{R+r}{2}}}\frac{|\delta_h u|^{\beta}}{|h|^{1+\vartheta\beta}} \,dx dt\leq \frac{C}{\mu}\,\int_{T_0}^{T_1}\left(1+\int_{B_{R}}\left|\frac{\delta_h u}{|h|^{\frac{1+\vartheta\beta}{\beta}}}\right|^{\frac{\beta\, q}{q-p+2}}\,dx\right)\,dt.
\end{split}
\end{equation}
In the last inequality we further used Young's inequality.
By inserting the estimates \eqref{I3I4est} and \eqref{I4est2} in \eqref{Ieq2}, using that $\tau$ is non-negative and such that $\tau =1$ on $[T_0+\mu,T_1]$, we obtain
\begin{equation}
\label{I2I3}
\begin{split}
\int_{T_0+\mu}^{T_1} &\left[\frac{|\delta_h u|^\frac{\beta-1}{p}\,\delta_h u}{|h|^\frac{1+\vartheta\,\beta}{p}}\,\eta\right]^p_{W^{s,p}(B_R)}\, dt+\frac{1}{\beta+1}\int_{B_R}\frac{|\delta_h u(x,T_1)|^{\beta+1}}{|h|^{1+\vartheta\beta}}\, \eta^p\, dx\\
&\leq C\,\int_{T_0}^{T_1} \left(\int_{B_{R}}\left|\frac{\delta_h u}{|h|^{\frac{1+\vartheta\beta}{\beta}}}\right|^{\frac{\beta\, q}{q-p+2}}\,dx+\sup_{0<|h|< h_0}\left\|\frac{\delta^2_h u}{|h|^s}\right\|_{L^q(B_{R+4h_0})}^q+1\right) dt.
\end{split}
\end{equation}
This is the parabolic counterpart of \cite[equation (4.10)]{bralinschi}. Observe that the constant $C$ now depends on $1/\mu$, as well.
\vskip.2cm\noindent
{\bf Step 5: Going back to the equation.} In this step, we can simply reproduce Step 4 of the proof of \cite[Proposition 4.1]{bralinschi}, so to obtain 
for any $0 < |\xi|,|h| < h_0$
\begin{equation}
\label{est4}
\left\|\frac{\delta_\xi \delta_h u}{|\xi|^\frac{s\,p}{\beta-1+p}|h|^\frac{1+\vartheta\,\beta}{\beta-1+p}}\right\|^{\beta-1+p}_{L^{\beta-1+p}(B_r)}\leq C\,\left[\frac{|\delta_h u|^\frac{\beta-1}{p}\,(\delta_h u)}{|h|^\frac{1+\vartheta\,\beta}{p}}\eta\right]^p_{W^{s,p}(B(R))}+C\,\left(\int_{B_{R}}\left|\frac{\delta_h u}{|h|^\frac{1+\vartheta\beta}{\beta}}\right|^\frac{q\,\beta}{q-p+2} dx+1\right),
\end{equation}
with $C=C(N,h_0,s,\beta)>0$. This is the analogous of \cite[equation (4.15)]{bralinschi}.
We then choose $\xi=h$, take the supremum over $h$ for $0<|h|< h_0$ and integrate in time. Then \eqref{est4} together with \eqref{I2I3} imply
\begin{equation}\label{almostfinal}
\begin{split}
\int_{T_0+\mu}^{T_1}\sup_{0<|h|< h_0}&\int_{B_r}\left|\frac{\delta^2_h u}{|h|^\frac{1+s\,p+\vartheta\,\beta}{\beta-1+p}}\right|^{\beta-1+p}\,dx\,dt\\
&+\frac{1}{\beta+1}\sup_{0<|h|< h_0}\int_{B_R}\frac{|\delta_h u(x,T_1)|^{\beta+1}}{|h|^{1+\vartheta\beta}} \eta^p dx\\
&\leq C\,\int_{T_0}^{T_1}\left(\sup_{0<|h|< h_0}\int_{B_R}\left|\frac{\delta_h u}{|h|^\frac{1+\vartheta\,\beta}{\beta}}\right|^\frac{q\,\beta}{q-p+2}\,dx +\sup_{0<|h|< h_0}\left\|\frac{\delta^2_h u}{|h|^s}\right\|_{L^q(B_{R+4h_0})}^q+ 1\right)dt,
\end{split}
\end{equation}
where $C=C(N,h_0,p,q,s,\beta,\mu)>0$.  
Since $(1+\vartheta\,\beta)/\beta< 1$, we can replace the first order difference quotients in the right-hand side of \eqref{almostfinal} with second order ones, just by using \cite[Lemma 2.6]{bralinschi}. This gives
\begin{equation}
\label{finalbeforeexp}
\begin{split}
&\int_{T_0+\mu}^{T_1}\sup_{0<|h|< h_0}\int_{B_r}\left|\frac{\delta^2_h u}{|h|^\frac{1+s\,p+\vartheta\,\beta}{\beta-1+p}}\right|^{\beta-1+p}\,dx\,dt+\frac{1}{\beta+1}\sup_{0<|h|< h_0}\int_{B_R}\frac{|\delta_h u(x,T_1)|^{\beta+1}}{|h|^{1+\vartheta\beta}} \eta^p dx\\
&\leq C\,\int_{T_0}^{T_1}\left(\sup_{0<|h|< h_0}\int_{B_R}\left|\frac{\delta^2_h u}{|h|^\frac{1+\vartheta\,\beta}{\beta}}\right|^\frac{q\,\beta}{q-p+2}\,dx +\sup_{0<|h|< h_0}\left\|\frac{\delta^2_h u}{|h|^s}\right\|_{L^q(B_{R+4h_0})}^q+ 1\right)dt,
\end{split}
\end{equation}
for some constant $C=C(N,h_0,p,q,s,\beta,\mu)>0$.
\vskip.2cm\noindent
{\bf Step 6: Conclusion for $T_1<0$.}
As in the final step of the step of \cite[Proposition 4.1]{bralinschi}, we now fix
$$
\beta=q-p+2\qquad \mbox{ and }\qquad \vartheta=\frac{(q-p+2)\,s-1}{q-p+2},
$$
where $q\ge p$ is as in the statement. These choices assure that
\[
\frac{1+s\,p+\vartheta\,\beta}{\beta-1+p}=\frac{s}{q+1}+s,
\]
\[
\beta-1+p=q+1,\qquad \frac{q\,\beta}{q-p+2}=q,
\]
and
\[
1+\vartheta\,\beta=(q-p+2)\,s,\qquad\frac{1+\vartheta\,\beta}{\beta}=s.
\]
Then \eqref{finalbeforeexp} becomes
\[
\begin{split}
\int_{T_0+\mu}^{T_1}\sup_{0<|h|< h_0}\left\|\frac{\delta^2_h u}{|h|^{\frac{s}{q+1}+s}}\right\|_{L^{q+1}(B_r)}^{q+1}dt&+\frac{1}{q+3-p}\sup_{0<|h|< h_0}\left\|\frac{\delta_h u(x,T_1)}{|h|^{\frac{(q+2-p)s}{q+3-p}}}\right\|_{L^{q+3-p}(B_r)}^{q+3-p}\\
&\leq C\,\int_{T_0}^{T_1}\left(\sup_{0<|h|< h_0}\left\|\frac{\delta^2_h u }{|h|^s}\right\|_{L^q(B_{R+4h_0})}^q+1\right)dt, 
\end{split}
\]
where $C=C(N,h_0,p,q,s)>0$. Up to a suitable modification of the constant $C$, we obtain in particular
\[
\begin{split}
\int_{T_0+\mu}^{T_1}\sup_{0<|h|< h_0}\left\|\frac{\delta^2_h u}{|h|^{s}}\right\|_{L^{q+1}(B_{R-4\,h_0})}^{q+1}dt &+\frac{1}{q+3-p}\sup_{0<|h|< h_0}\left\|\frac{\delta_h u(\cdot,T_1)}{|h|^{\frac{(q+2-p)\,s}{q+3-p}}}\right\|_{L^{q+3-p}(B_{R-4\,h_0})}^{q+3-p}\\
&\leq C\,\int_{T_0}^{T_1}\left(\sup_{0<|h|< h_0}\left\|\frac{\delta^2_h u }{|h|^s}\right\|_{L^q(B_{R+4h_0})}^q+1\right)dt,
\end{split}
\]
as desired. Observe that we used that $r=R-4\,h_0$.
\vskip.2cm\noindent
{\bf Step 7: Conclusion for $T_1=0$.} In this case, the previous proof does not directly work because it relies on Lemma \ref{Mlemreg}, which needed $T_1<0$. However, the constant $C$ in \eqref{iteralo} does not depend on $T_1$, we can thus use a limit argument. By assumption, we have that for some $q\ge p$ and $0<h_0<1/10$, it holds
\[
\int_{T_0}^{0}\sup_{0<|h|< h_0}\left\|\frac{\delta^2_h u }{|h|^s}\right\|_{L^q(B_{R+4\,h_0})}^q dt<+\infty,
\]
for a radius $4\,h_0<R\le 1-5\,h_0$ and a time instant $-1<T_0<0$. We fix $0<\mu<-T_0$, then for every $T<0$ such that $\mu+T_0<T$ we have
from {\bf Step 6}
\begin{equation}
\label{porconi}
\begin{split}
\int_{T_0+\mu}^{T}\sup_{0<|h|< h_0}\left\|\frac{\delta^2_h u}{|h|^{s}}\right\|_{L^{q+1}(B_{R-4\,h_0})}^{q+1}\,dt &+\frac{1}{q+3-p}\sup_{0<|h|< h_0}\left\|\frac{\delta_h u(\cdot,T)}{|h|^{\frac{(q+2-p)s}{q+3-p}}}\right\|_{L^{q+3-p}(B_{R-4\,h_0})}^{q+3-p}\\
&\leq C\,\int_{T_0}^{0}\left(\sup_{0<|h|< h_0}\left\|\frac{\delta^2_h u }{|h|^s}\right\|_{L^q(B_{R+4h_0})}^q+1\right)dt.
\end{split}
\end{equation}
We then observe that 
\begin{equation}
\label{porconi1}
\lim_{T\to 0^-}\int_{T_0+\mu}^{T}\sup_{0<|h|< h_0}\left\|\frac{\delta^2_h u}{|h|^{s}}\right\|_{L^{q+1}(B_{R-4\,h_0})}^{q+1}\,dt =\int_{T_0+\mu}^{0}\sup_{0<|h|< h_0}\left\|\frac{\delta^2_h u}{|h|^{s}}\right\|_{L^{q+1}(B_{R-4\,h_0})}^{q+1}\,dt,
\end{equation}
by the Dominated Convergence Theorem. As for the second term on the left-hand side, we know by definition of local weak solution that 
\[
t\mapsto \frac{\delta_h u(\cdot,t)}{|h|^{\frac{(q+2-p)\,s}{q+3-p}}},
\] 
is a continuous function on $(-2,0]$, with values in $L^2(B_{R-4\,h_0})$, for every fixed $0<|h|<h_0$. Thus 
\[
\lim_{T\to 0^-}\left\|\frac{\delta_h u(\cdot,T)}{|h|^{\frac{(q+2-p)s}{q+3-p}}}-\frac{\delta_h u(\cdot,0)}{|h|^{\frac{(q+2-p)s}{q+3-p}}}\right\|_{L^2(B_{R-4\,h_0})}=0.
\]
This in turn implies that\footnote{We use the following standard fact: if $\{f_n\}_{n\in\mathbb{N}}$ converges to $f$ in $L^\alpha(E)$, then
\[
\liminf_{n\to\infty} \|f_n\|_{L^\beta(E)}\ge \|f\|_{L^\beta(E)},
\]  
for any $\beta\not =\alpha$.} 
\begin{equation}
\label{porconi2}
\liminf_{T\to 0^-}\left\|\frac{\delta_h u(\cdot,T)}{|h|^{\frac{(q+2-p)s}{q+3-p}}}\right\|_{L^{q+3-p}(B_{R-4\,h_0})}^{q+3-p}\ge \left\|\frac{\delta_h u(\cdot,0)}{|h|^{\frac{(q+2-p)s}{q+3-p}}}\right\|_{L^{q+3-p}(B_{R-4\,h_0})}^{q+3-p},
\end{equation}
for every $0<|h|<h_0$. By using \eqref{porconi1} and \eqref{porconi2} in \eqref{porconi}, we get the desired conclusion for $T_1=0$, as well.
\end{proof}
As in \cite[Theorem 4.2]{bralinschi}, by iterating the previous result, we can obtain the following regularity estimate.
\begin{teo}[Spatial almost $C^s$ regularity]
\label{teo:1}
Let $\Omega\subset\mathbb{R}^N$ be a bounded and open set, $I=(t_0,t_1]$, $p\geq 2$ and $0<s<1$. Suppose $u$ is a local weak solution of 
\[
u_t+(-\Delta_p)^s u=0\qquad \mbox{ in }\Omega\times I,
\]
such that $u\in L^\infty_{\rm loc}(I;L^\infty(\mathbb{R}^N))$.
Then $u\in C^\delta_{x,\rm loc}(\Omega\times I)$ for every $0<\delta<s$. 
\par
More precisely, for every $0<\delta<s$, $R>0$ and every $(x_0,T_0)$ such that 
\[
Q_{2R,2R^{s\,p}}(x_0,T_0)\Subset\Omega\times (t_0,t_1],
\] 
there exists a constant $C=C(N,s,p,\delta)>0$ such that
\begin{equation}
\label{apriori1}
\begin{split}
\sup_{t\in \left[T_0-\frac{R^{s\,p}}{2},T_0\right]} [u(\cdot,t)]_{C^\delta(B_{R/2}(x_0))}&\leq \frac{C}{R^\delta}\,\left(\|u\|_{L^\infty(\mathbb{R}^N\times [T_0-R^{s\,p},T_0])}+1\right)\\
&+\frac{C}{R^\delta}\,\left(R^{-N}\,\int_{T_0-\frac{7}{8}\,R^{s\,p}}^{T_0} [u]^p_{W^{s,p}(B_{R}(x_0))}\,dt \right)^\frac{1}{p}
\end{split}
\end{equation}
\end{teo}

\begin{proof}  We assume for simplicity that $x_0=0$ and $T_0=0$, then we set
\[
\mathcal{M}_R = \|u\|_{L^\infty(\mathbb{R}^N\times [-R^{s\,p},0])} +\left(R^{-N}\,\int_{-\frac{7}{8}\,R^{s\,p}}^{0} [u]^p_{W^{s,p}(B_{R})}\,dt \right)^\frac{1}{p}+1.
\]
Let $\alpha\in [-R^{s\,p}(1-\mathcal{M}_R^{2-p}),0]$ and set
\[
u_{R,\alpha}(x,t):=\frac{1}{\mathcal{M}_R}\,u\left(R\,x,\frac{1}{\mathcal{M}_R^{p-2}}\,R^{s\,p}\,t+\alpha\right),\qquad \mbox{ for }x\in B_2,\ t\in (-2,0].
\]
By taking into account the scaling properties of our equation (see Remark \ref{oss:scalings}), the function $u_{R,\alpha}$ 
is a local weak solution of 
\[
u_t+(-\Delta_p)^s u=0,\qquad \mbox{ in } B_2\times (-2,0],
\]
 and satisfies
\begin{equation}
\label{assumption}
\|u_{R,\alpha}\|_{L^\infty(\mathbb{R}^N\times [-1,0])}\leq 1,\qquad  \int_{-\frac{7}{8}}^0 [u_{R,\alpha}]^p_{W^{s,p}(B_1)}\, dt \leq 1.
\end{equation}
We will prove that $u_{R,\alpha}$ 
satifies the estimate
\[
\sup_{t\in [-1/2,0]}[u_{R,\alpha}(\cdot,t)]_{C^{\delta}(B_{1/2})}\leq C,
\]
for $C=C(N,s,p,\delta)>0$ independent of $\alpha$. 
By scaling back, this would give
\[
\sup_{\alpha-\frac12\mathcal{M}_R^{2-p}\,R^{s\,p}\le t\le 0}[u(\cdot,t)]_{C^{\delta}(B_{R/2})}\leq \frac{C}{R^\delta}\mathcal{M}_R.
\] 
Since $\alpha\in [-R^{s\,p}(1-\mathcal{M}_R^{2-p}),0]$ and $\mathcal{M}_R^{2-p}\le 1$, this in turn would imply
\[
\sup_{t\in \left[-\frac{R^{s\,p}}{2},0\right]} [u(\cdot,t)]_{C^\delta(B_{R/2}(x_0))}\leq \frac{C}{R^\delta}\,\left(\|u\|_{L^\infty(\mathbb{R}^N\times [-R^{s\,p},0])} +\left(R^{-N}\,\int_{-R^{s\,p}}^0 [u]^p_{W^{s,p}(B_{R}(x_0))}\,dt \right)^\frac{1}{p}+1\right),
\]
which is the desired result.
In what follows, we suppress the subscript ${R,\alpha}$ and simply write $u$ in place of $u_{R,\alpha}$, in order not to overburden the presentation.
\vskip.2cm\noindent
We fix $0<\delta<s$ and choose $i_\infty\in\mathbb{N}\setminus\{0\}$ such that
\[
\delta<s\,\frac{2+i_\infty}{3+i_\infty}- \frac{N}{3+i_\infty}.
\]
Then we define the sequence of exponents  
\[
q_i=p+i,\qquad i=0,\dots,i_\infty.
\]
We define also 
$$
h_0=\frac{1}{64\,i_\infty},\qquad R_i=\frac{7}{8}-4\,(2i+1)\,h_0=\frac{7}{8}-\frac{2i+1}{16\,i_\infty},\qquad \mbox{ for } i=0,\dots,i_\infty.
$$
We note that 
\[
R_0+4\,h_0=\frac{7}{8}\qquad \mbox{ and }\qquad R_{i_\infty-1}-4\,h_0=\frac{3}{4}.
\] 
By applying Proposition \ref{prop:improve} (ignoring the second term in the left-hand side of \eqref{iteralo}) with\footnote{We observe that by construction we have
\[
4\,h_0<R_i\le 1-5\,h_0,\qquad \mbox{ for } i=0,\dots,i_\infty-1.
\]
Thus these choices are admissible in Proposition \ref{prop:improve}.} 
\[
T_1=0,\qquad T_0=-R_i-4\,h_0,\qquad \mu=8\,h_0,
\]
and
\[
R=R_i\qquad \mbox{ and }\qquad q=q_i=p+i,\qquad \mbox{ for } i=0,\ldots,i_\infty-1,
\] 
and observing that $R_i-4\,h_0=R_{i+1}+4\,h_0$,
we obtain the iterative scheme of inequalities:
\begin{itemize}
\item for $i=0$
\[
\int_{-(R_1+4h_0)}^0 \sup\limits_{0<|h|< h_0}\left\|\dfrac{\delta^2_h u}{|h|^{s}}\right\|_{L^{q_1}(B_{R_1+4h_0})}^{q_1}dt \leq C\,\displaystyle\int_{-\frac{7}{8}}^0 \sup\limits_{0<|h|< h_0}\left(\left\|\dfrac{\delta^2_h u }{|h|^s}\right\|_{L^p(B_{7/8})}^p+1\right)dt\\
\]
\item for $i=1,\dots,i_\infty-2$
\[
\int_{-(R_{i+1}+4h_0)}^0 \sup\limits_{0<|h|< h_0}\left\|\dfrac{\delta^2_h u}{|h|^{s}}\right\|_{L^{q_i+1}(B_{R_{i+1}+4h_0})}^{q_i+1}dt \leq C\,\int_{-(R_{i}+4h_0)}^0\sup\limits_{0<|h|< h_0}\left(\left\|\dfrac{\delta^2_h u }{|h|^s}\right\|_{L^{q_i}(B_{R_i+4h_0})}^{q_i}+1\right)dt,
\]
\item finally, for $i=i_\infty-1$
\[
\begin{split}
\int_{-\frac{3}{4}}^0\sup_{0<|h|< {h_0}}\left\|\frac{\delta^2_h u}{|h|^{s}}\right\|_{L^{q_{i_\infty}}(B_{3/4})}^{q_{i_\infty}}dt&=\int_{-(R_{i_\infty-1}-4h_0)}^0\sup_{0<|h|< h_0}\left\|\frac{\delta^2_h u}{|h|^{s}}\right\|_{L^{p+i_\infty}(B_{R_{i_\infty-1}-4h_0})}^{p+i_\infty}dt \\&\leq C\,\int_{-(R_{i_\infty-1}+4h_0)}^0\sup_{0<|h|< h_0}\left(\left\|\frac{\delta^2_h u }{|h|^s}\right\|_{L^{p+i_\infty-1}(B_{R_{i_\infty-1}+4h_0})}^{p+i_\infty-1}+1\right)\,dt.
\end{split}
\]
\end{itemize}
Here $C=C(N,\delta,p,s)>0$ as always. We note that by using the relation
\[
\delta^2_{h}u=\delta_{2h} u-2\,\delta_h u,
\]
and then appealing to \cite[Proposition 2.6]{Brolin},
we have
\begin{align}
\label{eq:1sttofrac}
\int_{-\frac{7}{8}}^0\sup_{0<|h|< h_0}\left\|\frac{\delta^2_h u }{|h|^s}\right\|_{L^{p}(B_{7/8})}^pdt&\leq C\,\int_{-\frac{7}{8}}^0\sup_{0<|h|<2\,h_0}\left\|\frac{\delta_h u }{|h|^s}\right\|_{L^{p}(B_{7/8})}^pdt\nonumber \\
&\leq C\left(\int_{-\frac{7}{8}}^0[u]_{W^{s,p}(B_{7/8+2\,h_0})}^p\, dt+\int_{-\frac{7}{8}}^0\|u\|_{L^\infty(B_{7/8+2\,h_0})}\,dt\right)\\
&\leq C\left(\int_{-\frac{7}{8}}^0[u]_{W^{s,p}(B_1)}\,dt+\int_{-\frac{7}{8}}^0\|u\|_{L^\infty(B_1)}^p\, dt\right)\leq C(N,\delta,s,p),\nonumber 
\end{align}
where we also have used the assumptions \eqref{assumption} on $u$. Hence, the iterative scheme of inequalities leads us to
\[
\int_{-\frac{3}{4}}^0\sup_{0<|h|< {h_0}}\left\|\frac{\delta^2_h u}{|h|^{s}}\right\|_{L^{q_{i_\infty}}(B_{3/4})}^{q_{i_\infty}}dt\leq C(N,\delta,p,s).
\]
It is now time to exploit the full power of Proposition \ref{prop:improve}: we apply it once more, with 
\[
T_0+\mu=-3/4,\qquad -\frac{1}{2}\le T_1\le0,
\]
\[
q=q_{i_\infty},\qquad R+4\,h_0=6/8\qquad  \mbox{ and }\qquad R-4\,h_0=6/8-8\,h_0>5/8.
\] 
We obtain (ignoring the first term in the left-hand side of \eqref{iteralo}, this time)
$$
\sup_{0<|h|< h_0}\left\|\frac{\delta_h u(\cdot,T_1)}{|h|^{\frac{(q_{i_\infty}+2-p)\,s}{q_{i_\infty}+3-p}}}\right\|_{L^{q_{i_\infty}+3-p}(B_{5/8})}^{q_{i_\infty}+3-p}\leq \int_{-3/4}^0\sup_{0<|h|< {h_0}}\left\|\frac{\delta^2_h u}{|h|^{s}}\right\|_{L^{q_{i_\infty}}(B_{3/4})}^{q_{i_\infty}}\,dt\leq C(N,\delta,p,s).
$$
Since this is valid for every $-1/2\le T_1\le 0$, this in turn implies that 
\begin{equation}
\label{otherest1}
\sup_{t\in [-1/2,0]}\sup_{0<|h|<h_0} \left\|\frac{\delta_h u}{|h|^{\frac{(q_{i_\infty}+2-p)\,s}{q_{i_\infty}+3-p}}}\right\|_{L^{q_{i_\infty}+3-p}(B_{5/8})}^{q_{i_\infty}+3-p}\le C(N,\delta,p,s).
\end{equation}
Take now $\chi\in C_0^\infty(B_{9/16})$ such that 
$$
0\leq \chi\leq 1, \qquad \chi =1 \text{ in $B_{1/2}$},\qquad |\nabla \chi|\leq C,\qquad |D^2 \chi|\leq C.
$$
In particular, we have for all $|h| > 0$
$$
\frac{|\delta_h\chi|}{|h|^\frac{(q_{i_\infty}+2-p)\,s}{q_{i_\infty}+3-p}}\leq C.
$$
We also recall that
$$
\delta_h (u\,\chi)=\chi_{h}\,\delta_h u+u\,\delta_h\chi.
$$
Hence, for $0<|h|< h_0$ and any $t\in [-5/8,0]$
\begin{align}\label{Neq}
\left\|\frac{\delta_h (u\,\chi)}{|h|^\frac{(q_{i_\infty}+2-p)\,s}{q_{i_\infty}+3-p}}\right\|_{L^{q_{i_\infty+3-p}}(\mathbb{R}^N)}&\leq C\,\left(\left\|\frac{\chi_{h}\,\delta_h u}{|h|^\frac{(q_{i_\infty}+2-p)s}{q_{i_\infty}+3-p}}\right\|_{L^{q_{i_\infty+3-p}}(\mathbb{R}^N)}+\left\|\frac{u\,\delta_h\chi}{|h|^\frac{(q_{i_\infty}+2-p)s}{q_{i_\infty}+3-p}}\right\|_{L^{q_{i_\infty+3-p}}(\mathbb{R}^N)}\right) \nonumber\\
&\leq C\,\left(\left\|\frac{\delta_h u}{|h|^\frac{(q_{i_\infty}+2-p)s}{q_{i_\infty}+3-p}}\right\|_{L^{q_{i_\infty+3-p}}(B_{9/16+\,h_0})}+\|u\|_{L^{q_{i_\infty+3-p}}(B_{9/16+h_0})}\right) \\
&\leq C\,\left(\left\|\frac{\delta_h u}{|h|^\frac{(q_{i_\infty}+2-p)s}{q_{i_\infty}+3-p}}\right\|_{L^{q_{i_\infty+3-p}}(B_{5/8})}+\|u\|_{L^{q_{i_\infty+3-p}}(B_{5/8})}\right)\leq C(N,\delta,p,s), \nonumber
\end{align}
by \eqref{otherest1}.
Finally, by noting that thanks to the choice of $i_\infty$ we have
\[
s\,(q_{i_\infty}+2-p)>N\qquad \mbox{and }\qquad \delta<\frac{(q_{i_\infty}+2-p)\,s}{q_{i_\infty}+3-p}-\frac{N}{q_{i_\infty}+3-p},
\] 
we may invoke the Morrey-type embedding of \cite[Theorem 2.8]{bralinschi} with 
\[
\beta=\frac{(q_{i_\infty}+2-p)\,s}{q_{i_\infty}+3-p},\qquad \alpha=\delta \qquad \mbox{ and }\qquad q=q_{i_\infty}+3-p.
\]
Thus we obtain
$$
[u(\cdot,t)]_{C^\delta(B_{1/2})}= [u\,\chi]_{C^\delta(B_{1/2})}\leq C\left([u\,\chi(\cdot,t)]_{\mathcal{N}_\infty^{\beta,q}(\mathbb{R}^N)}\right)^{\frac{\alpha\, q+N}{\beta q}}\,\left(\|u(\cdot,t)\,\chi\|_{L^q(\mathbb{R}^N)}\right)^\frac{(\beta-\alpha)\,q-N}{\beta q}\leq C(N,\delta,p,s),
$$
for any $t\in [-1/2,0]$, where we used \eqref{Neq}. This concludes the proof.
\end{proof}
\begin{oss}
\label{oss:flexibility}
Under the assumptions of the previous theorem, a covering argument combined with \eqref{apriori} implies the more flexible estimate: for every $0<\sigma<7/8$
\[
\begin{split}
\sup_{t\in[T_0-\sigma R^{s\,p},T_0]}[u(\cdot,t)]_{C^\delta(B_{\sigma R}(x_0))}&\leq \frac{C}{R^\delta}\,\left(\|u\|_{L^\infty(B_{R}(x_0)\times[T_0-R^{s\,p},T_0])}+1\right)\\
&+\frac{C}{R^\delta}\,\left(R^{-N}\int_{T_0-\frac{7}{8}\,R^{s\,p}}^{T_0}[u]_{W^{s,p}(B_{R}(x_0))}^pdt\right)^{\frac{1}{p}},
\end{split}
\]
with $C$ now depending on $\sigma$ as well (and blowing-up as $\sigma\nearrow 7/8$). Indeed, if $\sigma\le 1/2$ then this is immediate.
If $1/2<\sigma<7/8$, then we can cover $Q_{\sigma R,\sigma R^{s\,p}}(x_0,T_0)$ with a finite number of cylinders 
\[
Q_{r/2,r^{s\,p}/2}(x_i,t_j)=B_{r/2}(x_i)\times\left(t_j-\frac{r^{s\,p}}{2},t_j\right],\qquad \mbox{ for } 1\le i\le k, \, 1\le j\le m,
\]
where
\[
x_i\in B_{\sigma R}(x_0),\qquad T_0-\sigma\,R^{s\,p}\le t_j\le T_0,
\]
and $r=R/C_{\sigma,s,p}>0$ is a suitable radius, such that 
\[
B_r(x_i)\subset B_R(x_0),\quad B_{2\,r}(x_i)\Subset \Omega,
\]
and
\[
[t_j-r^{s\,p},t_j]\subset \left[T_0-\frac{7}{8}\,R^{s\,p},T_0\right],\quad [t_j-2\,r^{s\,p},t_j]\Subset I. 
\]
By using \eqref{apriori1} on each of these cylinders, we get
\[
\begin{split}
\sup_{t\in\left[t_j-\frac{r^{s\,p}}{2},t_j\right]}[u(\cdot,t)]_{C^\delta(B_{r/2}(x_i))}&\leq \frac{C}{r^\delta}\,\left(\|u\|_{L^\infty(\mathbb{R}^N\times[t_j-r^{s\,p},t_j])}+\left(r^{-N}\int_{t_j-r^{s\,p}}^{t_j}[u]_{W^{s,p}(B_{r}(x_i))}^p\,dt\right)^{\frac{1}{p}}+1\right).
\end{split}
\]
By taking the supremum over $1\le i\le k$ and $1\le j\le m$, we get the desired conclusion. 
\end{oss}

\section{Improved spatial H\"older regularity}
\label{sec:higher}
Once we know that solutions are locally spatially $\delta-$H\"older continuous for any $0<\delta<s$, we can obtain the following improvement of Proposition \ref{prop:improve}. The latter provided a recursive gain of integrability. In contrast, the next result provides a gain which is interlinked between differentiability and integrability.
\begin{prop}\label{prop:improve2}
Assume $p\ge 2$ and $0<s<1$. Let  $u$ be a local weak solution of $u_t+(-\Delta_p)^s u=0$ in $B_2\times (-2,0]$, such that
\[
\|u\|_{L^\infty(\mathbb{R}^N\times [-1,0])}\leq 1. 
\]
Assume further that for some $0<h_0<1/10$ and $\vartheta<1$, $\beta\ge 2$ such that $(1+\vartheta\, \beta)/\beta<1$, we have
\[
\int_{T_0}^{T_1}\sup_{0<|h|\leq h_0}\left\|\frac{\delta^2_h u }{|h|^\frac{1+\vartheta\, \beta}{\beta}}\right\|_{L^\beta(B_{R+4\,h_0})}^\beta dt<+\infty.
\]
for a radius $4\,h_0<R\le 1-5\,h_0$ and two time instants $-1<T_0< T_1\le 0$. Then it holds
\begin{equation}
\label{iteralo2}
\begin{split}
\int_{T_0+\mu}^{T_1}\sup_{0<|h|< h_0}\left\|\frac{\delta^2_h u}{|h|^{\frac{1+s\,p+\vartheta\,\beta}{\beta-1+p}}}\right\|_{L^{\beta-1+p}(B_{R-4\,h_0})}^{\beta-1+p}dt&+\frac{1}{\beta+1}\sup_{0<|h|< h_0}\left\|\frac{\delta_h u(\cdot,T_1)}{|h|^{\frac{1+\vartheta\,\beta}{\beta+1}}}\right\|_{L^{\beta+1}(B_{R-4\,h_0})}^{\beta+1}\\
&\leq C\,\int_{T_0}^{T_1}\sup_{0<|h|< h_0}\left(\left\|\frac{\delta^2_h u }{|h|^\frac{1+\vartheta\, \beta}{\beta}}\right\|_{L^\beta(B_{R+4\,h_0})}^\beta+1\right) dt.
\end{split}
\end{equation}
for every $0<\mu<T_1-T_0$.
Here $C$ depends on the $N$, $h_0$, $s$, $p$, $\mu$ and $\beta$.

\end{prop}
\begin{proof} 
This is analogous to the proof of \cite[Proposition 5.1]{bralinschi}. As above, we will refer to \cite{bralinschi} for the main computations and only list the major changes.
\par
We first notice that it sufficient to prove \eqref{iteralo2} for $T_1<0$, with a constant independent of $T_1$. Then the same argument of {\bf Step 7} in Proposition \ref{prop:improve} will be enough to handle the case $T_1=0$, as well.
\par
We  go back to the estimates in the proof of Proposition \ref{prop:improve}. The acquired knowledge on the spatial regularity of $u$ permits to improve the estimate on the term $\mathcal{I}_{11}(t)$ defined in \eqref{I11}. From Theorem~\ref{teo:1} and Remark \ref{oss:flexibility}, we can choose 
\[
0<\varepsilon<\min\left\{2\,\frac{1-s}{p-2},\,s\right\},
\] 
such that 
$$
\sup_{t\in [T_0,T_1]}[u(\cdot,t)]_{C^{s-\varepsilon}(B_{R+h_0})}\leq C(N,h_0,p,s).
$$
Using this together with the assumed regularity of $\eta$, we have for $(x,y)\in B_{R+h_0}$ and $t\in [T_0,T_1]$
$$
\frac{|u(x,t)-u(y,t)|^{p-2}\,\left|\eta(x)^\frac{p}{2}-\eta(y)^\frac{p}{2}\right|^2}{|x-y|^{N+s\,p}}\leq C\,|x-y|^{-N+2\,(1-s)-\varepsilon\,(p-2)}.
$$
As usual, we are suppressing the time dependence. Thanks to the choice of $\varepsilon$, the last exponent is strictly larger than $-N$ and we may conclude
$$
\int_{B_R}\frac{|u(x,t)-u(y,t)|^{p-2}\,\left|\eta(x)^\frac{p}{2}-\eta(y)^\frac{p}{2}\right|^2}{|x-y|^{N+s\,p}}\, dy\leq C(N,h_0,p,s),
$$
for any $x\in B_R$. Therefore, by suppressing as before the $t-$dependence for simplicity, we have the estimate
\[
\begin{split}
| \mathcal{I}_{11}(t)|&\leq C\,\int_{B_R}\frac{|\delta_h u(x)|^{\beta+1}}{|h|^{1+\vartheta\,\beta}} dx \\
&\leq C\,\|u\|_{L^\infty(B_R)}\,\int_{B_R}\frac{|\delta_h u(x)|^{\beta}}{|h|^{1+\vartheta\,\beta}} dx\le C\,\int_{B_R}\frac{|\delta_h u(x)|^{\beta}}{|h|^{1+\vartheta\,\beta}} dx ,\qquad \mbox{ for some } C=C(N,h_0,p,s)>0.
\end{split}
\]
As for $\mathcal{I}_{12}$, by going back to its definition \eqref{I12} and using the properties of the cut-off function $\eta$, we get
$$
|\mathcal{I}_{12}(t)|\leq C\int_{B_R} \frac{|\delta_h u(x)|^{\beta-1+p}}{|h|^{1+\vartheta\,\beta}}\,dx\leq C\int_{B_R}\frac{|\delta_h u(x)|^{\beta}}{|h|^{1+\vartheta\,\beta}} dx,\qquad \mbox{ for some }C=C(N,h_0,p,s)>0,
$$
where we used the local $L^\infty$ bound on $u$, as above.
In addition, from the first inequality in \eqref{I3I4est} together with the properties of the cut-off function $\tau$, we have
$$
|\mathcal{\widetilde I}_2|+|\mathcal{\widetilde I}_3|\leq C\,\int_{T_0}^{T_1}\int_{B_R}\frac{|\delta_h u(x)|^{\beta}}{|h|^{1+\vartheta\,\beta}}\, dx\, dt, \qquad C=C(h_0,p,s).
$$
Combining these new estimates with \eqref{I4est2} and \eqref{Ieq}, we can reproduce the last part of \cite[Proposition 5.1]{bralinschi} and arrive at
\[
\begin{split}
\int_{T_0+\mu}^{T_1}\sup_{0<|h|< h_0}\int_{B_r}\left|\frac{\delta^2_h u}{|h|^\frac{1+s\,p+\vartheta\,\beta}{\beta-1+p}}\right|^{\beta-1+p}\,dx\,dt&+\frac{1}{\beta+1}\sup_{0<|h|< h_0}\left\|\frac{\delta_h u(\cdot,T_1)}{|h|^{\frac{1+\vartheta\beta}{\beta+1}}}\right\|_{L^{\beta+1}(B_{R-4\,h_0})}^{\beta+1}\\&\leq C\int_{T_0}^{T_1}\sup_{0<|h|< h_0}\int_{B_R}\left|\frac{\delta_h u}{|h|^\frac{1+\vartheta\,\beta}{\beta}}\right|^\beta\,dx\,dt,
\end{split}
\]
for some $C=C(N,h_0,p,s,\beta)>0$. By appealing again to \cite[Lemma 2.6]{bralinschi} and using that
\[
\frac{1+\vartheta\,\beta}{\beta}<1,
\] 
we may replace the first order differential quotients in the right-hand side by second order ones. This leads to
\[
\begin{split}
\int_{T_0+\mu}^{T_1}\sup_{0<|h|< h_0}\int_{B_r}\left|\frac{\delta^2_h u}{|h|^\frac{1+s\,p+\vartheta\,\beta}{\beta-1+p}}\right|^{\beta-1+p}\,dx\,dt&+\frac{1}{\beta+1}\sup_{0<|h|< h_0}\left\|\frac{\delta_h u(\cdot,T_1)}{|h|^{\frac{1+\vartheta\beta}{\beta+1}}}\right\|_{L^{\beta+1}(B_{R-4\,h_0})}^{\beta+1}\\&\leq C\int_{T_0}^{T_1}\left(\sup_{0<|h|< h_0}\int_{B_R+4\,h_0}\left|\frac{\delta^2_h u}{|h|^\frac{1+\vartheta\,\beta}{\beta}}\right|^\beta\,dx+1\right)\,dt,
\end{split}
\]
for some $C=C(N,h_0,p,s,\beta)>0$.
By recalling again that $r=R-4\,h_0$, we eventually conclude the proof.
\end{proof}

We are now ready to prove the claimed H\"older regularity in space.
\begin{teo}
\label{teo:1higher} Let $\Omega$ be a bounded and open set, let $I=(t_0,t_1]$, $p\geq 2$ and $0<s<1$. Suppose $u$ is a local weak solution of 
\[
u_t+(-\Delta_p)^s u=0\qquad \mbox{ in }\Omega\times I,
\]
such that $u\in L^\infty_{\rm loc}(I;L^\infty(\mathbb{R}^N))$.
Then $u\in C^\delta_{x,\rm loc}(\Omega\times I)$ for every $0<\delta<\Theta(s,p)$, where $\Theta(s,p)$ is defined in \eqref{exponents}.
\par
More precisely, for every $0<\delta<\Theta(s,p)$, $R>0$, $x_0\in\Omega$ and $T_0$ such that 
\[
B_{2R}(x_0)\times [T_0-2\,R^{s\,p},T_0]\Subset\Omega\times (t_0,t_1],
\] 
there exists a constant $C=C(N,s,p,\delta)>0$ such that
\begin{equation}
\label{apriori}
\begin{split}
\sup_{t\in \left[T_0-\frac{R^{s\,p}}{2},T_0\right]} [u(\cdot,t)]_{C^\delta(B_{R/2}(x_0))}&\leq \frac{C}{R^\delta}\,\left(\|u\|_{L^\infty(\mathbb{R}^N\times [T_0-R^{s\,p},T_0])} + 1\right)\\
&+\frac{C}{R^\delta}\,\left(R^{-N}\,\int_{T_0-\frac{7}{8}\,R^{s\,p}}^{T_0} [u]^p_{W^{s,p}(B_{R}(x_0))}dt \right)^\frac{1}{p} .
\end{split}
\end{equation}
\end{teo}
\begin{proof}By the same scaling argument as in the proof of Theorem~\ref{teo:1}, it is enough to prove that
$$
\sup_{t\in [-1/2,0]}[u(\cdot,t)]_{C^\delta(B_{1/2})}\leq C(N,p,s,\delta),
$$
under the assumption that $u$ is a local weak solution of 
\[
u_t+(-\Delta_p)^s u=0,\qquad \mbox{ in } B_2\times (-2,0],
\]
which satisfies \eqref{assumption}. Define for $i\in\mathbb{N}$, the sequences of exponents
$$
\beta_i=p+i\,(p-1),
$$
and
$$
\vartheta_0=s-\frac{1}{p},\qquad \vartheta_{i+1}=\frac{\vartheta_i\,\beta_i+s\,p}{\beta_{i+1}}=\vartheta_i\,\frac{p+i\,(p-1)}{p+(i+1)(p-1)}+\frac{s\,p}{p+(i+1)(p-1)}.
$$
By induction, we see that $\{\vartheta_i\}_{i\in\mathbb{N}}$ is explicitely given by the increasing sequence
$$
 \vartheta_i =\left(s-\frac{1}{p}\right)\,\frac{p}{p+i\,(p-1)}+\frac{s\,p\,i}{p+i\,(p-1)},\qquad i\in\mathbb{N},
$$
and thus
$$
\lim_{i\to\infty} \vartheta_i = \frac{s\,p}{p-1}.
$$
The proof is now split into two different cases.
\vskip.2cm\noindent
{\bf  Case 1: $s\,p\leq (p-1)$.} 
Fix $0<\delta<s\,p/(p-1)$ and choose $i_\infty\in\mathbb{N}\setminus\{0\}$ such that
\[
\delta<\frac{1+\vartheta_{i_\infty}\,\beta_{i_\infty}}{\beta_{i_\infty}+1}-\frac{N}{\beta_{i_\infty}+1}.
\]
This is feasible, since 
\[
\lim_{i\to\infty} \beta_i=+\infty,\qquad \lim_{i\to\infty}\vartheta_i=\frac{s\,p}{p-1}\qquad \mbox{ and } \qquad \delta< \frac{s\,p}{p-1}.
\]
Define also 
$$
h_0=\frac{1}{64\,i_\infty},\qquad R_i=\frac{7}{8}-4\,(2\,i+1)\,h_0=\frac{7}{8}-\frac{2\,i+1}{16\,i_\infty},\qquad \mbox{ for } i=0,\dots,i_\infty.
$$
We note that 
\[
R_0+4\,h_0=\frac{7}{8}\qquad \mbox{ and }\qquad R_{i_\infty-1}-4\,h_0=\frac{3}{4}.
\] 
By applying\footnote{Note that in this case we will always have $1+\vartheta_i\beta_i<\beta_i$, so that the proposition applies.} Proposition \ref{prop:improve2} (ignoring the second term of the left-hand side of \eqref{iteralo2}) with
\[
T_1=0,\qquad T_0=-R_i-4\,h_0,\qquad \mu=8\,h_0,
\]
and
\[
R=R_i, \qquad \vartheta=\vartheta_i \qquad \mbox{ and }\qquad \beta=\beta_i,\quad \mbox{ for } i=0,\ldots,i_\infty-1,
\] 
and observing that $R_i-4\,h_0=R_{i+1}+4\,h_0$ and that by construction
$$
\frac{1+s\,p+\vartheta_i\,\beta_i}{\beta_i+(p-1)}=\frac{1+\vartheta_{i+1}\,\beta_{i+1}}{\beta_{i+1}},
$$
we obtain the iterative scheme of inequalities:
\begin{itemize}
\item for $i=0$
\[
\int_{-(R_1+4h_0)}^0\sup\limits_{0<|h|< h_0}\left\|\dfrac{\delta^2_h u}{|h|^{\frac{1+\vartheta_1\beta_1}{\beta_1}}}\right\|_{L^{\beta_1}(B_{R_1+4h_0})}^{\beta_1}dt \leq C\,\displaystyle\int_{-\frac{7}{8}}^0\sup\limits_{0<|h|< h_0}\left(\left\|\dfrac{\delta^2_h u }{|h|^s}\right\|_{L^{p}(B_{7/8})}^p+1\right)dt; 
\]
\item for $i=1,\ldots,i_\infty-2$
\[
\begin{split}
\int_{-(R_{i+1}+4h_0)}^0&\sup\limits_{0<|h|< h_0}\left\|\dfrac{\delta^2_h u}{|h|^{\frac{1+\vartheta_{i+1}\beta_{i+1}}{\beta_{i+1}}}}\right\|_{L^{\beta_{i+1}}(B_{R_{i+1}+4h_0})}^{\beta_{i+1}}dt\\
 &\le C\,\displaystyle\int_{-(R_{i}+4h_0)}^0\sup\limits_{0<|h|< h_0}\left(\left\|\dfrac{\delta^2_h u }{|h|^{\frac{1+\vartheta_{i}\beta_{i}}{\beta_{i}}}}\right\|_{L^{\beta_i}(B_{R_i+4\,h_0})}^{\beta_i}+1\right)dt ;
\end{split}
\]
\item finally, for $i=i_\infty-1$
\[
\begin{split}
\displaystyle\int_{-\frac{3}{4}}^0&\sup_{0<|h|< {h_0}}\left\|\frac{\delta^2_h u}{|h|^{\frac{1}{\beta_{i_\infty}}+\vartheta_{i_\infty}}}\right\|_{L^{\beta_{i_\infty}}(B_{3/4})}^{\beta_{i_\infty}}dt\\
& \leq \displaystyle C\int_{-(R_{i_\infty-1}+4h_0)}^0\sup_{0<|h|< h_0}\left(\left\|\frac{\delta^2_h u }{|h|^{\frac{1+\vartheta_{i_\infty-1}\beta_{i_\infty-1}}{\beta_{i_\infty-1}}}}\right\|_{L^{\beta_{i_\infty-1}}(B_{R_{i_\infty-1}+4\,h_0})}^{\beta_{i_\infty-1}}+1\right)dt.
\end{split}
\]
\end{itemize}
Here $C=C(N,p,s,\delta)>0$ as always. As in \eqref{eq:1sttofrac} we have
$$
\displaystyle\int_{-\frac{7}{8}}^0 \sup_{0<|h|< h_0}\left\|\frac{\delta^2_h u }{|h|^s}\right\|_{L^{p}(B_{7/8})}^pdt\leq C(N,\delta,s,p).\nonumber 
$$
Hence, the previous iterative scheme of inequalities implies
\[
\int_{-\frac{3}{4}}^0\sup_{0<|h|< {h_0}}\left\|\frac{\delta^2_h u}{|h|^{\frac{1}{\beta_{i_\infty}}+\vartheta_{i_\infty}}}\right\|_{L^{\beta_{i_\infty}}(B_{3/4})}^{\beta_{i_\infty}} dt\leq C(N,\delta,p,s).
\]
Now we apply Proposition \ref{prop:improve2} once more, this time with 
\[
T_0+\mu=-3/4,\qquad -\frac{1}{2}\le T_1\le 0,
\]
\[
\beta=\beta_{i_\infty},\qquad \theta=\theta_{i_\infty},\qquad R+4\,h_0=\frac{6}{8}\qquad\mbox{ and }\qquad R-4\,h_0=6/8-8\,h_0>5/8.
\] 
We obtain (now ignoring the first term in the left-hand side of \eqref{iteralo2})
\[
\sup_{0<|h|< h_0}\left\|\frac{\delta_h u(\cdot,T_1)}{|h|^{\frac{1+\vartheta_{i_\infty}\beta_{i_\infty}}{\beta_{i_\infty}+1}}}\right\|_{L^{\beta_{i_\infty}+1}(B_{5/8})}^{\beta_{i_\infty}+1}\le \int_{-\frac{3}{4}}^0\sup_{0<|h|< {h_0}}\left\|\frac{\delta^2_h u}{|h|^{\frac{1}{\beta_{i_\infty}}+\vartheta_{i_\infty}}}\right\|_{L^{\beta_{i_\infty}}(B_{3/4})}^{\beta_{i_\infty}} dt\leq C(N,\delta,p,s).
\]
Since this is valid for every $-1/2\le T_1\le 0$, we obtain
$$
\sup_{t\in [-1/2,0]}\sup_{0<|h|< h_0}\left\|\frac{\delta_h u}{|h|^{\frac{1+\vartheta_{i_\infty}\beta_{i_\infty}}{\beta_{i_\infty}+1}}}\right\|_{L^{\beta_{i_\infty}+1}(B_{5/8})}^{\beta_{i_\infty}+1}\leq C(N,\delta,p,s).
$$
From here, we may repeat the arguments at the end of the proof of Theorem~\ref{teo:1} (see \eqref{Neq}) and use the Morrey--type embedding of \cite[Theorem 2.8]{bralinschi}, with 
\[
\beta = \frac{1+\vartheta_{i_\infty}\,\beta_{i_\infty}}{\beta_{i_\infty}+1},\qquad q=\beta_{i_\infty}+1\qquad \mbox{ and }\qquad \alpha =\delta,
\]
to obtain 
$$
\sup_{t\in [-1/2,0]}[u(\cdot,t)]_{C^\delta(B_{1/2})}\leq C(N,\delta,p,s),
$$
which concludes the proof in this case.
\vskip.2cm\noindent
{\bf Case 2: $s\,p> (p-1)$.} Fix $0<\delta<1$. Let $i_\infty\in\mathbb{N}\setminus\{0\}$ be such that 
$$
\frac{1+\vartheta_{i_\infty-1}\,\beta_{i_\infty-1}}{\beta_{i_\infty-1}}< 1\qquad \mbox{ and }\qquad \frac{1+\vartheta_{i_\infty}\,\beta_{i_\infty}}{\beta_{i_\infty}}\ge 1.
$$
Observe that such a choice is feasible, since 
\[
\lim_{i\to\infty} \frac{1+\vartheta_i\,\beta_i}{\beta_i}=\frac{s\,p}{p-1}>1.
\] 
Now choose $j_\infty$ so that 
$$
\delta<\frac{\beta_{i_\infty+j_\infty}}{\beta_{i_\infty+j_\infty}+1}-\frac{N}{\beta_{i_\infty+j_\infty}+1},
$$
and let 
\[
\gamma=1-\varepsilon,\qquad \mbox{ for some } 0<\varepsilon<1 \mbox{ such that } \delta<(1-\varepsilon)\,\frac{\beta_{i_\infty+j_\infty}}{\beta_{i_\infty+j_\infty}+1}-\frac{N}{\beta_{i_\infty+j_\infty}+1}.
\] 
Define also 
$$
h_0=\frac{1}{64\,(i_\infty+j_\infty)},\qquad R_i=\frac{7}{8}-4\,(2\,i+1)\,h_0=\frac{7}{8}-\frac{2\,i+1}{16\,(i_\infty+j_\infty)},\qquad \mbox{ for } i=0,\dots,i_\infty+j_\infty.
$$
We note that 
\[
R_0+4\,h_0=\frac{7}{8}\qquad \mbox{ and }\qquad R_{(i_\infty+j_\infty)-1}-4\,h_0=\frac{3}{4}.
\] 
By applying\footnote{Note that for $i\leq i_\infty-1$ we have $1+\vartheta_i\,\beta_i<\beta_i$, so that the proposition applies.}  Proposition \ref{prop:improve2} with
\[
T_1=0,\qquad T_0=-R_i-4\,h_0,\qquad \mu=8\,h_0,
\]
and
\[
R=R_i, \qquad \vartheta=\vartheta_i \qquad \mbox{ and }\qquad \beta=\beta_i,\quad \mbox{ for } i=0,\ldots,i_\infty-1,
\] 
and observing that $R_i-4\,h_0=R_{i+1}+4\,h_0$ and that
$$
\frac{1+s\,p+\vartheta_i\,\beta_i}{\beta_i+(p-1)}=\frac{1+\vartheta_{i+1}\,\beta_{i+1}}{\beta_{i+1}},
$$
we arrive as in {\bf Case 1} at
\[
\begin{split}
\int_{-(R_{i_\infty}+4h_0)}^0&\sup_{0<|h|< {h_0}}\left\|\frac{\delta^2_h u}{|h|^{\gamma}}\right\|_{L^{\beta_{i_\infty}}(B_{R_{i_\infty}+4h_0})}^{\beta_{i_\infty}} dt\\
&\leq\int_{-(R_{i_\infty}+4h_0)}^0 \sup_{0<|h|< {h_0}}\left\|\frac{\delta^2_h u}{|h|^{\frac{1}{\beta_{i_\infty}}+\vartheta_{i_\infty}}}\right\|_{L^{\beta_{i_\infty}}(B_{R_{i_\infty}+4h_0})}^{\beta_{i_\infty}}dt\leq C(N,\delta,p,s),
\end{split}
\]
since $\gamma<1\le 1/\beta_{i_\infty}+\vartheta_{i_\infty}$. We now apply Proposition \ref{prop:improve2} with
\[
R=R_i, \qquad \beta=\beta_i\qquad \mbox{ and }\qquad \vartheta=\widetilde \vartheta_i=\gamma-\frac{1}{\beta_i} \qquad  \mbox{ for } i=i_\infty,\ldots,i_\infty+j_\infty-1.
\] 
Observe that by construction we have
$$
\frac{1+\widetilde \vartheta_i\,\beta_i}{\beta_i}=\gamma,\qquad \mbox{ for } i=i_\infty,\ldots,i_\infty+j_\infty-1,
$$
and using that $s\,p>(p-1)$
$$
\frac{1+s\,p+\widetilde \vartheta_i\, \beta_i}{\beta_i+p-1}>\frac{p+\widetilde \vartheta_i\, \beta_i}{\beta_i+p-1}=1+\frac{\beta_i\,(\gamma-1)}{\beta_i+p-1}>\gamma,\quad \mbox{ for } i=i_\infty,\ldots,i_\infty+j_\infty-1.
$$
This gives the following inequalities:
\begin{itemize}
\item for $i=i_\infty,\ldots,i_\infty+j_\infty-2$
$$
\int_{-(R_{i+1}+4h_0)}^0\sup\limits_{|h|\leq h_0}\left\|\dfrac{\delta^2_h u}{|h|^{\gamma}}\right\|_{L^{\beta_{i+1}}(B_{R_{i+1}+4h_0})}^{\beta_{i+1}}dt\leq  C\,\int_{-(R_{i}+4h_0)}^0\sup\limits_{0<|h|< h_0}\left(\left\|\dfrac{\delta^2_h u }{|h|^{\gamma}}\right\|_{L^{\beta_i}(B_{R_i+4h_0})}^{\beta_i}+1\right) dt,
$$
\item for $i=j_\infty-1$
\[
\begin{split}
\int_{-3/4}^0&\sup_{0<|h|< {h_0}}\left\|\frac{\delta^2_h u}{|h|^{\gamma}}\right\|_{L^{\beta_{i_\infty+j_\infty}}(B_{3/4})}^{\beta_{i_\infty+j_\infty}}dt\\
&  \leq C\int_{-(R_{i_\infty+j_\infty}+4h_0)}^0\sup_{0<|h|< h_0}\left(\left\|\frac{\delta^2_h u }{|h|^{\gamma}}\right\|_{L^{\beta_{i_\infty+j_\infty-1}}(B_{R_{i_\infty+j_\infty-1}+4h_0})}^{\beta_{i_\infty+j_\infty-1}}+1\right)dt.
\end{split}
\]
\end{itemize}
Hence, recalling that $\gamma=1-\varepsilon$, we conclude
$$
\int_{-3/4}^0\sup_{0<|h|< {h_0}}\left\|\frac{\delta^2_h u}{|h|^{1-\varepsilon}}\right\|_{L^{\beta_{i_\infty+j_\infty}}(B_{3/4})}^{\beta_{i_\infty+j_\infty}}dt\leq C(N,\delta,p,s).
$$
Now we apply Proposition \ref{prop:improve} again with $\beta=\beta_{i_\infty+j_\infty}$, $\theta=\gamma-1/\beta_{i_\infty+j_\infty}$, $R+4h_0=6/8$ and $R-4h_0=6/8-8\,h_0>5/8$. We obtain (ignoring again the first term in the left-hand side)
$$
\sup_{t\in [-5/8,0]}\sup_{0<|h|< h_0}\left\|\frac{\delta_h u}{|h|^{(1-\varepsilon)\,\frac{\beta_{i_\infty+j_\infty}}{\beta_{i_\infty+j_\infty}+1}}}\right\|_{L^{\beta_{i_\infty+j_\infty}+1}(B_{5/8})}^{\beta_{i_\infty+j_\infty}+1}\leq C(N,\delta,p,s).
$$
Once we land here, as before we can repeat the arguments at the end of the proof of Theorem~\ref{teo:1} and use the Morrey-type embedding,  this time with 
\[
\beta = (1-\varepsilon)\,\frac{\beta_{i_\infty+j_\infty}}{\beta_{i_\infty+j_\infty}+1},\qquad q=\beta_{i_\infty+j_\infty}+1\qquad \mbox{ and }\qquad \alpha =\delta.
\] 
This gives 
$$
\sup_{t\in [-1/2,0]}[u(\cdot,t)]_{C^\delta(B_{1/2})}\leq C(N,\delta,p,s),
$$
and the proof is concluded.
\end{proof}

\section{Regularity in time}\label{sec:time}
In this section, we prove H\"older regularity in time using the previously obtained regularity in space. This approach uses energy estimates to control the growth of local integrals which yields a Campanato--type estimate.
We will use the notation
\[
\overline{u}_{x_0,R} = \fint_{B_R(x_0)}u\,dx. 
\]
When the center $x_0$ is clear from the context, we often simply write $\overline{u}_R$. 
For $u\in L^1(Q_{R,r}(x_0,t_0))$, we set 
\[
\overline{u}_{(x_0,t_0),R,r} = \fint_{Q_{R,r}(x_0,t_0)}\,u\,dx\,dt.
\]
Again, when the center $(x_0,t_0)$ is clear from the context, we simply write $\overline{u}_{R,r}$. 
\par
The following simple Poincar\'e--type inequality will be useful.
\begin{lm}
\label{lm:poincare}
Let $1\le p<\infty$ and let $B_r = B_r(x_0)$. Suppose that $u\in W^{s,p}(B_r)$, then for any nonnegative $\eta\in C_0^\infty(B_r)$ such that $\overline{\eta}_r=1$, there holds 
\begin{equation}
\label{poincare_mean}
\int_{B_r}|u-\overline{(u\,\eta)}_r|^p\,dx\le  \left(\frac{2^{N+s\,p}}{\omega_N}\,\|\eta\|^p_{L^\infty(B_r)}\right)\, r^{s\,p}\,\iint_{B_r\times B_r}\frac{|u(x)-u(y)|^p}{|x-y|^{N+s\,q}}\,dx\,dy, 
\end{equation}
\end{lm}

\begin{proof}
By using the fact that $\int_{B_r}\eta\,dx=|B_r|$ and Jensen's inequality, we obtain
\[
\begin{split}
\int_{B_r}|u-\overline{(u\,\eta)}_r|^p\,dx &= \int_{B_r}\left|\frac{1}{|B_r|}\,\int_{B_r}(u(x)-u(y))\,\eta(y)\,dy\right|^p\,dx\\
&\le \frac{\|\eta\|_{L^\infty(B_r)}^p}{|B_r|}\,\iint_{B_r\times B_r}|u(x)-u(y)|^p\,dx\,dy\\
&\le \frac{\|\eta\|_{L^\infty(B_r)}^p}{|B_r|}\,(2\,r)^{N+s\,p}\iint_{B_r\times B_r}\frac{|u(x)-u(y)|^p}{|x-y|^{N+s\,p}}\,dx\,dy.
\end{split}
\]
This concludes the proof.
\end{proof}

\begin{prop}\label{prop_time}
Suppose that $u$ is a local weak solution of 
\[
\partial_tu+(-\Delta_p)^su=0,\qquad\mbox{ in } B_2\times (-2,0], 
\]
such that 
\[
\|u\|_{L^{\infty}(\mathbb{R}^N\times[-1,0])}\le 1,
\]
and
\begin{equation}
\label{holderseminorm}
\sup_{t\in [-1/2,0]}[u(\cdot,t)]_{C^\delta(B_{1/2})}\le K_\delta,\qquad \mbox{ for any } s<\delta<\Theta(s,p),
\end{equation}
where $\Theta(s,p)$ is the exponent defined in \eqref{exponents}.
Then there is a constant $C= C(N,s,p,K_\delta,\delta)>0$ such that 
\[
|u(x,t)-u(x,\tau)|\le C\,|t-\tau|^{\gamma},\qquad \mbox{ for every } (x,t),(x,\tau)\in Q_{\frac{1}{4},\frac{1}{4}},
\]
where 
\[
\gamma = \frac{1}{\dfrac{s\,p}{\delta}-(p-2)}. 
\]
In particular, $u\in C^\gamma_{t}(Q_{\frac{1}{4},\frac{1}{4}})$ for any $\gamma<\Gamma(s,p)$, where $\Gamma(s,p)$ is the exponent defined in \eqref{exponents}.
\end{prop}

\begin{proof}
We take $(x_0,t_0)\in Q_{1/4,1/4}$ and choose 
\[
0<r<\frac{1}{8},\qquad 0<\theta<\frac{1}{8}.
\]
Consider the parabolic cylinder
\[
Q_{r,\theta}(x_0,t_0)=B_r(x_0)\times(t_0-\theta,t_0].
\]
Observe that by construction we have 
\[
Q_{r,\theta}(x_0,t_0)\subset B_\frac{3}{8}\times \left(-\frac{1}{2},0\right].
\]
Let $\eta\in C_0^\infty(B_{r/2}(x_0))$ be a non-negative cut-off function, such that
\[
\eta\equiv \|\eta\|_{L^\infty(B_{r/2}(x_0))} \mbox{ on } B_{r/4}(x_0),\qquad \overline{\eta}_{r}=1\qquad \mbox{ and }\qquad \|\nabla \eta\|_{L^\infty(B_{r/2}(x_0))}\le \frac{C}{r},
\] 
for some constant $C=C(\|\eta\|_{L^\infty(B_{r/2}(x_0))},N)>0$.
Observe that, thanks to the condition on its average, we have 
\[
\|\eta\|_{L^\infty(B_{r/2}(x_0))}=\frac{1}{|B_{r/4}(x_0)|}\,\int_{B_{r/4}(x_0)} \eta\,dx\le \frac{|B_{r}(x_0)|}{|B_{r/4}(x_0)|}\,\overline{\eta}_r=4^N.
\]
Thus the constant appearing in \eqref{poincare_mean} will only depend on $N,s$ and $p$.
\par
We now write 
\[
u(x,t)-\overline{u}_{r,\theta} =\Big(u(x,t)- \overline{(u\,\eta)}_{r}(t)\Big)+\Big(\overline{(u\,\eta)}_{r,\theta}-\overline{u}_{r,
\theta}\Big) + \Big(\overline{(u\,\eta)}_r(t)   -\overline{(u\,\eta)}_{r,\theta}\Big), 
\]
where we have set 
\[
\overline{(u\,\eta)}_{r}(t) = \fint_{B_r(x_0)}u(y,t)\,\eta(y)\,dy.
\]
Then 
\[
\begin{split}
\fint_{Q_{r,\theta}(x_0,t_0)}|u(x,t)-\overline{u}_{r,\theta}|\,dx\,dt&\le\fint_{Q_{r,\theta}(x_0,t_0)}\left|u(x,t)-\overline{(u\,\eta)}_r(t)\right|\,dx\,dt\\
&+\fint_{Q_{r,\theta}(x_0,t_0)}\left|\overline{u}_{r,\theta}-\overline{(u\,\eta)}_{r,\theta}\right|\,dx\,dt \\
&+ \fint_{Q_{r,\theta}(x_0,t_0)}\left|\overline{(u\,\eta)}_{r,\theta}-\overline{(u\,\eta)}_r(t)\right|\,dx\,dt\\
& =: A_1+A_2+A_3.
\end{split}
\]
We first note that 
\begin{equation}
\label{I2}
\begin{split}
A_2 = \left|\overline{u}_{r,\theta}-\overline{(u\eta)}_{r,\theta}\right| &= \left|\fint_{Q_{r,\theta}(x_0,t_0)}\left(u(x,t)-\overline{(u\eta)}_{r,\theta}\right)\,dx\,dt\right|\\
&\le \fint_{Q_{r,\theta}(x_0,t_0)}\left|u(x,t)-\overline{(u\eta)}_r(t)\right|\,dx\,dt\\
& +\fint_{Q_{r,\theta}(x_0,t_0)}\left|\overline{(u\eta)}_{r,\theta}-\overline{(u\eta)}_r(t)\right|\,dx\,dt\\
& = A_1+A_3. 
\end{split}
\end{equation}
Thus it suffices to estimate $A_1$ and $A_3$. 
In view of Lemma \ref{lm:poincare}, we have 
\[
\begin{split}
A_1 & \le \left(\fint_{Q_{r,\theta}(x_0,t_0)}\left|u(x,t)-\overline{(u\,\eta)}_r(t)\right|^p\,dx\,dt\right)^{\frac{1}{p}}\\
& \le C\,\left(\frac{r^{s\,p}}{|Q_{r,\theta}(x_0,t_0)|}\int_{t_0-\theta}^{t_0}\iint_{B_r(x_0)\times B_r(x_0)}\frac{|u(x,t)-u(y,t)|^p}{|x-y|^{N+s\,p}}\,dx\,dy\,dt\right)^{\frac{1}{p}},
\end{split}
\]
for some $C=C(N,s,p)>0$.
Recalling that $\delta>s$ and using the spatial H\"older continuity of $u$, we find that 
\begin{equation}
\label{I1}
\begin{split}
A_1\le C\,K_\delta\, r^{\delta}, \qquad  \mbox{ for some } C=C(N,s,p)>0.
\end{split}
\end{equation}
Indeed, by observing that for every $x\in B_r(x_0)$ we have $B_r(x_0)\subset B_{2\,r}(x)\subset B_{1/2}$, we get
\[
\begin{split}
\iint_{B_r(x_0)\times B_r(x_0)}\frac{|u(x,t)-u(y,t)|^p}{|x-y|^{N+s\,p}}\,dx\,dy&\le K_\delta^p\,\int_{B_r(x_0)}\left(\int_{B_{2\,r}(x)}\,|x-y|^{(\delta-s)\,p-N}\,dy\right)\,dx\\
&=\frac{K_\delta^p\,|B_r(x_0)|}{(\delta-s)\,p}\,N\,\omega_N\,(2\,r)^{(\delta-s)\,p},
\end{split}
\] 
where we used spherical coordinates to compute the last integral.
Observe that the width $\theta$ of the time does not come into play here.
\par 
We now turn to $A_3$ and first note that 
\begin{equation}
\label{I3b}
A_3\le \sup_{T_0,T_1\in (t_0-\theta,t_0]}\left|\overline{(u\,\eta)}_r(T_0) - \overline{(u\,\eta)}_r(T_1)\right|. 
\end{equation}
If $T_0,T_1 \in (t_0-\theta,t_0]$ with $T_0<T_1$, we use the weak formulation \eqref{locweakeq} with $\phi(x,t)=\eta(x)$, to obtain
\begin{equation}
\label{I3bb}
\begin{split}
|B_r(x_0)|\,\Big|\overline{(u\,\eta)}_r(T_0) - \overline{(u\,\eta)}_r(T_1)\Big| &= \left|\int_{B_r(x_0)}u(x,T_0)\,\eta(x)\, dx - \int_{B_r(x_0)}u(x,T_1)\,\eta(x)\, dx\right|\\
&=\left|\int_{T_0}^{T_1}\iint_{\mathbb{R}^N\times\mathbb{R}^N}J_p(u(x,\tau)-u(y,\tau))\,(\eta(x)-\eta(y))\,d\mu(x,y)\,d\tau\right|\\
& \le \left|\int_{T_0}^{T_1}\iint_{B_r(x_0)\times B_r(x_0)}J_p(u(x,\tau)-u(y,\tau))\,(\eta(x)-\eta(y))\,d\mu(x,y)\,d\tau\right|\\
&+ 2\,\left|\int_{T_0}^{T_1}\iint_{(\mathbb{R}^N\setminus B_r(x_0))\times B_{r/2}(x_0)}J_p(u(x,\tau)-u(y,\tau))\,\eta(x)\,d\mu(x,y)\,d\tau\right|\\
& = J_1 + J_2. 
\end{split}
\end{equation}
In order to control $J_2$, we claim that for $t\in [-1/2,0]$, $x\in B_r(x_0)$ and $y\in \mathbb{R}^N$, 
\begin{equation}
\label{holderadesso}
|u(x,t)-u(y,t)|\le C\,|x-y|^\delta,\qquad \mbox{ for some }  C=C(K_\delta,\delta)>0. 
\end{equation}
Indeed, if $y\in B_{1/2}$ this follows directly from the assumption. On the other hand, if $y\in \mathbb{R}^N\setminus B_{1/2}$, then by construction
\[
|x-y|^\delta\ge 8^{-\delta}\ge 8^{-\delta}\,\|u\|_{L^\infty(\mathbb{R}^N\times[-1,0])}\ge \frac{8^{-\delta}}{2}\,|u(x,t)-u(y,t)|.
\]
Additionally, if $y\in \mathbb{R}^N\setminus B_r(x_0)$ and $x\in B_{r/2}(x_0)$, we have 
\[
|x-y|\ge |y-x_0|-|x-x_0|\ge |y-x_0|-\frac{r}{2}\ge \frac{1}{2}\,|y-x_0|. 
\]
Thus, by using this and \eqref{holderadesso}, we get
\[
\begin{split}
J_2&\le 2\,(T_1-T_0)\,\|\eta\|_{L^\infty(B_{r/2}(x_0))} 
\iint_{(\mathbb{R}^N\setminus B_r(x_0))\times B_{r/2}(x_0)}\frac{|u(x,t)-u(y,t)|^{p-1}}{|x-y|^{N+s\,p}}\,dy\,dx\\
&\le C\,\theta\,\iint_{(\mathbb{R}^N\setminus B_r(x_0))\times B_{r/2}(x_0)}|x-y|^{(p-1)\,\delta-N-s\,p}\,dy\,dx\\
&\le C\,\theta\,r^{N}\,\int_{\mathbb{R}^N\setminus B_r(x_0)}|x_0-y|^{(p-1)\,\delta-N-s\,p}\,dy\\
&\le C\,\theta\,r^{N+\delta\,(p-1)-s\,p},
\end{split}
\]
for some $C = C(\delta,N,s,p,K_\delta)>0$. Observe that we used that $\delta\,(p-1)-s\,p<0$, in order to assure that the integral on $\mathbb{R}^N\setminus B_r(x_0)$ converges.
\par
As for $J_1$, we have for $\delta>s$
\[
\begin{split}
J_1 & \le [\eta]_{W^{s,p}(B_r(x_0))}\,\int_{t_0-\theta}^{t_0}\left(\iint_{B_r(x_0)\times B_r(x_0)}|u(x,t)-u(y,t)|^p\,d\mu(x,y)\right)^{\frac{p-1}{p}}dt \\
& \le C\,K_\delta^{p-1}\,r^{\frac{N}{p}-s}\,\int_{t_0-\theta}^{t_0}\left(\iint_{B_r(x_0)\times B_r(x_0)}|x-y|^{\delta\, p}\,d\mu(x,y)\right)^{\frac{p-1}{p}}dt \\
&\le C\,K_\delta^{p-1}\,\theta\,\left(r^{N+(\delta-s)\,p}\right)^{\frac{p-1}{p}}\,r^{\frac{N}{p}-s}\\
& = CK_\delta^{p-1}\,\theta\, r^{N-s\,p+\delta\, (p-1)},
\end{split}
\]
for some $C=C(N,s,p,\delta)>0$. By recalling \eqref{I3b} and using the estimates on $J_1$ and $J_2$ in \eqref{I3bb}, we have thus shown that 
\[
A_3 \le C\,K_\delta^{p-1}\,\theta\,r^{\delta\,(p-1)-s\,p}, \qquad \mbox{ for some  }C=C(\delta,N,s,p)>0.
\]
Hence, by also using \eqref{I1} and \eqref{I2}, we get 
\begin{equation}
\label{AAA}
A_1+A_2+A_3 \le C\,K_\delta\, r^\delta + C\,K_\delta^{p-1}\,\theta\,r^{\delta\,(p-1)-s\,p}.
\end{equation}
 We now have to distinguish two cases:
\vskip.2cm\noindent
$\bullet$ {\bf Case $s\,p\ge (p-1)$.} We now choose $\theta$ as follows
\[
\theta=\frac{1}{8}\, r^{s\,p-\delta\,(p-2)}.
\]
Observe that since $s\,p\ge (p-1)$, then $\Theta(s,p)=1$ and we always have\footnote{Indeed, observe that 
\[
s\,p\ge (p-1)=(p-2)+1>\delta\,(p-2)+1,
\]
thanks to the fact that $0<\delta<1$. This in turn implies
\[
s\,p-\delta(p-2)>1,
\]
as claimed.}
\begin{equation}
\label{culo}
s\,p-\delta\,(p-2)> 1.
\end{equation}
We thus obtain from \eqref{AAA}
\[ 
\fint_{Q_{r,\theta}(x_0,t_0)}|u-\overline{u}_{r,\theta}|\,dx\,dt\le C\,r^\delta, \qquad \mbox{ for some } C=C(\delta,K_\delta, N,s,p)>0.
\]
By the characterization of Campanato spaces on $\mathbb{R}^{N+1}$ with respect to a general metric (see \cite[Teorema 3.I]{DaP} and also \cite[Theorem 3.2]{Go}), this implies that $u$ is $\delta-$H\"older continuous in $Q_{1/4,1/4}$ with respect to the metric
$$
\widetilde d((x,\tau_1),(y,\tau_2))=|x-y|+|\tau_1-\tau_2|^\frac{1}{s\,p-\delta\,(p-2)}.
$$
By keeping \eqref{culo} into account, we can infer that $\widetilde d$ is a true metric.
Thus, in particular, we have the estimate
$$
\sup_{x\in \overline{B_{1/4}}}|u(x,\tau_1)-u(x,\tau_2)|\leq C\,|\tau_1-\tau_2|^\gamma, \qquad \mbox{ for } 
\gamma = \frac{1}{\dfrac{s\,p}{\delta}-(p-2)},
$$
where $C=C(K_\delta, N,s,p)>0$. Observe that the continuous function 
\[
\delta\mapsto \frac{1}{\dfrac{s\,p}{\delta}-(p-2)},\qquad \mbox{ for } 0<\delta<1,
\]
is increasing and that
\[
\lim_{\delta\nearrow 1} \frac{1}{\dfrac{s\,p}{\delta}-(p-2)}=\frac{1}{s\,p-(p-2)}.
\]
Thus for every $0<\gamma<1/(s\,p-(p-2))$, there exists $s<\delta<1$ such that
\[
\gamma =\frac{1}{\dfrac{s\,p}{\delta}-(p-2)}.
\]
The proof is over in this case.
\vskip.2cm\noindent
$\bullet$ {\bf Case $s\,p< (p-1)$.} In this case, we revert the hierarchy between time and space and
choose $r$ as follows
\[
(8\,r)^{s\,p-(p-2)\,\delta}=\theta,\qquad \mbox{ i.\,e. }\quad r=\frac{1}{8}\,\theta^\frac{1}{s\,p-(p-2)\,\delta}.
\]
Observe that the exponent on $\theta$ is positive: indeed, for $p=2$ this is straightforward, while for $p>2$ we use that
\[
\delta\,(p-2)<\frac{s\,p}{p-1}\,(p-2)<s\,p.
\]
We further notice that now
\begin{equation}
\label{culo2}
s\,p-(p-2)\,\delta\le 1,
\end{equation}
up to choose $\delta$ sufficiently close\footnote{More precisely, it is sufficient to take
\[
\delta=\frac{s\,p}{p-1}-\varepsilon,
\]
with $0<\varepsilon<s/(p-1)$ such that
\[
\varepsilon\,(p-2)\le 1-\frac{s\,p}{p-1}.
\]
Such a choice is feasible, since now $s\,p<(p-1)$.} to $s\,p/(p-1)$. This time, we obtain from \eqref{AAA}
\[ 
\fint_{Q_{r,\theta}(x_0,t_0)}|u-\overline{u}_{r,\theta}|\,dx\,dt\le C\,\theta^\frac{\delta}{s\,p-(p-2)\,\delta}, \qquad \mbox{ for some } C=C(\delta,K_\delta, N,s,p)>0.
\]
Again by the Campanato--type theorem of \cite[Teorema 3.I]{DaP}, this shows that $u$ is $(\delta/(s\,p-(p-2)\,\delta))-$H\"older continuous in $Q_{1/4,1/4}$ with respect to the metric 
\[
\widetilde{d}((x,\tau_1),(y,\tau_2))=|x-y|^{s\,p-(p-2)\,\delta}+|\tau_1-\tau_2|.
\]
Observe that this is indeed a metric, thanks to \eqref{culo2}. In particular, we have the estimate
$$
\sup_{x\in \overline{B_{1/4}}}|u(x,\tau_1)-u(x,\tau_2)|\leq C\,|\tau_1-\tau_2|^\gamma, \qquad \mbox{ for } \
\gamma = \frac{1}{\dfrac{s\,p}{\delta}-(p-2)},
$$
where $C=C(\delta,K_\delta, N,s,p)>0$. We now use that the continuous function
\[
\delta\mapsto \frac{1}{\dfrac{s\,p}{\delta}-(p-2)},\qquad 0<\delta<\frac{s\,p}{p-1},
\]
is increasing and that
\[
\lim_{\delta\nearrow \frac{s\,p}{p-1}} \frac{1}{\dfrac{s\,p}{\delta}-(p-2)}=1.
\]
Thus, for every $\gamma<1$, there exists $s<\delta<s\,p/(p-1)$ such that
\[
\gamma=\frac{1}{\dfrac{s\,p}{\delta}-(p-2)}.
\] 
This concludes the proof in this case, as well.
\end{proof}

\section{Proof of the main theorem}\label{sec:main}
Before proving our main result, we will need the following lemma, which allows us to control the parabolic Sobolev-Slobodecki\u{\i} seminorm of a local weak solution $u$ in terms of its $L^\infty$ norm. 

\begin{lm}\label{lm:seminormestimate}
Let $p\ge 2$ and $0<s<1$. Let $u$ be a local weak solution of
\[
\partial_t u+(-\Delta_p)^su=0,\quad \mbox{ in } B_2\times (-2\,R^{s\,p},0],
\]
such that $u\in L^\infty(\mathbb{R}^N\times[-R^{s\,p},0])$.
Then 
\[
\left(R^{-N}\,\int_{-\frac{7}{8}\,R^{s\,p}}^{0}[u]_{W^{s,p}(B_{(3\,R)/4}(x_0))}^p\,dt\right)^{\frac{1}{p}} \le C\,\Big(\|u\|_{L^\infty(\mathbb{R}^N\times[-R^{s\,p},0]})+1\Big),
\]
for some $C=C(N,s,p)>0$.
\end{lm}
 \begin{proof}
Without loss of generality, we may suppose that $x_0=0$.
Let us set
\[
k = \|u\|_{L^\infty(\mathbb{R}^N\times[-R^{s\,p},0])} + 1\qquad \mbox{�and }\qquad \widetilde u = u + k.
\]
Then $\widetilde u$ is a local weak solution in $B_2\times (-2\,R^{s\,p},0]$ and $\widetilde u\ge 1$ in $\mathbb{R}^N\times[-R^{s\,p},0]$. For all $\phi(x,t) = \eta(x)\, \psi(t)$ with $\psi \in C^{\infty}$ such that $\psi(t)=0$ for $t\leq -R^{s\,p}$ and $\psi(0)=1$, and $\eta \in C_0^{\infty}(B_{R})$, we get from a slight modification of \cite[Lemma 2.2]{str}
\[
\begin{split}
		\int_{-R^{s\,p}}^{0} \Big[\widetilde u(\cdot,t)\,\phi(\cdot,t)\Big]^p_{W^{s,p}(B_R)}\, dt&\leq C \int_{-R^{s\,p}}^{0} \iint_{B_{R}\times B_{R}} \max\Big\{\widetilde u(x,t),\, \widetilde u(y,t)\Big\}^p\,|\phi(x,t)-\phi(y,t)|^p\, d\mu\, dt\\
&+C\left (\sup_{x\in\mathrm{supp\,}\eta}\int_{\mathbb{R}^N\setminus B_{R}}\frac{dy}{|x-y|^{N+s\,p}}\biggr)\biggl (\int_{-R^{s\,p}}^{0}\int_{B_{R}}\widetilde u(x,t)^p\,\phi(x,t)^p\, dx dt\right)\\
&+C\int_{-R^{s\,p}}^{0}\left(\sup_{x\in\mathrm{supp\,}\eta}\int_{\mathbb{R}^N\setminus B_{R}}\frac{(u(y,t)_+)^{p-1}}{|x-y|^{N+s\,p}}\,dy\,\int_{B_{R}}\widetilde u(x,t)\,\phi(x,t)^p\,dx\right)dt\\
		&+ \frac{1}{2}\int_{-R^{s\,p}}^{0} \int_{B_{R}}\widetilde u(x,t)^{2} \left ( \frac{\partial \phi^p}{\partial t} \right )_{+}\, dx\, dt+\int_{B_{R}} \widetilde u (x,0)\, dx.
\end{split}
\] 
We choose $\eta$ such that 
\[
\eta \equiv 1\ \text{ in }\ B_{\frac{3}{4}\, R},\qquad |\nabla \eta|\le \frac{C}{R}\qquad \mbox{ and }\qquad \textrm{dist}(\mathrm{supp\,}\eta,\mathbb{R}^N\setminus B_{R})\ge c_0\,R,
\]
and $\psi$ such that 
\[
\psi \equiv 1\ \text{ in }\ \left[-\frac{7}{8}\,R^{s\,p},0\right]\qquad \mbox{ and }\qquad |\psi'|\le \frac{C}{R^{s\,p}}.
\] 
It is then a routine matter to show that 
\[
\int_{-\frac{7}{8}\,R^{s\,p}}^{0}\big[\widetilde u\big]_{W^{s,p}(B_{3R/4})}^p\,dt \le C\,R^N\,(k^p+k^2+k)\le C\,R^N\,k^p, 
\]
where $C= C(N,s,p)>0$ and  we used that $p\ge 2$ and $k\ge 1$. This proves the lemma.
\end{proof}

We are now in the position to prove Theorem \ref{thm1}.
\begin{proof}[Proof of Theorem \ref{thm1}] The continuity in space is contained in Theorem \ref{teo:1higher}, thus we only need to prove the continuity in time. We take for simplicity $T_0=0$. If $u$ is a local weak solution in $B_2\times (-2\,R^{s\,p},0]$, 
we obtain from \eqref{apriori}
\[
\begin{split}
\sup_{t\in \left[-\frac{R^{s\,p}}{2},0\right]} [u(\cdot,t)]_{C^\delta(B_{R/2}(x_0))}&\leq \frac{C}{R^\delta}\,\left(\|u\|_{L^\infty\left(\mathbb{R}^N\times \left[-R^{s\,p},0\right]\right)} + 1\right)\\
&+\frac{C}{R^\delta}\,\left(R^{-N}\,\int_{-\frac{7}{8}\,R^{s\,p}}^0 [u]^p_{W^{s,p}(B_{3R/4}(x_0))}dt \right)^\frac{1}{p} .
\end{split}
\]
An application of Lemma \ref{lm:seminormestimate} gives 
\begin{equation}
\label{petula}
\sup_{t\in \left[-\frac{R^{s\,p}}{2},0\right]} [u(\cdot,t)]_{C^\delta(B_{R/2}(x_0))}\leq \frac{C}{R^\delta}\,\left(\|u\|_{L^\infty(\mathbb{R}^N\times [-R^{s\,p},0])}  + 1\right).
\end{equation}
We set 
\[
\mathcal{N}_R = \|u\|_{L^\infty(\mathbb{R}^N\times [-R^{s\,p},0])} + 1,
\]
then for $\alpha$ such that $\alpha\in [-R^{s\,p}(1-\mathcal{N}_R^{2-p}),0]$, we define the rescaled function 
\[
u_R(x,t) = \frac{1}{\mathcal{N}_R}u\left(R\,x,\frac{1}{\mathcal{N}_R^{p-2}}\,R^{s\,p}\,t+\alpha\right).
\]
This is a local weak solution in $B_2(x_0)\times(-2,0]$ satisfying the hypothesis of Proposition \ref{prop_time}. Indeed, by construction
\[
\|u_R\|_{L^\infty(\mathbb{R}^N\times[-1,0])}\le 1,
\]
and the estimate on the spatial H\"older seminorm \eqref{holderseminorm} of $u_R$ follows from \eqref{petula}.  
From Proposition \ref{prop_time} we obtain 
\[
\sup_{x\in B_{1/4}} [u_R(x,\cdot)]_{C^\gamma([-1/4,0])}\le C,
\]
for every $0<\gamma<\Gamma(s,p)$.
The claimed result follows by scaling back and varying $\alpha$ as in the proof of Theorem \ref{teo:1}.
\end{proof}

\appendix

\section{Existence for an initial boundary value problem}
\label{sec:A}

In order to give the definition of weak solution for an initial boundary value problem, we need to define a suitable functional space.  We assume that $\Omega\Subset\Omega'\subset \mathbb{R}^N$, where $\Omega'$ is a bounded open set in $\mathbb{R}^N$.
Given a function
\[
\psi\in W^{s,p}(\Omega')\cap L^{p-1}_{s\,p}(\mathbb{R}^N),
\]
we define as in \cite{KKP} (see also \cite[Proposition 2.12]{bralinschi}) the space 
\[
X_\psi^{s,p}(\Omega,\Omega') = \left\{v\in W^{s,p}(\Omega')\cap L^{p-1}_{s\,p}(\mathbb{R}^N)\,:\, v=\psi\text{ on }\mathbb{R}^N\setminus\Omega\right\}. 
\]
When $\psi\equiv 0$, the boundedness of $\Omega'$ entails that 
\[
X_0^{s,p}(\Omega,\Omega') = \{v\in W^{s,p}(\Omega')\cap L^{p-1}_{s\,p}(\mathbb{R}^N):v=0\text{ on }\mathbb{R}^N\setminus\Omega\}\subset W^{s,p}(\Omega').
\]
We endow the space $X_0^{s,p}(\Omega,\Omega')$ with the norm $W^{s,p}(\Omega')$, then this is a reflexive Banach space. Thanks to the previous inclusion, we also have that 
\[
(W^{s,p}(\Omega'))^*\subset (X^{s,p}_0(\Omega,\Omega'))^*. 
\]

\begin{defi}
\label{bdryweak}
Let $I=[t_0,t_1]$ and $p\ge 2$. With the notation above, assume that the functions $u_0,f$ and $g$ satisfy
\[
u_0\in L^2(\Omega), 
\]
\[
f\in L^{p'}(I;(W^{s,p}(\Omega'))^*),
\]
\[
g\in L^p(I;W^{s,p}(\Omega'))\cap L^{p-1}(I;L_{s\,p}^{p-1}(\mathbb{R}^N))\ \mbox{ and }\ \partial_t g\in L^{p'}(I;(W^{s,p}(\Omega'))^*). 
\]
We say that $u$ is a {\it weak solution of the initial boundary value problem} 
\begin{equation}
\label{Mweaksol}
 \left\{\begin{array}{rcll}
\partial_tu + (-\Delta_p)^su&=&f,&\mbox{ in }\Omega\times I,\\
u&=&g,& \mbox{ on }(\mathbb{R}^N\setminus\Omega)\times I,\\
u(\cdot,t_0) &=& u_0,&\mbox{ on }\Omega, 
\end{array}\right.
\end{equation}
if the following properties are verified:
\begin{itemize}
\item $u\in L^p(I;W^{s,p}(\Omega'))\cap L^{p-1}(I;L_{s\,p}^{p-1}(\mathbb{R}^N))\cap C(I;L^2(\Omega))$;
\vskip.2cm
\item $u\in X_{\mathbf{g}(t)}(\Omega,\Omega')$ for almost every $t\in I$, where $(\mathbf{g}(t))(x)=g(x,t)$;
\vskip.2cm
\item $\lim_{t\to t_0}\|u(\cdot,t) - u_0\|_{L^2(\Omega)}=0$;
\vskip.2cm
\item for every $J=[T_0,T_1]\subset I$ and every $\phi\in L^{p-1}(J;X_0^{s,p}(\Omega,\Omega'))\cap C^1(J;L^2(\Omega))$
\[
\begin{split}
-\int_J\int_\Omega u(x,t)\,\partial_t\phi(x,t)\,dx\,dt&+ \int_J\iint_{\mathbb{R}^N\times\mathbb{R}^N}\frac{J_p(u(x,t)-u(y,t))\,(\phi(x,t)-\phi(y,t))}{|x-y|^{N+s\,p}}\,dx\,dy\,dt  \\   
& = \int_\Omega u(x,T_0)\,\phi(x,T_0)\,dx -\int_\Omega u(x,T_1)\,\phi(x,T_1)\,dx \\
&+ \int_J\langle f(\cdot,t),\phi(\cdot,t)\rangle\,dt.
\end{split}
\]
\end{itemize}
\end{defi}
The starting point for proving the existence of weak solutions is an abstract theorem for parabolic equations in Banach spaces. Before stating the theorem, we will briefly explain its framework. Let $V$ be a separable reflexive Banach space and let $H$ be a Hilbert space that we identify with its dual, i.e.\ $H^* = H$. Suppose that $V$ is dense and continuously embedded in $H$. 
If $v\in V$ and $h\in H$, we identify $h$ as an element of $V^*$ through the relation\footnote{With these identifications, we have $V\subset H\subset V^*$. This is sometimes called in the literature {\it Gelfand triple}.}
\begin{equation}\label{HtoVstar}
\langle h,v\rangle = (h,v)_H.
\end{equation} 
Here $\langle\cdot,\cdot\rangle$ denotes the duality pairing between $V$ and $V^*$ and $(\cdot,\cdot)_H$ denotes the scalar product in $H$. 
Let $I$ be an interval and $1<p<\infty$.  By \cite[Proposition 1.2, Chapter]{Sh}, we have
\begin{equation}
\label{CIH1}
W_p(I) := \{v\in L^p(I;V):\, v'\in L^{p'}(I;V^*)\}\subset C(I;H),
\end{equation}
and
\[
\mbox{ for } v\in W_p(I),\qquad t\mapsto \|v(t)\|^2_H \mbox{ is absolutely continuous}\qquad
\frac{d}{dt}\|v(t)\|^2_H = 2\,\langle v'(t),v(t)\rangle.
\]
More generally, by \cite[Corollary 1.1, Chapter III]{Sh}, for every $u,v\in W_p(I)$ the scalar product $t\mapsto (u(t),v(t))_H$ is an absolutely continuous function and there holds
\[
\frac{d}{dt}(u(t),v(t))_H = \langle u'(t),v(t)\rangle + \langle u(t),v'(t)\rangle,\qquad \mbox{ for a.\,e. }t\in I. 
\]
We recall that an operator $\mathcal{A}:V\to V^*$ is said to be
\begin{itemize}
\item \emph{monotone} if for every $u,v\in V$,
\[
\langle\mathcal{A} (u)-\mathcal{A} (v),u-v\rangle\ge 0;
\] 
\item {\it hemicontinuous} if the real function $\lambda\mapsto \langle\mathcal{A}(u+\lambda\, v),v\rangle$ is continuous, for every $u,v\in V$.
\end{itemize}

\begin{teo}\label{thmShow}
Let $V$ be a separable, reflexive Banach space and let $\mathcal{V} = L^p(I;V)$, for $1<p<\infty$, where $I=[t_0,t_1]$. Suppose that $H$ is a Hilbert space such that $V$ is dense and continuously embedded in $H$ and that $H$ is embedded into $V^*$ according to the relation \eqref{HtoVstar}. Assume that the family of operators $\mathcal{A}(t,\cdot):V\to V^*$, $t\in I$ satisfies: 
\begin{enumerate}
\item[{\it (i)}] for every $v\in V$,  the function $\mathcal{A}(\cdot,v):I\to V^*$ is measurable;
\vskip.2cm
\item[{\it (ii)}] for almost every $t\in I$, the operator $\mathcal{A}(t,\cdot):V\to V^*$ is monotone, hemicontinuous and bounded by 
\[
\|\mathcal{A}(t,v)\|_{V^*}\le C\,\Big(\|v\|^{p-1}_{V}+k(t)\Big), \qquad \mbox{ for } v\in V\quad \mbox{ and }\quad k\in L^{p'}(I),
\]
\item[{\it (iii)}] 
there exist a real number $\beta>0$ and a function $\ell\in L^1(I)$ such that
\[ 
\langle \mathcal{A}(t,v),v\rangle + \ell(t)\ge \beta\,\|v\|^p_V, \qquad \mbox{ for a.\,e. }t\in I \mbox{ and }v\in V.
\]
\end{enumerate}
Then for each $f\in \mathcal{V}^*=L^{p'}(I;V^*)$ and $u_0\in H$, there exists a unique $u\in W_p(I)$ satisfying 
\[
u'(t) + \mathcal{A}(t,u(t)) = f(t), \quad\mbox{ in }\mathcal{V}^*,\qquad u(t_0)=u_0\text{ in }H.
\] 
This means that $u\in\mathcal{V}$, $u'\in \mathcal{V}^*$ and 
\[
\int_I \langle u'(t),\phi(t)\rangle\,dt + \int_I \langle \mathcal{A}(t,u(t)),\phi(t)\rangle\,dt = \int_I \langle f(t),\phi(t)\rangle\,dt,\qquad \mbox{ for all }\phi\in \mathcal{V}. 
\]
\end{teo}
\begin{proof}
The existence of a unique solution $u\in\mathcal{V}$ is contained in \cite[Proposition 4.1, Chapter III]{Sh}. The condition $(iii)$ is slightly different here, due to the presence of the function $\ell(t)$, but the proof of \cite[Proposition 4.1, Chapter III]{Sh} goes through with minor changes. 
\end{proof}
In order to prove existence for our problem \eqref{bdryweak}, we will use Theorem \ref{thmShow} with the choice $V=X_0^{s,p}(\Omega,\Omega')$. This is the content of the next result, which generalizes \cite[Theorem 2.5]{MRT}. The latter only deals with the case $f\equiv g\equiv 0$.
\begin{teo}\label{thm:exunsol}
Let $p\ge 2$, let $I = [t_0,t_1]$ and suppose that $g$ satisfies
\begin{align*}
&g\in L^p(I;W^{s,p}(\Omega'))\cap L^{p}(I;L^{p-1}_{s\,p}(\mathbb{R}^N)),\quad\partial_t g\in L^{p'}(I;(X_0^{s,p}(\Omega,\Omega'))^*),\\
&\lim_{t\to t_0}\|g(\cdot,t)-g_0\|_{L^2(\Omega)},\qquad \mbox{ for some }g_0\in L^2(\Omega). 
\end{align*}
Suppose also that 
\[
f\in L^{p'}(I;(X_0^{s,p}(\Omega,\Omega'))^*).
\]
Then for any initial datum $u_0\in L^2(\Omega)$, there exists a unique weak solution $u$ to problem \eqref{Mweaksol}. 
\end{teo}
\begin{proof}
We denote by $\bf g$ the mapping ${\bf g}:I\to W^{s,p}(\Omega')$, given by $({\bf{ g}}(t))(x) = g(x,t)$. 
For almost every $t\in I$, we define the operator 
\[
\mathcal{A}_t:X_{{\bf g}(t)}^{s,p}(\Omega,\Omega')\to (W^{s,p}(\Omega'))^*,
\] 
by 
\[
\begin{split}
\langle\mathcal{A}_t(v),\phi\rangle &= \iint_{\Omega'\times\Omega'}\frac{J_p(v(x)-v(y))\,(\phi(x)-\phi(y))}{|x-y|^{N+s\,p}}\,dx\,dy+2\,\iint_{\Omega\times(\mathbb{R}^N\setminus\Omega')}\frac{J_p(v(x)-g(y,t))\,\phi(x)}{|x-y|^{N+s\,p}}\,dx\,dy.
\end{split}
\]
It is easy to check that $\mathcal{A}_t(v)\in (W^{s,p}(\Omega))^*$ whenever $v\in X_{{\bf g}(t)}^{s,p}(\Omega,\Omega')$. Additionally, $\mathcal{A}_t$ is a monotone operator, see \cite[Lemma 3]{KKP}. We now define $\mathcal{A}:X_0^{s,p}(\Omega,\Omega')\times I\to (W^{s,p}(\Omega'))^*$ to be the operator defined by 
\[
\mathcal{A}(v,t) = \mathcal{A}_t(v+{\bf g}(t)). 
\]
Observe that this is well-defined, since 
\[
v+{\bf g}(t)\in X^{s,p}_{\mathbf{g}(t)}(\Omega,\Omega'),\qquad \mbox{ for every } v\in X^{s,p}_0(\Omega,\Omega').
\]
We next show that the operator $\mathcal{A}$, together with the spaces 
\[
V= X_0^{s,p}(\Omega,\Omega'),\qquad \mathcal{V}= L^p(I;X_0^{s,p}(\Omega,\Omega'))\qquad \mbox{ and }\quad H = L^2(\Omega),
\] 
fits into the framework of Theorem \ref{thmShow}. 
Since $p\ge 2$ and $\Omega'$ is bounded, $X_0^ {s,p}(\Omega,\Omega')$ is dense and continuously embedded in $L^2(\Omega)$. This follows from H\"older's inequality and the fact that smooth functions are dense in both spaces.
Note that $\mathcal{A}$ inherits the property of monotonicity from $\mathcal{A}_t$ since 
\[
\begin{split}
\langle \mathcal{A}(u,t)-\mathcal{A}(v,t),u-v\rangle &= \langle\mathcal{A}(u,t)-\mathcal{A}(v,t),u+{\bf{ g}}(t)-(v+{\bf{g}}(t))\rangle\\
& = \langle\mathcal{A}_t(u+{\bf{ g}}(t))-\mathcal{A}_t(v+{\bf{ g}}(t)),u+{\bf{ g}}(t)-(v+{\bf{g}}(t))\rangle\ge 0.
\end{split}
\]
We next claim that 
\begin{equation}\label{Mopnorm}
|\langle \mathcal{A}(v,t),\phi\rangle| \le C\|v\|^{p-1}_{W^{s,p}(\Omega')}\,\|\phi\|_{W^{s,p}(\Omega')} + C\,\Big(\|{\bf g}(t)\|_{W^{s,p}(\Omega')}^{p-1}+\|{\bf g}(t)\|_{L^{p-1}_{sp}(\mathbb{R}^N)}^{p-1}\Big)\,\|\phi\|_{W^{s,p}(\Omega')}. 
\end{equation}
We have 
\begin{equation}\label{Anoll}
\begin{split}
\langle\mathcal{A}(v,t),\phi\rangle &= \iint_{\Omega'\times \Omega'}\frac{J_p(v(x)-v(y)+(g(x,t)-g(y,t)))\,(\phi(x)-\phi(y))}{|x-y|^{N+s\,p}}\,dx\,dy\\
&+2\,\iint_{\Omega\times(\mathbb{R}^N\setminus\Omega')}\frac{J_p(v(x)+g(x,t)-g(y,t))\,\phi(x)}{|x-y|^{N+s\,p}}\,dx\,dy,
\end{split}
\end{equation}
The first term on the right-hand side of \eqref{Anoll} can be bounded by 
\[
C\,\Big(\|v\|^{p-1}_{W^{s,p}(\Omega')} + \|{\bf g}(t)\|^{p-1}_{W^{s,p}(\Omega')}\Big)\,\|\phi\|_{W^{s,p}(\Omega')},
\]
using H\"older's inequality. For the second term we observe that, when $x\in \Omega$ and $y\in \mathbb{R}^N\setminus \Omega'$,  
\[
\frac{1}{C}\,\frac{1}{1+|y|^{N+s\,p}}\le \frac{1}{|x-y|^{N+s\,p}}\le \frac{C}{1+|y|^{N+s\,p}}, 
\]
where $C>1$ depends only on the distance between $\Omega$ and $\Omega'$. Since $1/(1+|y|^{N+s\,p})\in L^1(\mathbb{R}^N)$, the second term in the right-hand side of \eqref{Anoll} can be estimated by 
\[
\begin{split}
C\,\int_{\Omega}\Big(|v(x)|^{p-1}+|g(x,t)|^{p-1}\Big)\,|\phi(x)|\,dx &+ C\,\left(\int_{\mathbb{R}^N\setminus \Omega'}\frac{|g(y,t)|^{p-1}}{1+|y|^{N+s\,p}}\,dy\right)\,\int_{\Omega}|\phi(x)|dx\\
&\le C\,\Big(\|v\|_{L^p(\Omega)}^{p-1} + \|\mathbf{g}(t)\|_{L^p(\Omega)}^{p-1}\Big)\, \|\phi\|_{L^p(\Omega)} + \|\mathbf{g}(t)\|_{L^{p-1}_{s\,p}(\mathbb{R}^N)}^{p-1}\,\|\phi\|_{L^1(\Omega)}\\
&\le C\,\Big(\|v\|^{p-1}_{W^{s,p}(\Omega')} + \|\mathbf{g}(t)\|^{p-1}_{W^{s,p}(\Omega')} + \|\mathbf{g}(t)\|_{L^{p-1}_{s\,p}(\mathbb{R}^N)}^{p-1}\Big)\,\|\phi\|_{W^{s,p}(\Omega')}, 
\end{split}
\]
where we used the continuous inclusion $L^p(\Omega)\subset W^{s,p}(\Omega')$. This finally shows \eqref{Mopnorm}.
Observe that 
\[
t\mapsto \|{\bf g}(t)\|^{p-1}_{W^{s,p}(\Omega')}+\|{\bf g}(t)\|_{L^{p-1}_{s\,p}(\mathbb{R}^N)}^{p-1} \quad \mbox{ belongs to }L^{p'}(I),
\] 
thanks to the assumptions on $g$. Thus in order to verify ($ii$) of Theorem \ref{thmShow}, we are left with proving hemicontinuity. For this, fixed $t\in I$ and $\lambda,\lambda_0\in\mathbb{R}$, we consider
\[
\langle \mathcal{A}(u+\lambda\,v,t),v\rangle-\langle \mathcal{A}(u+\lambda_0\,v,t),v\rangle,\qquad \mbox{ for } u,v\in X^{s,p}_0(\Omega,\Omega').
\] 
In order to show that this differences goes to $0$ as $\lambda$ goes to $\lambda_0$, it is sufficient to write
\[
\begin{split}
\langle \mathcal{A}(u+\lambda\,v,t),v\rangle-\langle \mathcal{A}(u+\lambda_0\,v,t),v\rangle&=\langle \mathcal{A}(u+\lambda\,v,t)-\mathcal{A}(u+\lambda_0\,v,t),v\rangle\\
&=\langle \mathcal{A}_t(u+\mathbf{g}(t)+\lambda\,v)-\mathcal{A}(u+\mathbf{g}(t)+\lambda_0\,v),v\rangle
\end{split}
\]
and then use \cite[Lemma 3]{KKP}. This proves that $\mathcal{A}$ is hemicontinuous for almost every\ $t\in I$. 
\par
Finally, as for hypothesis $(iii)$ of Theorem \ref{thmShow}, we observe that if $v\in X_0^{s,p}(\Omega,\Omega')$, then by using Poincar\'e inequality we have
\[
\|v\|_{W^{s,p}(\Omega')}=\|v\|_{L^p(\Omega')}+[v]_{W^{s,p}(\Omega')}\le C\,[v]_{W^{s,p}(\Omega')},
\]
for a constant $C=C(N,p,s,\Omega,\Omega')>0$.
Additionally, using H\"older's inequality and Young's inequality, we obtain 
\[
\langle\mathcal{A}(v,t),v\rangle \ge c\,[v]_{W^{s,p}(\Omega')}^p - C_1\,\|\mathbf{g}(t)\|_{W^{s,p}(\Omega')}^p - C_2\,\|\mathbf{g}(t)\|_{L^{p-1}_{s\,p}(\mathbb{R}^N)}^p.
\]
By combining this with the previous estimate, hypothesis $(iii)$ of Theorem \ref{thmShow} is checked.
According to \eqref{CIH1}, $g\in C(I;L^2(\Omega))$ and we may define $g_0=\mathbf{g}(t_0)$ in $L^2(\Omega)$. 
From Theorem \ref{thmShow}, for every $u_0\in L^2(\Omega)$ we obtain a unique solution 
\[
v\in W_p(I)=\Big\{\varphi \in L^p(I;X_0^{s,p}(\Omega,\Omega'))\, :\, \varphi'\in L^{p'}(I;(X_0^{s,p}(\Omega,\Omega'))^*)\Big\},
\]
to the problem 
\[
\partial_t v+\mathcal{A}( v,t) = -\partial_t {\bf{g}}(t) + f(t)\quad \text{in }L^{p'}(I;(X_0^{s,p}(\Omega,\Omega'))^*),\qquad \mbox{ with }\ v(t_0)=u_0-g_0. 
\]
Observe that again by \eqref{CIH1}, we also have $v\in C(I;L^2(\Omega))$. Since $v$ is a solution, we have
\[
\int_{t_0}^{t_1} \langle \partial_t v+\partial_t {\bf{g}}(t),\phi(t)\rangle\,dt+\int_{t_0}^{t_1}\langle\mathcal{A}_t( v+\mathbf{g}(t)),\phi\rangle\,dt=\int_{t_0}^{t_1} \langle f(t),\phi(t)\rangle\,dt,
\]
for every $\phi\in L^{p}(I;X_0^{s,p}(\Omega,\Omega'))$.
Upon setting $u=v+g$, we find that 
\[
u\in C(I;L^2(\Omega))\cap L^p(I;X_{\mathbf{g}(\cdot)}(\Omega,\Omega'))\cap L^{p}(I;L_{s\,p}^{p-1}(\mathbb{R}^N)),\quad \text{with }\partial_tu\in L^{p'}(I;(X^{s,p}_0(\Omega,\Omega'))^*). 
\]
and it verifies
\[
\int_{t_0}^{t_1} \langle \partial_t u,\phi(t)\rangle\,dt+\int_{t_0}^{t_1}\langle\mathcal{A}_t(u),\phi\rangle\,dt=\int_{t_0}^{t_1} \langle f(t),\phi(t)\rangle\,dt,
\]
for every $\phi\in L^{p}(I;X_0^{s,p}(\Omega,\Omega'))$. In particular, if we take $J=[T_0,T_1]\subset I$ and $\phi\in L^p(J;X_0^{s,p}(\Omega,\Omega'))$, by extending $\phi$ to be $0$ outside $J$ we get
\[
\int_{T_0}^{T_1} \langle \partial_t u,\phi(t)\rangle\,dt+\int_{T_0}^{T_1}\langle\mathcal{A}_t(u),\phi\rangle\,dt=\int_{T_0}^{T_1} \langle f(t),\phi(t)\rangle\,dt
\]
If now the test function $\phi$ is further supposed to belong to $L^{p}(J;X_0^{s,p}(\Omega,\Omega'))\cap C^1(J;L^2(\Omega))$, we can integrate by parts
\[
\begin{split}
\int_{T_0}^{T_1} \langle \partial_t u,\phi(t)\rangle\,dt&=\langle u(T_1),\phi(T_1)\rangle-\langle u(T_0),\phi(T_0)\rangle-\int_{T_0}^{T_1} \langle u,\phi'(t)\rangle\,dt\\
&=(u(T_1),\phi(T_1))_{L^2(\Omega)}-(u(T_0),\phi(T_0))_{L^2(\Omega)}-\int_{T_1}^{T_1} ( u,\phi'(t))_{L^2(\Omega)}\,dt\\
&=\int_\Omega u(x,T_1)\,\phi(x,T_1)\,dx-\int_\Omega u(x,T_0)\,\phi(x,T_0)\,dx -\int_{T_1}^{T_1} \langle u(t),\phi'(t)\rangle\,dt.
\end{split}
\]
Thus we obtained
\[
\int_\Omega u(x,T_1)\,\phi(x,T_1)\,dx-\int_\Omega u(x,T_0)\,\phi(x,T_0)\,dx -\int_J \langle u(t),\phi'(t)\rangle\,dt + \int_J \langle\mathcal{A}_t(u),\phi\rangle\,dt = \int_J \langle f(t),\phi(t)\rangle\,dt,
\] 
for every $J=[T_0,T_1]\subset I$ and every $\phi\in L^{p}(J;X_0^{s,p}(\Omega,\Omega'))\cap C^1(J;L^2(\Omega))$.
By recalling the definition of $\mathcal{A}_t$, this shows $u$ is a weak solution of \eqref{Mweaksol}. 
\end{proof}

\begin{prop}[Comparison principle]
\label{prop:linftyglobal}
Let $p\ge 2$, $I = [t_0,t_1]$ and suppose that $g$ satisfies
\begin{align*}
&g\in L^p(I;W^{s,p}(\Omega'))\cap L^{p-1}(I;L^{p-1}_{s\,p}(\mathbb{R}^N)),\quad\partial_t g\in L^{p'}(I;(X_0^{s,p}(\Omega,\Omega'))^*),\\
&\lim_{t\to t_0}\|g(\cdot,t)-g_0\|_{L^2(\Omega)},\qquad \mbox{ for some }g_0\in L^2(\Omega). 
\end{align*}
Given an initial datum $u_0\in L^2(\Omega)$, we consider the unique weak solution $u$ to the initial boundary value problem
\[
 \left\{\begin{array}{rcll}
\partial_tu + (-\Delta_p)^su&=&0,&\mbox{ in }\Omega\times I,\\
u&=&g,& \mbox{ on }(\mathbb{R}^N\setminus\Omega)\times I,\\
u(\cdot,t_0) &=& u_0,&\mbox{ on }\Omega.
\end{array}
\right.
\]
If there exists $M\in\mathbb{R}$ such that
\[
u_0\le M \mbox{ in }\Omega\qquad \mbox{ and }\qquad g\le M \mbox{ in }\mathbb{R}^N\times I,
\]
then we also have 
\[
u(x,t)\le M,\qquad \mbox{ in }\mathbb{R}^N\times I.
\] 
\end{prop}
\begin{proof}
We take $J=[T_0,T_1]\Subset (t_0,t_1)$, by proceeding as in the first part of Lemma \ref{Mlemreg}, we obtain
\[
\begin{split}
\int_{T_0}^{T_1} \iint_{\mathbb{R}^N\times\mathbb{R}^N} &\Big(J_p(u(x,t)-u(y,t))\Big)\,\Big(\phi^\varepsilon(x,t)-\phi^\varepsilon(y,t)\Big)\,d\mu(x,y)\,dt\\
&+\int_{\Omega}\int_{T_0+\frac{\varepsilon}{2}}^{T_1-\frac{\varepsilon}{2}} \partial_t u^\varepsilon(x,t)\,\phi(x,t)\,dt\,dx+\Sigma(\varepsilon)\\
& = \int_{\Omega} \left[u(x,T_0)\,\phi(x,T_0)-u^\varepsilon\left(x,T_0+\frac{\varepsilon}{2}\right)\,\phi\left(x,T_0+\frac{\varepsilon}{2}\right)\right]\,dx \\
&+\int_{\Omega} \left[u^\varepsilon\left(x,T_1-\frac{\varepsilon}{2}\right)\,\phi\left(x,T_1-\frac{\varepsilon}{2}\right)-u(x,T_1)\,\phi(x,T_1)\,dx\right]\,dx,
\end{split}
\]
for every $\phi\in L^{p}(J;X_0^{s,p}(\Omega,\Omega'))\cap C^1(J;L^2(\Omega))$. We still use the notation $\phi^\varepsilon$ and $u^\varepsilon$ for the convolution in the time variable, as defined in \eqref{convolution}. Moreover, we still indicate by $\Sigma(\varepsilon)$ the error term \eqref{sigma}. We now 
take the test function\footnote{By construction, for $x\in\mathbb{R}^N\setminus\Omega$ and $t\in J$ we have
\[
u^\varepsilon(x,t)=\frac{1}{\varepsilon}\,\int_{-\frac{\varepsilon}{2}}^\frac{\varepsilon}{2} \zeta\left(\frac{\sigma}{\varepsilon}\right)\,u(x,t-\sigma)\,d\sigma=\frac{1}{\varepsilon}\,\int_{-\frac{\varepsilon}{2}}^\frac{\varepsilon}{2} \zeta\left(\frac{\sigma}{\varepsilon}\right)\,g(x,t-\sigma)\,d\sigma\le M, 
\]
for $\varepsilon\ll 1$. Thus $(u(\cdot,t)-M)_+\in X^{s,p}_0(\Omega,\Omega')$.}
\[
\phi(x,t)=(u^\varepsilon(x,t)-M)_+.
\]
Observe that this function is only Lipschitz in time, but it is not difficult to see that Lipschitz functions are still feasible test functions (by a simple density argument). This gives
\[
\begin{split}
\int_{\Omega}\int_{T_0+\frac{\varepsilon}{2}}^{T_1-\frac{\varepsilon}{2}} \partial_t u^\varepsilon(x,t)\,\phi(x,t)\,dt\,dx&=\int_{\Omega}\int_{T_0+\frac{\varepsilon}{2}}^{T_1-\frac{\varepsilon}{2}} \partial_t u^\varepsilon(x,t)\,(u^\varepsilon(x,t)-M)_+\,dt\,dx\\
&=\int_{\Omega}\int_{T_0+\frac{\varepsilon}{2}}^{T_1-\frac{\varepsilon}{2}} \partial_t \frac{(u^\varepsilon(x,t)-M)^2_+}{2}\,dt\,dx\\
&=\frac{1}{2}\, \int_{\Omega} \left(u^\varepsilon\left(x,T_1-\frac{\varepsilon}{2}\right)-M\right)^2_+\,dx-\frac{1}{2}\,\int_{B_2} \left(u^\varepsilon\left(x,T_0+\frac{\varepsilon}{2}\right)-M\right)^2_+\,dx.
\end{split}
\]
On the other hand
\[
\begin{split}
\int_{\Omega} &\left[u(x,T_0)\,\phi(x,T_0)-u^\varepsilon\left(x,T_0+\frac{\varepsilon}{2}\right)\,\phi\left(x,T_0+\frac{\varepsilon}{2}\right)\right]\,dx \\
&+\int_{\Omega} \left[u^\varepsilon\left(x,T_1-\frac{\varepsilon}{2}\right)\,\phi\left(x,T_1-\frac{\varepsilon}{2}\right)-u(x,T_1)\,\phi(x,T_1)\,dx\right]\,dx\\
&=\int_{\Omega} \left[u(x,T_0)\,(u^\varepsilon(x,T_0)-M)_+-u^\varepsilon\left(x,T_0+\frac{\varepsilon}{2}\right)\,\left(u^\varepsilon\left(x,T_0+\frac{\varepsilon}{2}\right)-M\right)_+\right]\,dx \\
&+\int_{\Omega} \left[u^\varepsilon\left(x,T_1-\frac{\varepsilon}{2}\right)\,\left(u^\varepsilon\left(x,T_1-\frac{\varepsilon}{2}\right)-M\right)_+-u(x,T_1)\,(u^\varepsilon(x,T_1)-M)_+\right]\,dx.
\end{split}
\]
By taking the limit as $\varepsilon$ goes to $0$, we thus get
\[
\begin{split}
\int_{T_0}^{T_1} \iint_{\mathbb{R}^N\times\mathbb{R}^N} &\Big(J_p(u(x,t)-u(y,t))\Big)\,\Big((u(x,t)-M)_+-(u(y,t)-M)_+\Big)\,d\mu(x,y)\,dt\\
&+\frac{1}{2}\, \int_{\Omega} \left(u\left(x,T_1\right)-M\right)^2_+\,dx-\frac{1}{2}\,\int_{\Omega} \left(u\left(x,T_0\right)-M\right)^2_+\,dx=0.
\end{split}
\]
By using that (see \cite[Lemma A.2]{BP})
\[
J_p(a-b)\,((a-M)_+-(b-M)_+)\ge |(a-M)_+-(b-M)_+|^p,
\]
we thus get
\[
\int_{T_0}^{T_1} \big[(u-M)_+\big]_{W^{s,p}(\mathbb{R}^N)}^p\,dt\le \frac{1}{2}\,\int_{\Omega} \left(u\left(x,T_0\right)-M\right)^2_+\,dx-\frac{1}{2}\, \int_{\Omega} \left(u\left(x,T_1\right)-M\right)^2_+\,dx.
\]
This is valid for every $t_0<T_0<T_1<t_1$.
By using that 
\[
u\in L^p(I;W^{s,p}(\Omega'))\cap L^{p-1}(I;L_{s\,p}^{p-1}(\mathbb{R}^N))\subset L^p(I;W^{s,p}(\mathbb{R}^N)),
\] 
and that $u\in C(I;L^2(\Omega))$, we can pass to the limit as $T_0$ goes to $t_0$ and obtain
\[
\begin{split}
0\le \int_{t_0}^{T_1} \big[(u-M)_+\big]_{W^{s,p}(\mathbb{R}^N)}^p\,dt\le \frac{1}{2}\,\int_{\Omega} \left(u_0(x)-M\right)^2_+\,dx&-\frac{1}{2}\, \int_{\Omega} \left(u\left(x,T_1\right)-M\right)^2_+\,dx\\
&=-\frac{1}{2}\, \int_{\Omega} \left(u\left(x,T_1\right)-M\right)^2_+\,dx.
\end{split}
\]
We used that $u_0\le M$, by construction.
This implies that 
\[
u(x,T_1)\le M,\qquad \mbox{ for a.\,e. }x\in\Omega.
\]
Since $T_1$ is arbitrary, we finally get that 
\[
u(x,t)\le M,\qquad \mbox{ for a.\,e. } x \in\Omega, \mbox{ for } t\in I.
\]
This concludes the proof.
\end{proof}
As a straightforward consequence of the previous result, we get the following
\begin{coro}[Global $L^\infty$ estimate]
\label{coro:linftyglobal}
Under the assumptions of Proposition \ref{prop:linftyglobal}, assume further that
\[
g\in L^\infty(I;L^\infty(\mathbb{R}^N))\qquad \mbox{ and }\qquad u_0\in L^\infty(\Omega).
\]
Then 
\[
\|u\|_{L^\infty(\mathbb{R}^N\times I)}\le \|u_0\|_{L^\infty(\Omega)}+\|g\|_{L^\infty(\mathbb{R}^N\times I)}.
\] 
\end{coro}
\begin{proof}
By using Proposition \eqref{prop:linftyglobal} with
\[
M=\|u_0\|_{L^\infty(\Omega)}+\|g\|_{L^\infty(\mathbb{R}^N\times I)},
\]
we get $u\le M$. To get the lower bound, it is sufficient to observe that $-u$ solves the initial boundary value problem for the same equation, with data $-g\le M$ and $-u_0\le M$. By Proposition \eqref{prop:linftyglobal} again, we get $-u\le M$, as well.
\end{proof}

We also include the following comparison principle with bounded subsolutions.

\begin{prop}[Comparison with subsolutions]
\label{prop:subcomp}
Let $p\ge 2$, $I = [t_0,t_1]$ and suppose that $v\in L^\infty(I;L^\infty(\mathbb{R}^N))$ is a local weak subsolution in $\Omega \times I$ satisfying
\begin{align*}
&v\in L^p(I;W^{s,p}(\Omega'))\cap C(I;L^2(\Omega)),\quad\partial_t v\in L^{p'}(I;(X_0^{s,p}(\Omega,\Omega'))^*),\\
&\lim_{t\to t_0}\|v(\cdot,t)-v_0\|_{L^2(\Omega)},\qquad \mbox{ for some }v_0\in L^2(\Omega). 
\end{align*}
Consider the unique weak solution $u$ to the initial boundary value problem
\[
 \left\{\begin{array}{rcll}
\partial_tu + (-\Delta_p)^su&=&0,&\mbox{ in }\Omega\times I,\\
u&=&v,& \mbox{ on }\mathbb{R}^N\setminus\Omega\times I,\\
u(\cdot,t_0) &=& v_0,&\mbox{ on }\Omega.
\end{array}
\right.
\]
Then 
\[
u(x,t)\ge v(x,t),\qquad \mbox{ in }\mathbb{R}^N\times I.
\] 
\end{prop}
\begin{proof} The proof is almost identical with the proof of Proposition \ref{prop:linftyglobal}. We give some details below. Take $J=[T_0,T_1]\Subset (t_0,t_1)$. Again, as in the first part of Lemma \ref{Mlemreg}, we obtain
\[
\begin{split}
\int_{T_0}^{T_1} \iint_{\mathbb{R}^N\times\mathbb{R}^N} &\Big((J_p(v(x,t)-v(y,t))-J_p(u(x,t)-u(y,t))\Big)\,\Big(\phi^\varepsilon(x,t)-\phi^\varepsilon(y,t)\Big)\,d\mu(x,y)\,dt\\
&+\int_{\Omega}\int_{T_0+\frac{\varepsilon}{2}}^{T_1-\frac{\varepsilon}{2}} \partial_t (v^\varepsilon(x,t)-u^\varepsilon(x,\tau))\,\phi(x,t)\,dt\,dx+\Sigma(\varepsilon)\\
& \leq \int_{\Omega} \left[(v(x,T_0)-u(x,T_0))\,\phi(x,T_0)-\left(v^\varepsilon\left(x,T_0+\frac{\varepsilon}{2}\right)-u^\varepsilon\left(x,T_0+\frac{\varepsilon}{2}\right)\right)\,\phi\left(x,T_0+\frac{\varepsilon}{2}\right)\right]\,dx \\
&+\int_{\Omega} \left[\left(v^\varepsilon\, \left(x,T_1-\frac{\varepsilon}{2}\right)-u^\varepsilon\left(x,T_1-\frac{\varepsilon}{2}\right)\right)\,\phi\left(x,T_1-\frac{\varepsilon}{2}\right)-(v(x,T_1)-u(x,T_1))\,\phi(x,T_1)\right]\,dx,
\end{split}\]
for every non-negative $\phi\in L^{p}(J;X_0^{s,p}(\Omega,\Omega'))\cap C^1(J;L^2(\Omega))$. The quantity $\Sigma(\varepsilon)$ is still defined in \eqref{sigma}, with $v-u$ in place of $u$.
Observe that now we have an inequality, since $v$ is merely a subsolution. Take the test function
\[
\phi(x,t)=(v^\varepsilon(x,t)-u^\varepsilon(x,t))_+.
\]
This gives
\[
\begin{split}
\int_{\Omega}\int_{T_0+\frac{\varepsilon}{2}}^{T_1-\frac{\varepsilon}{2}} \partial_t (v^\varepsilon(x,t)-u^\varepsilon(x,\tau))\,\phi(x,t)\,dt\,dx&=\int_{\Omega}\int_{T_0+\frac{\varepsilon}{2}}^{T_1-\frac{\varepsilon}{2}} \partial_t (v^\varepsilon(x,t)-u^\varepsilon(x,t))\,(v^\varepsilon(x,t)-u^\varepsilon(x,t))_+\,dt\,dx\\
&=\int_{\Omega}\int_{T_0+\frac{\varepsilon}{2}}^{T_1-\frac{\varepsilon}{2}} \partial_t \frac{(v^\varepsilon(x,t)-u^\varepsilon(x,t))^2_+}{2}\,dt\,dx\\
&=\frac{1}{2}\, \int_{\Omega} \left(v^\varepsilon\left(x,T_1-\frac{\varepsilon}{2}\right)-u^\varepsilon\left(x,T_1-\frac{\varepsilon}{2}\right)\right)^2_+\,dx\\
&-\frac{1}{2}\,\int_{B_2} \left(v^\varepsilon\left(x,T_0+\frac{\varepsilon}{2}\right)-u^\varepsilon\left(x,T_0+\frac{\varepsilon}{2}\right)\right)^2_+\,dx.
\end{split}
\]
As before, the terms
\[
\begin{split}
&\int_{\Omega} \left[(v(x,T_0)-u(x,T_0))\,\phi(x,T_0)-\left(v^\varepsilon\left(x,T_0+\frac{\varepsilon}{2}\right)-u^\varepsilon\left(x,T_0+\frac{\varepsilon}{2}\right)\right)\,\phi\left(x,T_0+\frac{\varepsilon}{2}\right)\right]\,dx \\
&+\int_{\Omega} \left[\left(v^\varepsilon\left(x,T_1-\frac{\varepsilon}{2}\right)-u^\varepsilon\left(x,T_1-\frac{\varepsilon}{2}\right)\right)\,\phi\left(x,T_1-\frac{\varepsilon}{2}\right)-(v(x,T_1)-u(x,T_1))\,\phi(x,T_1)\,dx\right]\,dx,
\end{split}
\]
go to zero, as $\varepsilon$ goes to $0$. Therefore, by taking the limit as $\varepsilon$ goes to $0$, we arrive at
\[
\begin{split}
\int_{T_0}^{T_1} \iint_{\mathbb{R}^N\times\mathbb{R}^N} &\Big(J_p(v(x,t)-v(y,t))-J_p(u(x,t)-u(y,t))\Big)\,\Big((v(x,t)-u(x,t))_+-(v(y,t)-u(y,t))_+\Big)\,d\mu(x,y)\,dt\\
&+\frac{1}{2}\, \int_{\Omega} \left(v\left(x,T_1\right)-u\left(x,T_1\right)\right)^2_+\,dx-\frac{1}{2}\,\int_{B_2} \left(v\left(x,T_0\right)-u\left(x,T_0\right)\right)^2_+\,dx\leq  0.
\end{split}
\]
By \cite[Lemma A.3]{bralinschi}, we have
\[
\begin{split}
\Big(J_p(a-b)-J_p(c-d)\Big)\,\big((a-c)_+-(b-d)_+\big)&\ge  C\,|a-b-(c-d)|^{p-1}\,|(a-c)_+-(b-d)_+|\\
&\geq C\,|(a-c)_+-(b-d)_+|^p,
\end{split}
\]
for some $C=C(p)>0$. Then
\[
C\int_{T_0}^{T_1} \big[(v-u)_+\big]_{W^{s,p}(\mathbb{R}^N)}^p\,dt\le \frac{1}{2}\,\int_{\Omega} \left(v\left(x,T_0\right)-u\left(x,T_0\right)\right)^2_+\,dx-\frac{1}{2}\, \int_{\Omega} \left(u\left(x,T_1\right)-u\left(x,T_1\right)\right)^2_+\,dx,
\] for every $t_0<T_0<T_1<t_1$.
We can now let $T_0$ converge to $t_0$ and obtain
\[
\begin{split}
0\le C \int_{t_0}^{T_1} \big[(v-u)_+\big]_{W^{s,p}(\mathbb{R}^N)}^p\,dt\le-\frac{1}{2}\, \int_{\Omega} \left(v\left(x,T_1\right)-u\left(x,T_1\right)\right)^2_+\,dx.
\end{split}
\]
This implies  
\[
u(x,T_1)\ge v(x,T_1),\qquad \mbox{ for a.\,e. }x\in\Omega.
\]
Since $T_1$ is arbitrary, this entails the desired result.
\end{proof}

\section{Some complements to the proof of Lemma \ref{Mlemreg}}
\label{sec:lemma33}

We keep on using the same notation of Lemma \ref{Mlemreg}. For every $0<|h|<h_0/4$ and $0<\varepsilon<\varepsilon_0$, we set
\[
\begin{split}
\mathcal{A}_\varepsilon&:=\int_{T_0}^{T_1}\iint_{\mathbb{R}^N\times\mathbb{R}^N}{\Big(J_p(u_h(x,t)-u_h(y,t))-J_p(u(x,t)-u(y,t))\Big)}\\
&\times\left(\Big(F(\delta_h u^\varepsilon(x,t))\,\tau_\varepsilon(t)\Big)^\varepsilon\,\eta(x)^p-\Big( F(\delta_h u^\varepsilon(y,t))\,\tau_\varepsilon(t)\Big)^\varepsilon\,\eta(y)^p\right)\,d\mu\, dt,
\end{split}
\]
and
\[
\begin{split}
\mathcal{A}&:=\int_{T_0}^{T_1}\iint_{\mathbb{R}^N\times\mathbb{R}^N}{\Big(J_p(u_h(x,t)-u_h(y,t))-J_p(u(x,t)-u(y,t))\Big)}\\
&\times\left(\Big(F(\delta_h u(x,t))\,\tau(t)\Big)\,\eta(x)^p-\Big( F(\delta_h u(y,t))\,\tau(t)\Big)\,\eta(y)^p\right)\,d\mu\, dt.
\end{split}
\]
Then
\[
\begin{split}
|\mathcal{A}_\varepsilon-\mathcal{A}|&= \left|\int_{T_0}^{T_1}\iint_{\mathbb{R}^N\times\mathbb{R}^N}\Big( J_p(u_h(x,t)-u_h(y,t))-J_p(u(x,t)-u(y,t))\Big)\right.\\
&\times\left(\Big(\Big(F(\delta_h u^\varepsilon(x,t))\,\tau_\varepsilon(t)\Big)^\varepsilon-F(\delta_h u(x,t))\,\tau(t)\Big)\,\eta(x)^p\right.\\
&-\left.\left.\left(\Big( F(\delta_h u^\varepsilon(y,t))\,\tau_\varepsilon(t)\Big)^\varepsilon- F(\delta_h u(y,t))\,\tau(t)\right)\,\eta(y)^p\right)\,d\mu\,dt\right|.
\end{split}
\]
We need to show that 
\[
\lim_{\varepsilon\to 0}|\mathcal{A}_\varepsilon-\mathcal{A}|=0.
\]
We start by splitting the integral as follows 
\[
\begin{split}
\int_{T_0}^{T_1}&\iint_{\mathbb{R}^N\times\mathbb{R}^N}\Big( J_p(u_h(x,t)-u_h(y,t))-J_p(u(x,t)-u(y,t))\Big)\\
&\times\left(\Big(\Big(F(\delta_h u^\varepsilon(x,t))\,\tau_\varepsilon(t)\Big)^\varepsilon-F(\delta_h u(x,t))\,\tau(t)\Big)\,\eta(x)^p\right.\\
&-\left.\left(\Big( F(\delta_h u^\varepsilon(y,t))\,\tau_\varepsilon(t)\Big)^\varepsilon- F(\delta_h u(y,t))\,\tau(t)\right)\,\eta(y)^p\right)\,d\mu\,dt\\
&=\int_{T_0}^{T_1}\iint_{B_{2-2\,h}\times B_{2-2\,h}}\Big( J_p(u_h(x,t)-u_h(y,t))-J_p(u(x,t)-u(y,t))\Big)\\
&\times\left(\Big(\Big(F(\delta_h u^\varepsilon(x,t))\,\tau_\varepsilon(t)\Big)^\varepsilon-F(\delta_h u(x,t))\,\tau(t)\Big)\,\eta(x)^p\right.\\
&-\left.\left(\Big( F(\delta_h u^\varepsilon(y,t))\,\tau_\varepsilon(t)\Big)^\varepsilon- F(\delta_h u(y,t))\,\tau(t)\right)\,\eta(y)^p\right)\,d\mu\,dt\\
&+2\,\int_{T_0}^{T_1}\iint_{B_{2-2\,h}\times(\mathbb{R}^N\setminus B_{2-h})}\Big( J_p(u_h(x,t)-u_h(y,t))-J_p(u(x,t)-u(y,t))\Big)\\
&\times\Big(\Big(F(\delta_h u^\varepsilon(x,t))\,\tau_\varepsilon(t)\Big)^\varepsilon-F(\delta_h u(x,t))\,\tau(t)\Big)\,\eta(x)^p\,d\mu\,dt=\Theta_1(\varepsilon)+\Theta_2(\varepsilon).
\end{split}
\]
We now observe that
\[
\begin{split}
\int_{T_0}^{T_1}&\Big[\Big(F(\delta_h u^\varepsilon(\cdot,t))\,\tau_\varepsilon(t)\Big)^\varepsilon\,\eta^p\Big]^p_{W^{s,p}(B_{2-2\,h})}\,dt\\
&\le C\,\|\eta\|_{L^\infty}^{p^2}\,\int_{T_0}^{T_1}\Big[\Big(F(\delta_h u^\varepsilon(\cdot,t))\,\tau_\varepsilon(t)\Big)^\varepsilon\,\Big]^p_{W^{s,p}(B_{2-2\,h})}\,dt\\
&+C\,\|\nabla \eta\|_{L^\infty}^p\,\|\eta\|_{L^\infty}^{p\,(p-1)}\,\int_{T_0}^{T_1}\Big\|\Big(F(\delta_h u^\varepsilon(\cdot,t))\,\tau_\varepsilon(t)\Big)^\varepsilon\,\Big\|^p_{L^p(B_{2-2\,h})}\,dt\le C,
\end{split}
\]
where we used the properties of convolutions, the fact that $F$ is locally Lipschitz and the uniform $L^\infty$ bound \eqref{unifaus}. Thus, up to extracting a subsequence, we can infer weak convergence in
\[
L^{p}([T_0,T_1];W^{s,p}(B_{2-2\,h})),
\] 
of 
\[
\Big(F(\delta_h u^\varepsilon(x,t))\,\tau_\varepsilon(t)\Big)^\varepsilon\,\eta^p,
\]
to the function
\[
F(\delta_h u(x,t))\,\tau(t)\,\eta^p.
\]
By definition, this is the same as saying that the function
\[
\frac{\Big(F(\delta_h u^\varepsilon(x,t))\,\tau_\varepsilon(t)\Big)^\varepsilon\,\eta(x)^p-\Big(F(\delta_h u^\varepsilon(y,t))\,\tau_\varepsilon(t)\Big)^\varepsilon\,\eta(y)^p}{|x-y|^{\frac{N}{p}+s}},
\]
weakly converges in $L^p([T_0,T_1];L^p(B_{2-2\,h}\times B_{2-2\,h}))$. This permits to conclude that
\[
\lim_{\varepsilon\to 0} \Theta_1(\varepsilon)=0,
\]
thanks to the fact that
\[
\frac{J_p(u_h(x,t)-u_h(y,t))-J_p(u(x,t)-u(y,t))}{|x-y|^{\frac{N}{p'}+s\,(p-1)}},
\]
belongs to $L^{p'}([T_0,T_1];L^{p'}(B_{2-2\,h}\times B_{2-2\,h}))$.
\par
For $\Theta_2(\varepsilon)$ we use a similar argument. More precisely, we observe that if we set
\[
\mathfrak{F}(x,t)=\int_{\mathbb{R}^N\setminus B_{2-h}} \frac{J_p(u_h(x,t)-u_h(y,t))}{|x-y|^{N+s\,p}}\,dy,\qquad \mbox{ for a.\,e. }x\in B_{2-2\,h},\, t\in [T_0,T_1],
\]
we have
\[
\begin{split}
|\mathfrak{F}(x,t)|&\le C_h\,\int_{\mathbb{R}^N\setminus B_{2-h}} \frac{|u_h(x,t)|^{p-1}+|u_h(y,t)|^{p-1}}{1+|y|^{N+s\,p}}\,dy\\
&\le C_h\,\left(|u_h(x,t)|^{p-1}+\|\delta_h u(\cdot,t)\|_{L^{p-1}_{s\,p}(\mathbb{R}^N)}^{p-1}\right).
\end{split}
\]
By using the definition of local weak solution, this implies that $\mathfrak{F}\in L^1([T_0,T_1]\times B_{2-2\,h})$. On the other hand, for $t\in [T_0,T_1]$ and $x\in B_{2-2\,h}$ we have
\[
\begin{split}
\Big|\Big(F(\delta_h u^\varepsilon(x,t))\,\tau_\varepsilon(t)\Big)^\varepsilon\,\eta(x)^p\Big|&\le \|\eta\|_{L^\infty}^p\,\int_{-\frac{1}{2}}^\frac{1}{2} \zeta(\sigma)\,|F(\delta_h u^\varepsilon(x,t-\varepsilon\,\sigma))|\,d\sigma\\
&\le C\,\|\eta\|_{L^\infty}^p\,\int_{-\frac{1}{2}}^\frac{1}{2}\zeta(\sigma)\, |\delta_h u^{\varepsilon}(x,t-\varepsilon\,\sigma)|\,d\sigma\\
&\le C\,\|\eta\|_{L^\infty}^p \|\delta_h u^\varepsilon\|_{L^\infty\left(\left[T_0-\frac{\varepsilon}{2},T_1+\frac{\varepsilon}{2}\right]\times B_{2-2\,h}\right)}.
\end{split}
\]
By recalling \eqref{unifaus},
this implies that
\[
\Big(F(\delta_h u^\varepsilon(x,t))\,\tau_\varepsilon(t)\Big)^\varepsilon,
\]
is uniformly bounded in $L^\infty([T_0,T_1]\times B_{2-2\,h})$. The last two facts implies that 
\[
\begin{split}
\lim_{\varepsilon\to 0}\int_{T_0}^{T_1}&\iint_{B_{2-2\,h}\times(\mathbb{R}^N\setminus B_{2-h})}\Big( J_p(u_h(x,t)-u_h(y,t))\Big)\,\Big(F(\delta_h u^\varepsilon(x,t))\,\tau_\varepsilon(t)\Big)^\varepsilon\,\eta(x)^p\,d\mu\,dt\\
&=\int_{T_0}^{T_1}\iint_{B_{2-2\,h}\times(\mathbb{R}^N\setminus B_{2-h})}\Big(J_p(u_h(x,t)-u_h(y,t))\Big)\,F(\delta_h u(x,t))\,\eta(x)^p\,d\mu\,dt,
\end{split}
\]
up to extracting a subsequence. In the exact same way, we can show that 
\[
\begin{split}
\lim_{\varepsilon\to 0}\int_{T_0}^{T_1}&\iint_{B_{2-2\,h}\times(\mathbb{R}^N\setminus B_{2-h})}\Big( J_p(u(x,t)-u(y,t))\Big)\,\Big(F(\delta_h u^\varepsilon(x,t))\,\tau_\varepsilon(t)\Big)^\varepsilon\,\eta(x)^p\,d\mu\,dt\\
&=\int_{T_0}^{T_1}\iint_{B_{2-2\,h}\times(\mathbb{R}^N\setminus B_{2-h})}\Big(J_p(u(x,t)-u(y,t))\Big)\,F(\delta_h u(x,t))\,\eta(x)^p\,d\mu\,dt.
\end{split}
\]
This in turn permits to infer that $\Theta_2(\varepsilon)$ goes to $0$, as well. This concludes the proof of Lemma \ref{Mlemreg}.

\end{document}